\documentclass[11pt]{article}

\usepackage[a4paper,left=1.05in,right=1.05in,top=0.85in,bottom=0.9in]{geometry} 
\usepackage[T1]{fontenc}
\usepackage{lmodern}
\usepackage{amsmath,amssymb}
\usepackage{microtype}
\usepackage{setspace}
\usepackage{cite}
\usepackage{graphicx} 
\usepackage{amsmath,amsthm}
\usepackage{graphicx}
\usepackage{subcaption}
\usepackage{multirow}
\usepackage{graphicx}
\usepackage[colorlinks=true, linkcolor=blue]{hyperref}
\usepackage{cleveref}
\usepackage{float}
\usepackage{amssymb}
\newcounter{figthree}
\renewcommand{\thefigthree}{3.\arabic{figthree}}
\newcounter{figfour}
\renewcommand{\thefigfour}{4.\arabic{figfour}}
\newcounter{figseven}
\renewcommand{\thefigseven}{7.\arabic{figseven}}
\newcounter{fignine}
\renewcommand{\thefignine}{9.\arabic{fignine}}
\newtheorem{theorem}{Theorem}[section]
\newtheorem{lemma}{Lemma}[section]

\newcommand{\Tat}[1]{{\color{blue}#1}}

\title{ An Adaptive and Physics-Preserving Multiscale Method for Two-Phase Flow Simulations in High-Contrast Heterogeneous Porous Media}

\author{Junhao Huang \qquad Eric Chung \qquad Wing Tat Leung}
\date{\today}

\begin{document}
 \vspace{-1.2em}
 \maketitle

\begin{abstract}
In this paper, we propose an adaptive physics-preserving multiscale method for incompressible and immiscible two-phase flow in high-contrast porous media. The method couples a physics-preserving implicit-pressure explicit-saturation scheme (P-IMPES) with the mixed constraint energy minimizing generalized multiscale finite element method. The core algorithmic component is an adaptive update strategy for the saturation-dependent coefficient. Since the effective permeability \(\kappa_n=\lambda_t(S_w^n)K\) depends on the evolving saturation through the total mobility, we introduce an adaptive update algorithm that monitors the variation of the mobility-weighted coefficient and regenerates the multiscale spaces only when a prescribed tolerance is exceeded. A local postprocessing step is further used to recover fine-grid mass conservation.
The analysis is a central part of the paper. We prove local conservation for both phases, the unbiased property of the phase formulation, and bounds preservation under a suitable CFL condition. For the advection-dominated case, we establish velocity and saturation error estimates, which clearly identify the contributions from the adaptive tolerance, the coarse mesh size, the spectral approximation, and the front-layer error. Numerical experiments on different high-contrast permeability fields confirm the physical properties of the method and show that smaller adaptive tolerances improve the saturation approximation while avoiding unnecessary updates of the multiscale spaces.
\end{abstract}

\section{Introduction}

Numerical simulation of incompressible and immiscible two-phase flow in porous media is an important problem in petroleum reservoir engineering, hydrology, and subsurface energy applications. The governing model is derived from mass conservation, Darcy's law, the saturation constraint, and the capillary pressure relation. In realistic geological media, the absolute permeability usually contains multiple spatial scales, high-conductivity channels, inclusions, and very high contrast. These heterogeneous features strongly affect
the pressure and velocity fields, while the saturation equation may develop sharp fronts or discontinuous profiles. Direct fine-grid simulation can resolve these effects, but it is computationally expensive when the pressure-velocity system has to be solved repeatedly during a long-time two-phase flow simulation
\cite{AzizSettari1979,BrooksCorey1964,ChenHuanMa2006,MonteagudoFiroozabadi2007}.

For the temporal discretization of two-phase flow, fully implicit methods are robust but require solving large nonlinear systems at each time step. The implicit-pressure explicit-saturation (IMPES) method is more efficient since it separates the pressure solve from the saturation update
\cite{SheldonZondekCardwell1959,Coats2001CFL,ThomasThurnau1983}.
However, standard IMPES schemes may suffer from restrictive stability conditions, loss of mass conservation for one phase, phase bias, and violation of saturation bounds. In heterogeneous media with capillary pressure, saturation discontinuities across material interfaces introduce further difficulties. Several improved IMPES, discontinuous Galerkin, and capillarity-treatment strategies have been proposed to address these issues
\cite{HoteitFiroozabadi2008A,ErnMozolevskiSchuh2010,
KouSun2014,HouChenSunChen2016}.

Recently, Chen and Sun proposed a physics-preserving IMPES scheme, hereafter referred to as P-IMPES, for incompressible and immiscible two-phase flow in heterogeneous porous media \cite{ChenSun2021PIMPES}. The main idea is to rewrite the phase Darcy velocities in terms of the total velocity and an auxiliary velocity related to the capillary potential gradient. The resulting scheme is locally mass
conservative for both phases, preserves the normal continuity of the total velocity, treats the wetting and non-wetting phases in an unbiased way, and keeps the saturation within physical bounds under a suitable CFL condition. These properties are especially important because the computed velocity field is directly used in the saturation transport equation.

Besides time discretization, the spatial approximation of the pressure--velocity system is another major computational challenge. Classical upscaling methods reduce fine-scale heterogeneity to effective coarse-scale coefficients
\cite{BarkerThibeau1997,Durlofsky1991,Durlofsky1998,ChenDurlofskyGerritsenWen2003}. Multiscale finite element and multiscale finite volume methods provide an alternative reduced-order strategy by constructing coarse-scale basis functions from local fine-scale problems
\cite{HouWu1997,EfendievHouWu2000,EfendievHou2009,JennyLeeTchelepi2003}. Since the saturation equation is very sensitive to the quality of the velocity approximation, local mass conservation is essential for multiscale simulations of porous-media flow. This motivates mixed multiscale finite element methods, multiscale mortar methods, and conservative flux postprocessing techniques
\cite{ChenHou2003,Aarnes2004,AarnesKrogstadLie2006,OdsaterWheelerKvamsdalLarson2017}.

The Generalized Multiscale Finite Element Method (GMsFEM) improves the flexibility of multiscale methods by constructing multiple basis functions through local spectral decompositions
\cite{EfendievGalvisWu2011,EfendievGalvisHou2013,ChungEfendievHou2016,ChungEfendievLee2015}. The GMsFEM framework has been further developed in adaptive enrichment, oversampling, randomized and sparse model reduction, and wave and space-time applications
\cite{ChungEfendievLi2014,EfendievGalvisLiPresho2014,ChungLeung2013,
ChungEfendievLeungLi2016,ChungEfendievGibson2011,ChungEfendievLeung2014,
ChungEfendievHou2023Book}. For high-contrast media, however, important features such as high-conductivity channels may be nonlocal and cannot be captured accurately by standard localized basis functions. Related localized and operator-adapted multiscale ideas for rough-coefficient elliptic problems have also been developed in
\cite{MalqvistPeterseim2014,Owhadi2017}. To overcome the difficulty caused by high contrast and nonlocal features, the Constraint Energy Minimizing Generalized Multiscale Finite Element Method (CEM-GMsFEM) was introduced in \cite{ChungEfendievLeung2018CEM}. The method constructs an auxiliary space from local spectral problems and then computes constraint energy minimizing
basis functions in oversampling regions. By including the eigenfunctions associated with small contrast-dependent eigenvalues, CEM-GMsFEM captures channelized features and achieves contrast-independent convergence with respect to the coarse mesh size. Its mixed formulation
was developed in \cite{ChungEfendievLeung2018MixedCEM}, where auxiliary pressure basis functions are used to construct localized velocity basis functions. The resulting method provides fine-grid mass-conservative velocity approximations and first-order convergence independent of the contrast, provided that enough oversampling layers are used.

Although P-IMPES provides a physics-preserving time discretization and mixed CEM-GMsFEM provides a contrast-robust spatial discretization, their combination for two-phase flow still faces a time-dependent coefficient issue. At time level $t_n$, the coefficient in the pressure--velocity system is
$
    \kappa_n=\lambda_t(S_w^n)K,
$
where $K$ is the absolute permeability and $\lambda_t(S_w^n)$ is the total mobility depending on the current wetting-phase saturation. Hence the coefficient changes as the saturation evolves. If the multiscale spaces are constructed only from the initial coefficient and kept fixed, the resulting basis functions may no longer be well adapted to the current coefficient field, which can deteriorate the velocity approximation and the saturation front.

Compared with the P-IMPES-MsFEM method in \cite{WangChungSun2025}, which used residual-driven enrichment and conservative postprocessing and established phasewise conservation, unbiasedness, conditional bounds preservation, and velocity–saturation error relations, the present work addresses mobility-induced coefficient obsolescence with mixed CEM-GMsFEM, whose oversampled energy-minimizing bases are designed for contrast-robust approximation. Its distinct ingredients are coefficient-variation-triggered global space reconstruction, fine-cell-residual marking for postprocessing instead of source-based marking, and an error estimate separating contributions from the update tolerance, coarse mesh, spectral truncation, and front layer. The P-IMPES properties are therefore inherited and reverified, rather than claimed as new.

In this paper, we propose an adaptive mixed CEM-GMsFEM for the physics-preserving IMPES simulation of incompressible and immiscible two-phase flow in high-contrast porous media. The method couples the P-IMPES formulation with mixed CEM-GMsFEM spaces for the total velocity, auxiliary capillary velocity, and phase pressure. To reduce the error caused by coefficient variation, we propose an error indicator $\eta_n$ to monitor the saturation-induced variation of the mobility-weighted coefficient field. When $\eta_n$ exceeds a prescribed tolerance, the multiscale spaces are regenerated using the current coefficient; otherwise, the existing spaces are reused. A local postprocessing step is also applied to recover fine-grid mass conservation of the total velocity.

The rest of the paper is organized as follows. Section~\ref{sec 2} introduces
the two-phase flow model and the P-IMPES discretization. Section~\ref{sec 3} presents the mixed CEM-GMsFEM construction and the adaptive update strategy. Section~\ref{sec 4} gives the analysis of conservation, unbiased property, bounds preservation, and error estimates. Numerical results are reported in Section~\ref{sec 5}, and conclusions are given in Section~\ref{sec 6}.

\section{Preliminaries}

\label{sec 2}

In this section, we introduce the foundational mathematical model for incompressible and immiscible two-phase flow within a porous medium. We denote the wetting and non-wetting phases by subscripts $w$ and $n$, respectively. The mathematical model is derived from the conservation of mass, Darcy's law, the capillary pressure relationship, and the saturation constraint. Following the model formulation, we introduce the P-IMPES scheme, which is unbiased with regard to the two phases and satisfies the mass conservation for both phases.

\subsection{Mathematical model}

We consider a model in porous media $\Omega \subset \mathbb{R}^2$ given as follows,
\begin{align}
    \phi\frac{\partial S_\alpha}{\partial t} + \nabla \cdot \mathbf{u}_\alpha = F_\alpha, &\quad \text{in } \Omega, \quad \alpha = w, n, \nonumber \tag{2.1}
    \label{2.1}\\
    \mathbf{u}_\alpha = -\frac{k_{r\alpha}}{\mu_\alpha}\mathbf{K}\nabla p_\alpha , &\quad \text{in } \Omega, \quad \alpha = w, n, \tag{2.2} \\
    S_n + S_w = 1, &\quad \text{in } \Omega, \nonumber \\
    p_c(S_w) = p_n - p_w, &\quad \text{in } \Omega. \nonumber
\end{align}

  Here $\phi$ is the porosity of the medium, which is a constant. Let $\mathbf{K}$, $S_\alpha$, $\mathbf{u}_\alpha$ and $p_\alpha$ be the absolute permeability tensor, saturation, Darcy's velocity and pressure corresponding to phase $\alpha$. Total velocity is denoted by $\mathbf{u}_t = \mathbf{u}_w + \mathbf{u}_n$. Define $k_{r\alpha}$, $\mu_\alpha$ and $F_\alpha$ as relative permeability, viscosity and sink/source term of phase $\alpha$. Besides, $F_t = F_w + F_n$. Phase mobility is denoted by $\lambda_\alpha = \frac{k_{r\alpha}}{\mu_\alpha}$. Total mobility is further defined by $\lambda_t = \lambda_w + \lambda_n$. Then we define fractional flow functions as $f_w = \lambda_w/\lambda_t$, $f_n = \lambda_n/\lambda_t$.
  
  We let the boundary be $\Gamma := \partial \Omega$ and it can be partitioned as $\Gamma = \Gamma_D \cup \Gamma_N$, where $\Gamma_D$ and $\Gamma_N$ are Dirichlet and Neumann boundaries to solve $\mathbf{u}_\alpha$. Moreover, to solve $S_\alpha$, we let $\Gamma = \Gamma_{\mathrm{in}} \cup \Gamma_{\mathrm{out}}$, where $\Gamma_{\mathrm{in}} = \{x \in \Gamma : \mathbf{u}_t(x) \cdot \mathbf{n}(x) < 0\}$ is the inflow boundary, $\Gamma_{\mathrm{out}} = \{x \in \Gamma : \mathbf{u}_t(x) \cdot \mathbf{n}(x) \ge 0\}$ is the outflow boundary, and $\mathbf{n}$ is the unit outer normal vector to $\Gamma$. The initial and boundary conditions are $S_\alpha = S_\alpha^0$, for $t = 0$, $p_\alpha = p_\alpha^B$ on $\Gamma_D$ and $\mathbf{u}_\alpha \cdot \mathbf{n} = g_\alpha^N$ on $\Gamma_N$, where $\alpha = w, n$.

\subsection{Physics-preserving IMPES scheme (P-IMPES)}

To compute the reference solutions, we solve Darcy flow on a fine-scale mesh $\mathcal{T}_h$. We use the lowest Raviart-Thomas vector field (RT0) denoted by $V_h$ for fine scale velocity bases. Besides, we define $Q_h$ as the pressure space spanned by piecewise constant functions and each basis corresponds to a particular fine-scale element. And we use $V_h(D)$ and $Q_h(D)$ to represent $V_h$ and $Q_h$ restricted in $D$. The P-IMPES scheme results in the following equation system. \\
For any test functions $\mathbf{v} \in V_h$ and $q \in Q_h$, the scheme satisfies:
{
\begin{align}
& (\kappa_n^{-1}\mathbf{u}_t^{h,n+1}, \mathbf{v}) - (p_w^{h,n+1}, \nabla \cdot \mathbf{v}) = (\kappa_n^{-1} f_n(S_w^{h,n})\xi_c^{h,n+1}, \mathbf{v}) - \int_{\Gamma_D} p_w^B \mathbf{v} \cdot \mathbf{n}, \tag{2.3} \label{2.3}\\
& \sum_\alpha \beta_\alpha(\mathbf{u}_t^{h,n+1}, q; S_w^{h,n}) = (F_t, q), \tag{2.4} \label{2.4}\\
& \left(\phi \frac{S_\alpha^{h,n+1} - S_\alpha^{h,n}}{t_{n+1} - t_n}, q\right) + \beta_\alpha(\mathbf{u}_t^{h,n+1}, q; S_w^{h,n}) = (F_\alpha, q) + \sigma_\alpha \beta_c(\xi_c^{h,n+1}, q; S_w^{h,n}), \tag{2.5} \label{2.5}\\
& (\kappa_n^{-1} \xi_c^{h,n+1}, \mathbf{v}) = (p_c(S_w^{h,n}), \nabla \cdot \mathbf{v}) - \int_{\Gamma_D} (p_n^B - p_w^B)\mathbf{v} \cdot \mathbf{n}, \tag{2.6} \\
& (S_n^{h,n+1} + S_w^{h,n+1}, q) = (1, q), \tag{2.7} \\
& (p_n^{h,n+1} - p_w^{h,n+1}, q) = (p_c(S_w^{h,n}), q). \tag{2.8}
\end{align}
}

Define $t_i$ as the $i$-th time step in a uniform partition of $[0, T]$, where $T$ is the final time $\Delta t = t_{n+1} - t_n$ for each $n$. Set $\sigma_w = 1$ and $\sigma_n = -1$. Let $\mathbf{u}_t^{h, n} \in V_h(\Omega)$ and $\xi_c^{h, n} \in V_h(\Omega)$ be reference velocity solutions to (2.3) and (2.6) at time $t_n$. $p_w^{h, n}$ is the reference pressure of wetting phase at time $t_n$ which is the solution to (2.3). And  $p_n^{h, n}$ is defined to be the pressure of non-wetting phase at time $t_n$ and it can be solved by (2.8).

Define $\kappa_n = \lambda_t(S_w^{h, n})\mathbf{K}$. After we get the solutions for $\mathbf{u}_t^{h,n+1}$ and $\mathbf{\xi}_c^{h,n+1}$, we can update the wetting and non-wetting velocities $\mathbf{u}_w^{h,n+1}$ and $\mathbf{u}_n^{h,n+1}$ on each element as follows:
\begin{align*}
    \mathbf{u}_w^{h,n+1} &= f_w(S_w^{h,n})\mathbf{u}_t^{h,n+1} - f_w(S_w^{h,n})f_n(S_w^{h,n}){\xi}_c^{h,n+1},  \\
    \mathbf{u}_n^{h,n+1} &= f_n(S_w^{h,n})\mathbf{u}_t^{h,n+1} + f_w(S_w^{h,n})f_n(S_w^{h,n}){\xi}_c^{h,n+1}. 
\end{align*}
For $q \in Q_h$ is piecewise constant,
    \begin{align}
        \beta_\alpha(\mathbf{v}, q; S_w^h) &= \sum_{K^h \in \mathcal{T}_h} \int_{\partial K^h} f_\alpha(S_{w,\alpha}^{*,h})\mathbf{v} \cdot \mathbf{n} q, \quad \alpha = w, n, \tag{2.9} \\
        \beta_c(\mathbf{v}, q; S_w^h) &= \sum_{K^h \in \mathcal{T}_h} \int_{\partial K^h} f_n(S_{w,n}^{*,h}) f_w(S_{w,w}^{*,h})\mathbf{v} \cdot \mathbf{n} q, \tag{2.10}
    \end{align}
    where the upwind value $S_{w,\alpha}^{*,h}$ on $e \subset \partial K^h$ in the function $f_\alpha(S_{w,\alpha}^{*,h})$ is defined as follows: 
    \[
        S_{\alpha}^{*,h}|_e = 
        \begin{cases} 
            S_{\alpha}^h|_{K^h}, & \text{if } \{\mathbf{u}_{\alpha}^h \cdot \mathbf{n}\}_e \ge 0, \\ 
            S_{\alpha}^h|_{K^{h,1}}, & \text{if } \{\mathbf{u}_{\alpha}^h \cdot \mathbf{n}\}_e < 0, 
        \end{cases} 
        \qquad
        S_{w,\alpha}^{*,h} = 
        \begin{cases} 
            S_w^{*,h}, & \alpha = w, \\ 
            1 - S_n^{*,h}, & \alpha = n. 
        \end{cases}
    \]
    Here $K^h \cap K^{h,1} = e$ and $\mathbf{n}$ is an outward normal vector to $K^h$. For $e \subset \Gamma_{\mathrm{in}}$, $S_{w,\alpha}^{*,h}|_e = S_w^h|_{K^h}$. We further define $\beta_t = \beta_w + \beta_n$. At the initial time $t=0$, using the initial saturation \(S_w^0\), the numerical flux terms reduce to 
    \begin{align*}
        \beta_{\alpha}(\mathbf{v},q;S_w^{0})&=\sum_{K^h \in \mathcal{T}_h} \int_{K^h} f_{\alpha}(S_w^0)q \nabla \cdot \mathbf{v} \,  dx , \qquad\alpha=w,n, \\
        \beta_c(\mathbf{v},q;S_w^{0}) &=\sum_{K^h \in \mathcal{T}_h} \int_{K^h} f_w(S_w^0)f_n(S_w^0)q \nabla \cdot \mathbf{v} \,  dx
    \end{align*}

\section{Construction of Mixed CEM-GMsFEM }

\label{sec 3}

In this section, we will introduce the framework of the mixed CEM-GMsFEM. The multiscale finite element method consists of two steps. First, we construct a multiscale space $Q_{ms}$ for approximating the pressure. Based on the space $Q_{ms}$, we construct another multiscale space $V_{ms}$ for the velocity.

First, we introduce some notations that will be used later. Given a subset $S \subset \Omega$, we define $V_{h,0}(S) := \{v \in V_h \cap H(\operatorname{div}; S) : v \cdot \mathbf{n}_S = 0 \text{ on } \partial S\}$ and $Q_h(S) := Q_h \cap L^2(S)$. Denote $\kappa_n^H=\lambda_t(S_w^{H,n}) \mathbf{K},\ \ p_c^n=p_c(S_w^{H,n})$

Define the bilinear forms as follows:
\begin{align*}
&a_n(v,w)=\int_\Omega \kappa_n^{-1} v \cdot w \, dx, \quad a_n^H(v, w) = \int_\Omega (\kappa_n^H)^{-1} v \cdot w \, dx,\\ 
& b(v, q) = \int_\Omega q \nabla \cdot v \, dx, \quad \quad   b_D(v,p^B)=\int_{\Gamma_D} p^B v\cdot n \,ds, 
\end{align*}

Note that the pressure basis functions are locally supported on individual coarse elements. In contrast, the associated velocity basis functions are supported on oversampled regions formed by adding several coarse grid layers to the original element (see Figure 1). The localized feature of these velocity bases is the key of our approach.

\begin{figure}[htbp]
    \centering
    \includegraphics[width=0.8\textwidth]{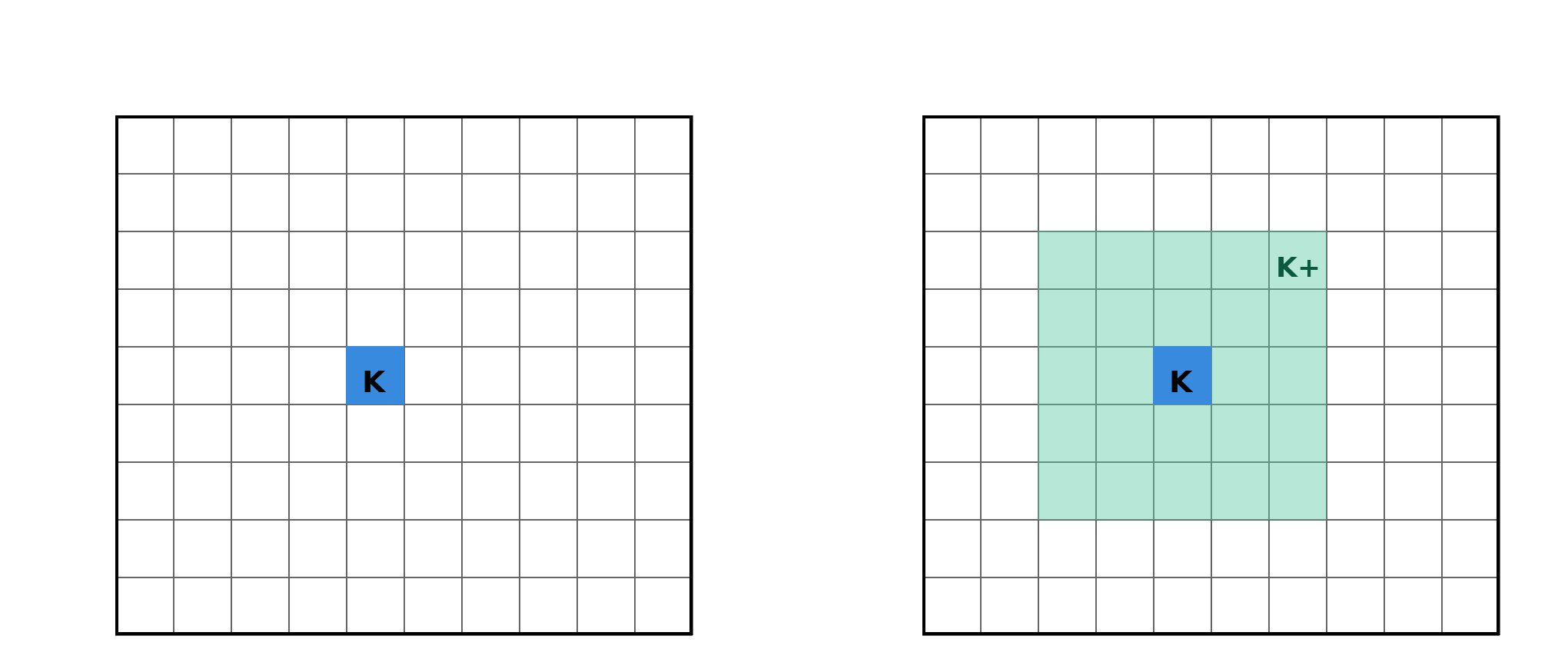}
    \caption{Left: coarse element $K$. Right: oversampled region $K^+$ extending 2 coarse grid layers around $K$.}
    \label{fig:oversampling}
\end{figure}

\subsection{Construction of pressure multiscale space}

In this section, we outline the procedure for generating the pressure multiscale approximation space, denoted as $Q_{\mathrm{ms}}$. For each coarse grid block $K_i \in \mathcal{T}^H$, we establish a localized eigenvalue problem. Specifically, we seek the eigenpairs $(\phi_j^i, p_j^i) \in V_{h,0}(K_i) \times Q_h(K_i)$ and the corresponding eigenvalues $\lambda_j^i \in \mathbb{R}$ satisfying:
\begin{align*}
a_0^H(\phi_j^i, v) - b(v, p_j^i) &= 0 && \forall v \in V_{h,0}(K_i),  \\
b(\phi_j^i, q) &= \lambda_j^i s_i^H(p_j^i, q) && \forall q \in Q_h(K_i), 
\end{align*}
where $j = 1, \dots, L_i$, with the integer $L_i \in \mathbb{N}^+$ representing the total number of localized degrees of freedom determined by the fine and coarse meshes. The localized bilinear operator $s_i: Q_h(K_i) \times Q_h(K_i) \to \mathbb{R}$ is formulated as follows:
\[
s_i^H(p, q) = \int_{K_i} \tilde{\kappa}^H p q \, dx, \quad \text{where } \tilde{\kappa}^H = \kappa_0^H \sum_{j=1}^{N_c} |\nabla \chi_j|^2.
\]

Here, the basis functions $\{\chi_j\}_{j=1}^{N_c}$ constitute a specialized multiscale partition of unity (POU). To be precise, for any internal coarse  node $x_j$ and its associated overlapping neighborhood $\omega_j = \bigcup \{K \in \mathcal{T}^H : x_j \in \partial K\}$, the POU function $\chi_j$ is computed by solving the following local elliptic PDEs within each $K \subset \omega_j$:
\begin{align*}
-\nabla \cdot (\kappa_0^H \nabla \chi_j) &= 0 && \text{in the interior of } K \subset \omega_j, \\
\chi_j &= g_j && \text{on } \partial K \setminus \partial \omega_j \ (\text{for all } K \subset \omega_j), \\
\chi_j &= 0 && \text{on the boundary } \partial \omega_j.
\end{align*}
In the boundary conditions above, $g_j$ is a piecewise linear and continuous function along the edges of the coarse element. 

Based on the normalization, we set the constraint $s_i^H(p_j^i, p_j^i) = 1$. Then ordering the eigenvalues such that $0 \le \lambda_1^i \le \lambda_2^i \le \dots \le \lambda_{L_i}^i$ and truncate the spectrum, retain only the first $J_i$ (where $1 \le J_i \le L_i$) eigenfunctions that correspond to the smallest eigenvalues. The global multiscale space for the pressure is defined as the linear span of these selected basis functions:
\[
Q_{\mathrm{ms}} = \operatorname{span}\{p_j^i : i = 1, \dots, N, \ j = 1, \dots, J_i\}.
\]

\subsection{Construction of velocity multiscale functions}

In this subsection, we detail the derivation of the multiscale velocity space, denoted as $V_{\mathrm{ms}}$. First, we define a global projection operator $\pi: Q_h \to Q_{\mathrm{ms}}$ , which can be explicitly formulated as:
\[
\pi_s q = \sum_{i=1}^N \sum_{j=1}^{J_i} s_i^H(p_j^i, q) p_j^i, \quad \forall q \in Q_h.
\]

First of all, we define a global bilinear form $s: Q_h \times Q_h \to \mathbb{R}$ by the local bilinear form $s_i$ as $s(p, q) = \sum_{i=1}^N s_i^H(p, q)$. From this perspective, the operator $\pi_s$ is an orthogonal projection mapping the fine grid space $Q_h$ onto $Q_{\mathrm{ms}}$ with respect to the $s(\cdot, \cdot)$ inner product.

Global energy-minimizing solutions have decay property and are highly localized. This feature allows us to approximate the global velocity basis within a truncated oversampled region. For any given coarse grid block $K_i \in \mathcal{T}^H$, we expand its boundary by $\ell$ ($\ell \in \mathbb{N}^+$) coarse grid layers, see figure 1 below. The resulting oversampled domains are defined as:
\[
K_{i,0} = K_i, \quad K_{i,\ell} = \bigcup \left\{ K \in \mathcal{T}^H : K \cap \overline{K_{i,\ell-1}} \neq \emptyset \right\}, \quad \text{for } \ell = 1, 2, \dots
\]
For notational simplicity, we let $K_i^+$ denote the target localized region $K_{i,\ell}$ with a sufficiently large layer parameter $\ell$.

For each selected pressure eigenfunction $p_j^i \in Q_{ms}$, the corresponding multiscale velocity basis function $\psi_{j,{ms}}^i \in V_{h,0}(K_i^+)$ along with $q_{j,{ms}}^i \in Q_h(K_i^+)$ are determined by solving the following localized system:
\begin{align*}
a_0^H(\psi_{j,{ms}}^i, v) - b(v, q_{j,{ms}}^i) &= 0 && \forall v \in V_{h,0}(K_i^+),  \\
s(\pi_s q_{j,{ms}}^i, \pi q) + b(\psi_{j,{ms}}^i, q) &= s(p_j^i, q) && \forall q \in Q_h(K_i^+). 
\end{align*}

Then, the global multiscale velocity space is defined by taking the linear span of all these localized velocity basis functions:
\[
V_{{ms}} = \operatorname{span} \left\{ \psi_{j,{ms}}^i : i = 1, \dots, N, \ j = 1, \dots, J_i \right\}.
\]

\subsection{The adaptive method and discretizations}

From the previous discussion, the multiscale spaces for pressure and velocity
have been constructed as $Q_{ ms}=\operatorname{span}\{p_1,\ldots,p_{N_{\rm ms}}\}$, $V_{ ms}=\operatorname{span}\{\psi_1,\ldots,\psi_{N_{\rm ms}}\}$, 
where \(N_{ ms}\) denotes the number of multiscale basis functions in the
final space. In the standard implementation, these multiscale spaces are
constructed at the initial time and then kept fixed throughout the whole
simulation. However, since the effective permeability
$
    \kappa_n=\lambda_t(S_w^n)K
$
depends on the saturation, the coefficient entering the pressure and velocity
equations changes as the saturation updates. Therefore, the multiscale basis
functions constructed from the initial coefficient may become outdated at later
time levels.

To reduce this coefficient-outdated error, we introduce an adaptive update
strategy for the multiscale spaces. The main idea is to monitor the variation
of the coefficient field by an error indicator and update the CEM
multiscale spaces once this indicator exceeds a prescribed tolerance. More
precisely, when the update criterion is satisfied, the multiscale spaces are
reconstructed by replacing the initial coefficient used in the bilinear forms
in Subsections 3.1 and 3.2 with the current coefficient \(\kappa_n^H\).

We first introduce an error indicator to measure the variation of the
coefficient field. Let \(i\) denote the last time level at which the
multiscale space was updated, and let \(n\) be the current time level. We define
\[
    \eta_n
    =
    \left\|
    \left(\kappa_i^H\right)^{-\frac12}
    -
    \left(\kappa_n^H\right)^{-\frac12}
    \right\|_{L^2(\Omega)} 
\]
Then the adaptive multiscale solution can be obtained by the following steps. \\

For given $S_w^{H,n}$ at the time step $t_n$ and most recent updated space $Q_{ms}^i$ and $V_{ms}^i$, we seek the multiscale solutions of the P-IMPES scheme at the time step $t_{n+1}$ as follows: 

Step 1. First, find the capillary velocity $\xi_c^{H,n+1}$ by solving
\begin{align}
    a_n^H( \xi_c^{H,n+1},w)=b(w,p_c(S_w^{H,n}))-b_D(w,p_n^B-p_w^B) && \forall w\in V_{ms}^i \tag{3.1} \label{3.1}
\end{align}

Step 2. Then seek $u_t^{H,n+1} \in V_{ms}^i $ and $p_w^{H,n+1} \in Q_{ms}^i $ satisfying
{ 
\begin{align*}
   a_n^H(u_t^{H,n+1},w) - b(w,p_w^{H,n+1}) &= a_n^H(f_n(S_w^{H,n})\xi_c^{H,n+1} ,w) - b_D(w,p_w^B) \quad \forall w \in V_{ms}^i \tag{3.2} \label{3.2} \\
    \sum_\alpha \beta_\alpha({u}_t^{H,n+1}, q; S_w^{H,n})& = (F_t,q)  \qquad \qquad \forall q \in Q_{ms}^i, \tag{3.3} \label{3.3}
\end{align*}
}Moreover, to guarantee the uniqueness of the solution, we add an additional constraint as $\int_{\Omega} p_w^{H,n+1} = 0$ in step 2.

Step 3. Update the two phase saturation $S_w^{H,n+1}$ and $S_n^{H,n+1}$ in every fine element $K^h \in \mathcal{T}_h$ by
\begin{align*}
  (\phi \frac{S_w^{H,n+1} - S_w^{H,n}}{\Delta t},\mathbf{1}_{K^h})& = -\beta_w({u}_t^{H,n+1}, \mathbf{1}_{K^h}; S_w^{H,n}) +\beta_c(\xi_c^{H,n+1},\mathbf{1}_{K^h}; S_w^{H,n}) +( F_w , \mathbf{1}_{K^h}) \tag{3.4} \label{3.4}  \\
  S_n^{H,n+1}&=1-S_w^{H,n+1}  \tag{3.5}
\end{align*}

Step 4. Then update the effective permeability by $\kappa_{n+1}^H=\lambda_t(S_w^{H,n+1})K $, and we compute the error indicator
\[
    \eta_{n+1}
    =
    \left\|
    \left(\kappa_i^H\right)^{-\frac12}
    -
    \left(\kappa_{n+1}^H\right)^{-\frac12}
    \right\|_{L^2(\Omega)} ,
\]
For a given prescribed tolerance \(\varepsilon_{\rm }>0\).

If $ \eta_{n+1}>\varepsilon$, we regenerate the multiscale spaces in Subsections 3.1 and 3.2 using the coefficient 
\(\kappa_{n+1}^H\). The newly constructed spaces are denoted by
$   Q_{ ms}^{n+1}$, $ V_{ ms}^{n+1}$.
Then we set $i=n+1$ .

Otherwise, if \(\eta_{n+1}\leq \varepsilon\), the current multiscale
spaces are reused, namely $   Q_{ ms}^{i}$, $ V_{ ms}^{i}$, which are kept unchanged for the next time step. \\
\qed

Although our method is formulated by focusing on the wetting phase ($\alpha = w$), in fact, this P-IMPES scheme satisfies an unbiased property. It means we can get exactly the same solution using $\alpha = n$. We will theoretically prove this property in the next section and numerically check it in the numerical experiment.

We note that the velocity obtained from \eqref{3.2}--\eqref{3.3} is conservative
only with respect to the multiscale pressure space. To avoid unnecessary
local solves, we postprocess only those coarse elements containing a
nonzero fine-grid mass residual. Let
$\widehat{\mathbf u}_t^{H,n+1}$ denote the multiscale velocity before
postprocessing. For each fine element $K^h\in\mathcal T_h$, define
\(
R_{K^h}^{n+1}
:=
(F_t,\mathbf{1}_{K^h})
-
\sum_{\alpha}
\beta_\alpha
\left(
\widehat{\mathbf u}_t^{H,n+1},
\mathbf{1}_{K^h};S_w^{H,n}
\right),
\) and introduce the marked set
\(
\mathcal T_{H,\mathrm{pt}}^{n+1}
=
\left\{
K\in\mathcal T_H: 
\exists K^h\in K,R_{K^h}^{n+1}\neq0
\right\}.
\) For each $K\in\mathcal T_{H,\mathrm{pt}}^{n+1}$, we seek
$\mathbf u_{t,p,K}^{H,n+1}\in V_h(K)$ and
$p_{w,p,K}^{H,n+1}\in Q_h(K)/\mathbb R$ satisfying
\[
\begin{aligned}
a_{n,K}^H(\mathbf u_{t,p,K}^{H,n+1},\mathbf v)
-b_K(\mathbf v,p_{w,p,K}^{H,n+1})
&=
a_{n,K}^H
\left(
f_n(S_w^{H,n})\boldsymbol\xi_c^{H,n+1},
\mathbf v
\right),
&&\forall\mathbf v\in V_{h,0}(K),\\
\sum_\alpha
\beta_{\alpha,K}
\left(
\mathbf u_{t,p,K}^{H,n+1},
q;S_w^{H,n}
\right)
&=(F_t,q)_K,
&&\forall q\in Q_h(K),
\end{aligned}
\]
subject to
\[
\mathbf u_{t,p,K}^{H,n+1}\cdot\mathbf n_K
=
\widehat{\mathbf u}_t^{H,n+1}\cdot\mathbf n_K
\quad\text{on }\partial K,
\qquad
(p_{w,p,K}^{H,n+1},\mathbf{1}_K)=0.
\]
The upwind directions are kept fixed during each local solve. On the
unmarked coarse elements, no local problem is solved, and we set
\(
\mathbf u_{t,p}^{H,n+1}|_K
=
\widehat{\mathbf u}_t^{H,n+1}|_K,\text{ for }K\notin\mathcal T_{H,\mathrm{pt}}^{n+1}.
\)
Since the coarse-element constants belong to $Q_{\mathrm{ms}}^i$,
equation \eqref{3.3} implies
\(
\sum_{K^h\subset K}R_{K^h}^{n+1}=0,
\)
which gives the compatibility condition for every marked local problem.
Moreover, the normal trace is unchanged on each coarse-element
boundary, and hence $\mathbf u_{t,p}^{H,n+1}\in V_h$. We introduce the cumulative marked set
\(
\mathcal A_{H,\mathrm{pt}}^{i,n+1}
=
\bigcup_{k=i+1}^{n+1}
\mathcal T_{H,\mathrm{pt}}^{k},
\)
where $i$ denotes the most recent time level at which the CEM space
was updated. We then define the selectively enriched velocity space by
\(
\widetilde V_{\mathrm{ms}}^{i,n+1}
=
V_{\mathrm{ms}}^i
+
\bigoplus_{K\in\mathcal A_{H,\mathrm{pt}}^{i,n+1}}
V_{h,0}(K)
\subset V_h.
\)
The postprocessed pair
\(
(\mathbf u_{t,p}^{H,n+1},p_{w,p}^{H,n+1})
\in
\widetilde V_{\mathrm{ms}}^{i,n+1}
\times Q_h/\mathbb R
\)
is characterized by
\begin{align}
a_n^H
\left(
\mathbf u_{t,p}^{H,n+1},\mathbf w
\right)
-
b
\left(
\mathbf w,p_{w,p}^{H,n+1}
\right)
&=
a_n^H
\left(
f_n(S_w^{H,n})\boldsymbol\xi_c^{H,n+1},
\mathbf w
\right)
-
b_D(\mathbf w,p_w^B),
&&
\forall\mathbf w\in
\widetilde V_{\mathrm{ms}}^{i,n+1},
\tag{3.6}
\label{3.6}
\\
\sum_\alpha
\beta_\alpha
\left(
\mathbf u_{t,p}^{H,n+1},
q;S_w^{H,n}
\right)
&=
(F_t,q),
&&
\forall q\in Q_h.
\tag{3.7}
\label{3.7}
\end{align}
We therefore set
\(
\mathbf u_t^{H,n+1}
=
\mathbf u_{t,p}^{H,n+1}, 
p_w^{H,n+1}
=
p_{w,p}^{H,n+1}
\)
in the subsequent analysis and saturation update.

Next, we will discuss the full discretization of the problem. Let the fine-scale spaces for pressure and velocity be spanned by $Q_h = \text{span}\{q_1, \ldots, q_{N_f}\}$ and $V_h = \text{span}\{\phi_1, \ldots, \phi_{N_e}\}$, and the corresponding multiscale spaces be $Q_{ms}^i = \text{span}\{p_1, \ldots, p_{N_{ms}}\}$ and $V_{ms}^i = \text{span}\{\psi_1, \ldots, \psi_{N_{ms}}\}$. To relate these spaces algebraically, we construct the mapping matrices. Since for every $\psi_i \in V_{ms} \subset V_h$, we can express $\psi_i$ by the linear span of fine scale basis as $\psi_i=\sum\limits_{j=1}^{N_e} {v}_j \phi_j$. Collect all these representative vectors $\mathbf{v}_i=(v_j)_{N_e \times 1}$ and assemble them into a mapping matrix denoted by $V_i=[\mathbf{v}_1, \ldots, \mathbf{v}_{N_{ms}}] \in \mathbf{R}^{N_e \times N_{ms}}$. Similarly, we have the mapping matrix for the multiscale pressure space $P_i=[\mathbf{p}_1, \ldots,\mathbf{p}_{N_{ms}}]$. Recall that \(i\) is the last time level when the most recent multiscale space was updated.

First, we give the following definitions of the matrices:
\begin{align*}
    &A^n=\left(a_n^H(\phi_i,\phi_j)  \right)_{N_{e} \times N_{e}}, \quad  A^{H,n}=V_i^TA^nV_i\\
    &B^n=\left( b(\phi_i,q_j) \right)_{N_{e} \times N_{f}}, \qquad B^{H,n}=V^T_iB^nP_i \\ 
    & B_{\alpha}^n=\left( \beta_\alpha(\phi_i, q_j; S_w^{H,n})\right)_{N_{e} \times N_f}, \quad  B_{\alpha}^{H,n}=V^T_iB_{\alpha}^{n}P_i \\
    &B_c^n=\left(\beta_c(\phi_i, q_j; S_w^{H,n})\right)_{N_{e} \times N_{f}}, \quad  B_{c}^{H,n}=V^T_iB_{c}^{n}P_i \\
    &A_c^n=\left(a_n^H(f_n(S_w^{H,n})\phi_i,\phi_j)  \right)_{N_{e} \times N_{e}}, \quad A_c^{H,n}=V^T_iA_c^nV_i \\
    &D =\left( b_D(\phi_i,1) \right)_{N_{e} \times1} ,\qquad  \qquad D^H=V^T_iD \\
    & \mathbf{F}_{\alpha}=(F_{\alpha},q_i)_{N_{f} \times1} , \quad \mathbf{F}_t=\mathbf{F}_w +\mathbf{F}_n , \quad   
    \Pi_i\mathbf{F}_{\alpha}=P_i^T \mathbf{F}_{\alpha}
\end{align*} 

\textbf{(Matrix\Tat{–}Vector Formulation)} For given $\mathbf{S}_w^{H,n}\in \mathbb{R}^{N_f}$ at the time step $t_n$. We seek the multiscale solutions of the P-IMPES scheme at the time step $t_{n+1}$ as follows: \\
Step 1. First, discretize the capillary velocity $\xi_c^{H,n+1}=\sum\limits_{i=1}^{N_{ms}} x_i^{n+1} \psi_i$ and obtain $x_c^{n+1}=(x_i^{n+1})_{N_{ms}\times 1}$ by solving 
\begin{align*}
    A^{H,n}x_c^{n+1}=V^T_iB^np_c(\mathbf{S}_w^{H,n})-(p_n^B-p_w^B)D^H \tag{3.8}
\end{align*}
Step 2. We further seek $\mathbf{u}_t^{H,n+1} \in \mathbb{R}^{N_{ms}}$ , $\mathbf{p}_w^{H,n+1}\in \mathbb{R}^{N_{ms}}$ by
\begin{equation*}
\begin{bmatrix}
{A}^{H,n} & -B^{H,n} \\
-(B_{w}^{H,n}+B_n^{H,n})^T & 0
\end{bmatrix}
\begin{bmatrix}
\mathbf{u}_t^{H,n+1} \\
\mathbf{p}_w^{H,n+1}
\end{bmatrix}
=
\begin{bmatrix}
{A}_c^{H,n} x_c^{n+1} - p_w^B{D^H} \\
-\Pi_i\mathbf{F}_t
\end{bmatrix}.
\tag{3.9}
\end{equation*}
Then apply the local postprocessing procedure in every coarse element
\(K\in \mathcal A_{H,\mathrm{pt}}^{i,n+1}\) and assemble the postprocessed fine-grid velocity
\(\mathbf u_{t,p}^{H,n+1}\).\\
Step 3. Finally, update the two phase saturation $\mathbf{S}_w^{H,n+1}$ and $\mathbf{S}_n^{H,n+1}$ by
\begin{align*}
  \phi \frac{\mathbf{S}_w^{H,n+1} - \mathbf{S}_w^{H,n}}{\Delta t}& = -(B_w^{n})^T\mathbf u_{t,p}^{H,n+1} +(B_c^{n})^TV_ix_c^{n+1}+\mathbf{F}_w   \\
  \mathbf{S}_n^{H,n+1}&=1-\mathbf{S}_w^{H,n+1}  
  \tag{3.10}
\end{align*}
Step 4. Adaptive update of the multiscale spaces. After obtaining
\(\mathbf S_w^{H,n+1}\), set
\(
    \kappa_{n+1}^H=\lambda_t(\mathbf S_w^{H,n+1})K
\)
and compute
\(
    \eta_{n+1}
    =
    \|
    \left(\kappa_i^H\right)^{-\frac12}
    -
    \left(\kappa_{n+1}^H\right)^{-\frac12}
    \|_2 .
\)

If \(\eta_{n+1}>\varepsilon\), then regenerate the CEM multiscale
spaces with the current coefficient \(\kappa_{n+1}^H\), and update
\[
    Q_{ms}^{i}=Q_{ ms}^{n+1},\qquad
    V_{ms}^{i}=V_{ ms}^{n+1},\qquad
    P_i=P_{n+1},\qquad
    V_i=V_{n+1},\qquad
    i=n+1.
\]
Otherwise, the current spaces and mapping matrices are kept unchanged.

We remark that Step 2 is solved together with the condition $\int_{\Omega} p = 0$, so we will add Lagrange multiplier in (3.9) to satisfy the constraint.\\

\section{Analysis}
\label{sec 4}

In this section, we present the theoretical analysis of the proposed method in two main parts. First, following the analytical framework in \cite{ChenSun2021PIMPES}, we establish several key properties of the scheme. Subsequently, we derive an essential relationship between the saturation error and the velocity error.

\subsection{Local conservation of mass and unbiased property}
\label{4.1}
In this subsection, we prove two physical properties of the proposed multiscale
P-IMPES scheme: local mass conservation and the unbiased property. Although only
the wetting-phase saturation is updated explicitly, the non-wetting-phase
saturation is determined by the saturation constraint and also satisfies a
conservative discrete equation. This follows from the fine-grid conservation of
the total velocity and the wetting-phase conservation equation. We also show that
the same discrete solution is obtained whether the saturation equation is written
for the wetting phase or for the non-wetting phase.

Next We  will prove the result in two parts.

\textbf{(i) Local conservation of mass.}
The scheme updates only the wetting phase via
\begin{align*}
\left(\phi\frac{S_w^{H,n+1}-S_w^{H,n}}{t_{n+1}-t_n},q\right)
+\beta_w(\mathbf{u}_t^{H,n+1},q;S_w^{H,n})
&=(F_w,q)+\beta_c(\xi_c^{H,n+1},q;S_w^{H,n}), \tag{4.1}\\
S_n^{H,n+1}&=1-S_w^{H,n+1}
\end{align*}
The two phase mass conservation means that both phases satisfy:
\begin{align}
\left(\phi\frac{S_w^{H,n+1}-S_w^{H,n}}{t_{n+1}-t_n},q\right)
+\beta_w(\mathbf{u}_t^{H,n+1},q;S_w^{H,n})
&=(F_w,q)+\beta_c(\xi_c^{H,n+1},q;S_w^{H,n}), \tag{4.2}
\label{4.2}\\
\left(\phi\frac{S_n^{H,n+1}-S_n^{H,n}}{t_{n+1}-t_n},q\right)
+\beta_n(\mathbf{u}_t^{H,n+1},q;S_w^{H,n})
&=(F_n,q)-\beta_c(\xi_c^{H,n+1},q;S_w^{H,n}). \tag{4.3}
\label{4.3}
\end{align}
Since (4.1) and (4.2) coincide, it suffices to show that $S_n^{H,n+1}$ from $S_n^{H,n+1}=1-S_w^{H,n+1}$ agrees with the solution of (4.3). On the one hand, adding (4.2) and (4.3), using \eqref{3.7} we have
\[
\sum_\alpha\left(\phi\frac{S_\alpha^{H,n+1}-S_\alpha^{H,n}}{t_{n+1}-t_n},q\right)=0.
\]
Choosing $q=\mathbf{1}_{K^h}$ gives $\sum_\alpha(S_\alpha^{H,n+1}-S_\alpha^{H,n})=0$ on each $K^h$. Combined with the initial condition $S_w^{H,0}+S_n^{H,0}=1$, this derives $S_w^{H,n+1}+S_n^{H,n+1}=1$ for all $n$. On the other hand, subtracting (4.1) by (3.7) and replacing $S_w^{H}=1-S_n^{H} $ in the first term left, we can obtain (4.3) directly. 

\textbf{(ii) Unbiased property for velocity.}
From part (i), $S_w^{H,n+1}$ and $S_n^{H,n+1}$ are unbiased. The total velocity $\mathbf{u}_t^{H,n+1}$ and the wetting pressure $p_w^{H,n+1}$ are uniquely determined by \eqref{3.2} and \eqref{3.3}; $\xi_c^{H,n+1}$ is uniquely solved by \eqref{3.1}. The phase velocities $\mathbf{u}_\alpha^{H,n+1}$ are unbiased via
\begin{align*}
\mathbf{u}_w^{H,n+1}&=f_w(S_w^{H,n})\mathbf{u}_t^{H,n+1}-f_w(S_w^{H,n})f_n(S_w^{H,n})\xi_c^{H,n+1},\\
\mathbf{u}_n^{H,n+1}&=f_n(S_w^{H,n})\mathbf{u}_t^{H,n+1}+f_w(S_w^{H,n})f_n(S_w^{H,n})\xi_c^{H,n+1},
\end{align*}

\subsection{Bounds-preserving property for both phases}

This part is motivated by the arguments in  \cite{ChenSun2021PIMPES}. We first give two auxiliary lemmas that will be used to establish the bounds-preserving property.

\begin{lemma}
    There exists a positive constant $l_\alpha$ such that
\[
    f_\alpha(S_\alpha) \leq l_\alpha S_\alpha, 
    \qquad \alpha = w,n .
\]
\end{lemma}
This lemma shows that each fractional flow function can be bounded above by a linear function of the corresponding saturation.

\begin{lemma}
    Suppose that $F_\alpha$ defined in \eqref{2.1} is a sink term, namely $F_\alpha \leq 0$. Then there exist two positive constants $l_1$ and $l_2$ such that
\[
    l_1 S_\alpha \leq |F_\alpha| \leq l_2 S_\alpha .
\]
\end{lemma}
The above two lemmas have been proved in \cite{ChenSun2021PIMPES}. Based on these results, we now state the theorem that guarantees the bounds-preserving property.
\begin{theorem}
    Assume that $S_\alpha^{H,n}\in(0,1)$ and that there exists a tolerance saturation
$S_{t\alpha}>0$ such that
\[
    S_\alpha^{H,n}\ge S_{t\alpha}>0, 
    \qquad \alpha=w,n .
\]
Let $\Delta t=t_{n+1}-t_n$. For any $K^h\in\mathcal T_h$, if $\frac{\Delta t}{h}$ is sufficiently small, then
\[
    S_\alpha^{H,n+1}\in(0,1),
    \qquad \alpha=w,n .
\]
\end{theorem}

\begin{proof}
The proof follows the argument of Lemma 4.3 in \cite{ChenSun2021PIMPES}, with 
$\mathbf u_t^{h,n+1}$ replaced by the multiscale velocity 
$\mathbf u_t^{H,n+1}$. We only prove the positivity of the wetting-phase
saturation; the non-wetting phase can be treated in the same way.

For any $K^h\in\mathcal T_h$, taking $q=\mathbf 1_{K^h}$ in \eqref{3.4}
gives
\begin{align*}
\phi \left(S_{w}^{H,n+1}-S_{w}^{H,n}\right)
= \Delta t F_w
+ \frac{\Delta t}{h}
\sum_{F\subset\partial K^h}
f_w(S_{w,w}^{*,H,n})
\left(
f_n(S_{w,n}^{*,H,n}) \xi_c^{H,n+1}\cdot \mathbf n
-\mathbf u_t^{H,n+1}\cdot \mathbf n
\right)\big|_F \tag{4.4}
\end{align*}
Define $\mathbf v_w^{H,n+1}\in V_{ms}$ by
\[
\mathbf v_w^{H,n+1}\cdot\mathbf n
=
\mathbf u_t^{H,n+1}\cdot\mathbf n
-
f_n(S_{w,n}^{*,H,n}) \xi_c^{H,n+1}\cdot \mathbf n
\qquad \text{on }F\in\mathcal E_f .
\]
Then
\begin{equation}
S_w^{H,n+1}
=
S_w^{H,n}
+
\frac{\Delta t}{\phi}F_w
-
\frac{\Delta t}{\phi h}
\sum_{F\subset\partial K^h}
f_w(S_{w,w}^{*,H,n})
\mathbf v_w^{H,n+1}\cdot\mathbf n\big|_F .
\label{4.5}
\tag{4.5}
\end{equation}

Let
\[
\partial K_{w,n+1}^{h,+}
=
\left\{
F\subset\partial K^h:
\mathbf v_w^{H,n+1}\cdot\mathbf n\big|_F\ge 0
\right\}.
\]
The incoming part gives a nonnegative contribution in \eqref{4.5}, and hence can be omitted when deriving a lower bound. Therefore,
\[
S_w^{H,n+1}
\ge
S_w^{H,n}
+
\frac{\Delta t}{\phi}F_w
-
\frac{\Delta t}{\phi h}
\sum_{F\subset\partial K_{w,n+1}^{h,+}}
f_w(S_{w,w}^{*,H,n})
\mathbf v_w^{H,n+1}\cdot\mathbf n\big|_F .
\]
For the sink term, Lemma 4.2 implies $F_w\ge -l_2 S_w^{H,n}$, while Lemma 4.1 gives
\[
f_w(S_{w,w}^{*,H,n})\le l_w S_{w,w}^{*,H,n}.
\]
Thus the outflow contribution is bounded by a constant multiple of
$S_w^{H,n}$, except possibly on the fine edges where the sign of
$\mathbf v_w^{H,n+1}\cdot\mathbf n$ changes from time level $n$ to
$n+1$. On those sign-changing edges, the normal flux can be made arbitrarily
small when $\Delta t$ is sufficiently small. Hence there exist positive
constants $\varepsilon_1$ and $\varepsilon_2$ such that, if
\[
\frac{\Delta t}{h}\le \varepsilon_1,
\qquad
\left|\mathbf v_w^{H,n+1}\cdot\mathbf n\right|_F\le \varepsilon_2
\quad
\text{on the sign-changing fine edges},
\]
then, for some $\eta\in(0,1)$,
\[
S_w^{H,n+1}\ge \eta S_{tw}>0.
\]
Since $S_w^{H,n}\ge S_{tw}>0$, the above two conditions are satisfied once
$\Delta t/h$ is sufficiently small.

Applying the same argument to the non-wetting phase yields
\[
S_n^{H,n+1}\ge \eta S_{tn}>0.
\]
Finally, using the discrete phase conservation relation
\[
S_w^{H,n+1}+S_n^{H,n+1}=1,
\]
we obtain
\[
0<S_\alpha^{H,n+1}<1,
\qquad \alpha=w,n.
\]
This completes the proof.
\end{proof}

\noindent\textbf{Remark.}
By the proof of Theorem 4.1, the stability restriction on the time step can be written in a phase-flux form as
\[
    \max_{\alpha=w,n} E_\alpha^{n+1}\frac{\Delta t}{h}+C\Delta t<\phi,
\]
where $0<\phi<1$, $C$ depends on the sink term, and
\[
    E_\alpha^{n+1}
    =
    \max_{K^h\in\mathcal T_h}
    \sum_{F\subset\partial K^h}
    (f_\alpha(S_\alpha^{*,H,n})
    \mathbf v_\alpha^{H,n+1}\cdot\mathbf n
    )^+ ,
    \qquad \alpha=w,n .
\]
Here $a^+=\max\{a,0\}$, and $\mathbf v_\alpha^{H,n+1}$ denotes the effective phase velocity appearing in the saturation equation. In the numerical experiments, we use the following CFL condition, which includes
both the convective flux induced by the total velocity and the capillary
diffusion flux induced by $\boldsymbol{\xi}_c^{H,n+1}$:
\begin{align*}
\theta_{cfl}
&=
\max_{n}
\frac{\Delta t}{\phi h}
\Bigg(
\|\mathbf u_t^{H,n+1}\|_{L^\infty(\Omega)}
\|f_w'\|_{L^\infty(0,1)}
+
\|{\xi}_c^{H,n+1}\|_{L^\infty(\Omega)}
\|(f_w f_n)'\|_{L^\infty(0,1)}
\Bigg)
< 1 .
\end{align*}
The time step size $\Delta t$ is chosen such that
$\theta_{cfl}<1$.

\subsection{Error estimate}
In this subsection, we establish the error estimates. To obtain the numerical errors induced by the multiscale velocity approximation, we restrict our analysis to the zero-capillary form, namely the advection-dominated Buckley-Leverett flow. Incorporating capillary pressure leads to a degenerate parabolic--hyperbolic problem and requires additional estimates for the capillary
velocity, which are left for future work. Under this condition, we set zero Dirichlet and zero Neumann boundary condition, then \eqref{2.3}, \eqref{2.4} can be reduced to:
\begin{align}
    a_n(\mathbf{u}_t^{h, n+1},v) - b(v,p_w^{h, n+1} ) &= 0, && \forall v\in V_h \tag{4.6}\label{4.6} \\
    b(\mathbf{u}_t^{h, n+1}, q) &= (F_t, q)  && \forall q\in Q_h   \tag{4.7}\label{4.7}
\end{align}

Similarly for \eqref{3.6}, \eqref{3.7}, the postprocessed multiscale solution satisfies:
\begin{align}
   a_n^H(\mathbf{u}_t^{H, n+1},w) - b(w,p_w^{H, n+1} ) &= 0 && \forall w\in \widetilde{V}_{ms}^{i,n+1}\tag{4.8}\label{4.8} \\
     b(\mathbf{u}_t^{H, n+1}, q) &= (F_t, q)  && \forall q\in Q_{h}\tag{4.9}\label{4.9}
\end{align}

Besides, define a local version of the norms $\|\cdot\|_{a_i(D)}$, $\|\cdot\|_{V_i(D)}$ and $\|\cdot\|_{s_i(D)} $ at corresponding time $t=t_i$.
\begin{align*}
    \|v\|_{V_i(D)}^2 &= \int_D \tilde{\kappa}^{-1}_i|\nabla \cdot v|^2 + \int_D \kappa^{-1}_i|v|^2,  \quad   \|q\|_{s_i(D)}^2 = \int_D \tilde{\kappa}_i|q|^2\\
\|v\|_{a_i(D)}^2 &= \int_D \kappa_i^{-1}|v|^2 ,   \quad \|v\|_{a_i^H(D)}^2 = \int_D (\kappa_i^H)^{-1}|v|^2.
\end{align*}

Next, we will introduce some assumptions used in the error analysis.

\paragraph{Same-direction assumption.}
For every fine edge \(F\subset \partial K_h\),
\[
\left(u_t^{h,n+1}\cdot \mathbf{n}\right)
\left(u_t^{H,n+1}\cdot \mathbf{n}\right)\big|_F
\ge 0.
\tag{4.10}
\]
Thus the fine-scale and multiscale velocities determine the same upwind
direction on each fine edge.

\paragraph{Front-layer decomposition.}
For each time level \(t_n\), we decompose the fine mesh into a smooth region and
a front layer,
\[
\mathcal T_h=\mathcal G_h^n\cup\mathcal B_h^n,
\qquad
\mathcal G_h^n\cap\mathcal B_h^n=\emptyset.
\]
The set \(\mathcal B_h^n\) contains the cells close to the saturation fronts,
shocks, and possible mismatch strips between the fine-scale and multiscale
fronts. We assume that
\[
|\Omega_{\mathcal B}^n|
\le \delta_h , 
\qquad
0\le n\le N.
\]
On the smooth region, the two saturation profiles are locally smooth in the
following discrete sense: for neighboring fine cells \(K^h,K^{h,1}\in\mathcal G_h^n\)
which share an edge \(F\), we define
\[
\begin{aligned}
\zeta_h
=
\max_{0\le n\le N}
\max_{\substack{
K^h,K^{h,1}\in \mathcal G_h^n\\
F=K^h\cap K^{h,1}
}}
\Big[
\omega_{n,F}^{\frac12}
|S_{w,K^h}^{h,n}-S_{w,K^{h,1}}^{h,n}|
+
(\omega_{n,F}^H)^{\frac12}
|S_{w,K^h}^{H,n}-S_{w,K^{h,1}}^{H,n}|
\Big],
\end{aligned}
\tag{4.11}
\label{4.11}
\]
where
\(
\omega_{n,F}=\max(\kappa_n|_{K^h},\kappa_n|_{K^{h,1}}),
\quad
\omega_{n,F}^H=\max(\kappa_n^H|_{K^h},\kappa_n^H|_{K^{h,1}}).
\)
On the smooth region, we assume that
\[
\frac{\zeta_h}{h}=O(1)
\qquad \text{as } h\to0.
\] Moreover, the front layer is enlarged by a fixed number of fine-cell layers so
that the following stencil closure property holds: if
\(K^h\in\mathcal G_h^{n+1}\), then every fine cell involved in the upwind
update of \(K^h\) at time \(t_n\) belongs to \(\mathcal G_h^n\).

For a velocity \(u\), define the upwind update operator $\mathcal L_u $ by
\begin{align*}
(\mathcal L_u S)_{K^h}
=
S_{K^h}
-
\frac{\Delta t}{\phi |K^h|}
\sum_{F\subset\partial K^h}
|F|\left(u_{K^{h,F}}^{+} f_w(S_{K^{h}})+
u_{K^{h,F}}^{-} f_w(S_{K^{h,i}})
\right)
+\frac{\Delta t}{\phi |K^h|}(F_w,1_{K^h}),
\tag{4.12}
\label{4.12}
\end{align*}
Here $u^{+}=\max\{u,0\} $ denotes the outgoing flux, $u^{-}=\min\{u,0\}$  denotes the incoming flux.

\begin{lemma}[Front-aware transport estimates]
Under the bounds-preserving property \(0\le S_w^{h,n},S_w^{H,n}\le 1\),
the same-direction condition and the front-layer assumptions. Then the following estimates hold.

\label{lemma 4.3}
\emph{(i) Front-layer contribution.}
\[
\|e_S^n\|_{L^2(\Omega_{\mathcal B}^n)}
\le
|\Omega_{\mathcal B}^n|^{1/2}
\le
\delta_h ^{1/2}.
\tag{4.13}
\]

\emph{(ii) Saturation stability on the smooth region.}

Let \(S_1,S_2\in Q_h\), \(0\le S_1,S_2\le 1\), and suppose that the smooth-region
condition (\ref{4.11}) holds for both \(S_1\) and \(S_2\). Then
\[
\|\mathcal L_u S_1-\mathcal L_u S_2\|_{2,\mathcal G_h^{n+1}}
\le
(1+C_{\rm st}\Delta t)
\|S_1-S_2\|_{2,\mathcal G_h^n},
\tag{4.14}
\]
where
$C_{\rm st}
=
C\phi^{-1}
\left(
L_f\|F_t\|_{L^\infty(\Omega)}
+
\frac{L_{f}\zeta_h}{h\sqrt{\kappa_{min}}}\|u\|_{L^\infty(\Omega)}
\right)$
and \(C\) depends only on the shape regularity of the fine mesh.

\emph{(iii) Velocity estimate on the smooth region.}
\[
E_G^{n+1}
\le
(1+C_{\rm st}\Delta t)E_G^n
+
C_{\rm v}\Delta t
\|e_{u_t}^{n+1}\|_{a^H_n(\Omega)}.
\tag{4.15}
\label{4.15}
\]
where $C_{\rm v}=C\phi^{-1}L_f\zeta_hh^{-1}$, and denote $E_G^n=\|e_S^n\|_{2,\mathcal G_h^n}.$

\end{lemma}
\begin{proof}
We prove the estimates in four parts.

Step 1, since the scheme is bounds-preserving, we have
\(0\le S_w^{h,n},S_w^{H,n}\le 1\). Hence
\[
 |e_S^n|=|S_w^{h,n}-S_w^{H,n}|\le 1,
\]
and therefore
\[
\|e_S^n\|_{L^2(\Omega_{\mathcal B}^n)}^2
\le |\Omega_{\mathcal B}^n|
\le  \delta_h .
\]
This gives
\[
\|e_S^n\|_{L^2(\Omega_{\mathcal B}^n)}
\le  \delta_h ^{1/2},
\]
which proves (i).

Step 2: Saturation stability on the smooth region. Set $d_{K^h}=S_1|_{K^h}-S_2|_{K^h}$, $D_{K^h}=(\mathcal L_u S_1-\mathcal L_u S_2)_{K^h}.$
 For \(K^h\in\mathcal G_h^{n+1}\), the closure property implies that every
cell entering the update of \(K^h\) belongs to \(\mathcal G_h^n\).
From \eqref{4.12}, we have
\begin{align*}
    D_{K^h}=d_{K^h}-
\frac{\Delta t}{\phi |K^h|}
\sum_{F=K^h\cap K^{h,i}}^{N_{neigh}}
|F|
[
u_{K^{h,F}}^{+}&(f_w(S_1|_{K^{h}})-f_w(S_2|_{K^{h}})) 
\\&+u_{K^{h,F}}^{-}(f_w(S_1|_{K^{h,i}})-f_w(S_2|_{K^{h,i}}))
].
\end{align*}

Denote $\theta_{K^h}=\frac {f_w(S_1|_{K^{h}})-f_w(S_2|_{K^{h}})}{d_{K^h}}$, $\theta_{i}=\frac{f_w(S_1|_{K^{h,i}})-f_w(S_2|_{K^{h,i}})}{d_{K^{h,i}}}$ and define
\[
r_{K^h}
=
1
+
\frac{\Delta t}{\phi |K^h|}
(
\sum_{\substack{F\subset\partial K^h\\ u_{K^{h,F}}<0}}
|F|(-u_{K^{h,F}})\theta_{i}
-
\sum_{\substack{F\subset\partial K^h\\ u_{K^{h,F}}>0}}
|F|u_{K^{h,F}}\theta_{K^h}
).
\]
Adding and subtracting $\sum\limits_{\substack{F\subset\partial K^h\\ u_{K^{h,F}}<0}}|F|u_{K^{h,F}}\theta_{K^h}$, we obtain
\[
r_{K^h}
=1-\frac{\Delta t}{\phi |K^h|}\theta_{K^h}\sum_{F\subset\partial K^h}|F|u_{K_h,F}+
\frac{\Delta t}{\phi |K_h|}
\sum_{\substack{F\subset\partial K^h\\ u_{K^{h,F}}<0}}
|F|(-u_{K^{h,F}})(\theta_{i}-\theta_{K^h}).
\]
Since \(u\in RT_0\),
\[
\sum_{F\subset\partial K^h}|F|u_{K^{h,F}}
=\int_{\partial K^h}u\cdot n
=
\int_{K^h}\nabla\cdot u
=\int_{K^h} F_t\]
And by the Lipschitz condition in $f_w$, we obtain
\[
\left|
\frac{1}{|K^h|}
\theta_{K^h}
\sum_{F\subset\partial K^h}|F|u_{K^{h,F}}
\right|
\le
L_f\|F_t\|_{L^\infty(\Omega)}.
\]
Moreover,
\[
|\theta_{i}-\theta_{K^h}|
\le
C L_{f}
(
\left|S_1|_{K^{h,i}}-S_1|_{K^h}\right|
+
\left|S_2|_{K^{h,i}}-S_2|_{K^h}
\right|).
\]
Using the smooth-region assumption (4.11), we obtain
\[
|\theta_{i}-\theta_{K^h}|
\le
\frac{C L_{f} \zeta_h}{\sqrt{\kappa_{min}}} .
\]
Therefore,
\[
\frac{1}{|K^h|}
\sum_{\substack{F\subset\partial K^h\\ u_{K^{h,F}}<0}}
|F|(-u_{K^{h,F}})|\theta_{i}-\theta_{K^h}|
\le C
\frac{ L_{f}\zeta_h}{h\sqrt{\kappa_{min}}}\|u\|_{L^\infty(\Omega)}.
\]
Combining above discussions gives
\[
r_{K^h}
\le
1+C_{\rm st}\Delta t.
\tag{4.16}
\]
Here
$
C_{\rm st}
=
C\phi^{-1}
\left(
L_f\|F_t\|_{L^\infty(\Omega)}
+
\frac{L_{f}\zeta_h}{h\sqrt{\kappa_{min}}}\|u\|_{L^\infty(\Omega)}
\right)$, \(C\) only depends on the shape regularity of the mesh.

Equivalently, using the expression of \(D_{K^h}\), we can write
\[
\begin{aligned}
D_{K^h}
=\left(1-\frac{\Delta t}{\phi |K^h|}
\sum_{\substack{F\subset\partial K^h\\ u_{K^{h,F}}>0}}
|F|u_{K^{h,F}}\theta_{K^h}
\right)d_{K^h}
+
\frac{\Delta t}{\phi |K^h|}
\sum_{\substack{F\subset\partial K^h\\ u_{K^{h,F}}<0}}
|F|(-u_{K^{h,F}})\theta_i d_{K^{h,i}}.
\end{aligned}
\]
By the CFL condition, all coefficients in the above convex combination
are nonnegative. Then applying Jensen's inequality in $\varphi(t)=t^2$ gives
\[
\begin{aligned}
|D_{K^h}|^2
&\le
r_{K^h}
\Bigg[
\Bigg(
1-
\frac{\Delta t}{\phi |K^h|}
\sum_{\substack{F\subset\partial K^h\\ u_{K^{h,F}}>0}}
|F|u_{K^{h,F}}\theta_{K^h}
\Bigg)|d_{K^h}|^2
+
\frac{\Delta t}{\phi |K^h|}
\sum_{\substack{F\subset\partial K^h\\ u_{K^{h,F}}<0}}
|F|(-u_{K^{h,F}})\theta_i |d_{K^{h,i}}|^2
\Bigg].
\end{aligned}
\]
 Using (4.16), we further get
\begin{align*}
|D_{K^h}|^2
&\le
(1+C_{\rm st}\Delta t)
\Bigg[
\Bigg(
1-
\frac{\Delta t}{\phi |K^h|}
\sum_{\substack{F\subset\partial K^h\\ u_{K^{h,F}}>0}}
|F|u_{K^h,F}\theta_{K^h}
\Bigg)|d_{K^h}|^2
\\
&\qquad\qquad
+
\frac{\Delta t}{\phi |K^h|}
\sum_{\substack{F\subset\partial K^h\\ u_{K^{h,F}}<0}}
|F|(-u_{K^{h,F}})\theta_i |d_{K^{h,i}}|^2
\Bigg].
\end{align*}

Multiplying \(\phi |K^h|\) and summing over
\(K^h\in\mathcal G_h^{n+1}\), we obtain
\[
\begin{aligned}
\sum_{K^h\in\mathcal G_h^{n+1}}
\phi |K^h||D_{K^h}|^2
\le&
(1+C_{\rm st}\Delta t)
\sum_{K^h\in\mathcal G_h^{n+1}}
\phi |K^h|
\Bigg(1-
\frac{\Delta t}{\phi |K^h|}
\sum_{\substack{F\subset\partial K^h\\ u_{K^{h,F}}>0}}
|F|u_{K^{h,F}}\theta_{K^h}
\Bigg)|d_{K^h}|^2
\\
&\quad
+
(1+C_{\rm st}\Delta t)
\sum_{K^h\in\mathcal G_h^{n+1}}
\Delta t
\sum_{\substack{F\subset\partial K^h\\ u_{K^{h,F}}<0}}
|F|(-u_{K^{h,F}})\theta_i |d_{K^{h,i}}|^2 .
\end{aligned}
\tag{4.17}
\]

We now use the conservation of the finite volume flux in full mesh under the zero Neumann boundary condition:
\[
\begin{aligned}
&\sum_{K^h\in\mathcal T_h}
\phi |K^h|
\left(
1-
\frac{\Delta t}{\phi |K^h|}
\sum_{\substack{F\subset\partial K^h\\ u_{K^{h,F}}>0}}
|F|u_{K^{h,F}}\theta_{K^h}
\right)|d_{K^h}|^2
\\
&\quad
+
\sum_{K^h\in\mathcal T_h}
\Delta t
\sum_{\substack{F\subset\partial K^h\\ u_{K^{h,F}}<0}}
|F|(-u_{K^{h,F}})\theta_i |d_{K^{h,i}}|^2
=
\sum_{K^h\in\mathcal T_h}
\phi |K^h||d_{K^h}|^2 .
\end{aligned}
\tag{4.18}
\]
If non-zero Neumann boundary conditions are prescribed, the same argument is
applied after subtracting the common boundary contribution.

Since all the terms in the two sums on the left-hand side of (4.18) are
nonnegative, the restricted sum over \(K^h\in\mathcal G_h^{n+1}\) is no larger
than the full-grid sum. Moreover, by the closure property of the smooth region,
only the values \(d_{K^h}\) with \(K^h\in\mathcal G_h^n\) can appear in the
restricted update. Thus (4.17) and (4.18) imply
\[
\sum_{K^h\in\mathcal G_h^{n+1}}
\phi |K^h||D_{K^h}|^2
\le
(1+C_{\rm st}\Delta t)
\sum_{K^h\in\mathcal G_h^n}
\phi |K^h||d_{K^h}|^2 ,
\]
which is equivalent to
\[
\|\mathcal L_uS_1-\mathcal L_uS_2\|_{2,\mathcal G_h^{n+1}}^2
\le
(1+C_{\rm st}\Delta t)
\|S_1-S_2\|_{2,\mathcal G_h^n}^2 .
\]
which gives:
\[
\|\mathcal L_uS_1-\mathcal L_uS_2\|_{2,\mathcal G_h^{n+1}}
\le
(1+C_{\rm st}\Delta t)
\|S_1-S_2\|_{2,\mathcal G_h^n} .
\]
This proves the stability estimate on the smooth region.

Step 3: Velocity perturbation on the smooth region.
Let
\[
e_{u_t}^{n+1}=u_t^{h,n+1}-u_t^{H,n+1}.
\]
By \eqref{4.7} and \eqref{4.9},
\[
b(e_{u_t}^{n+1},q)=0,
\qquad
\forall q\in Q_h.
\]
Thus, on each fine cell \(K^h\),
\[
\sum_{F\subset\partial K^h}|F|\,e_{u_t}^{n+1}\cdot n =0.
\tag{4.19}
\]

Because of the same-direction assumption (4.10), the upwind cell is the same
for \(u_t^{h,n+1}\) and \(u_t^{H,n+1}\). Therefore, for
\(S=S_w^{H,n}\),
\[
\begin{aligned}
&
(\mathcal L_{u_t^{h,n+1}}S-
\mathcal L_{u_t^{H,n+1}}S)_{K^h}
=
-\frac{\Delta t}{\phi |K^h|}
\sum_{F\subset\partial K^h}
|F|\,f_w(S_{w}^{*,H})\,e_{u_t}^{n+1}\cdot n_{K^{h,F}}.
\end{aligned}
\tag{4.20}
\]
Using (4.19), we subtract \(f_w(S_{K^h}^{H,n})\) and obtain
\[
\begin{aligned}
\sum_{F\subset\partial K^h}
|F|f_w(S_{w}^{*,H})e_{u_t}^{n+1}\cdot n_{K^{h,F}}
=
\sum_{F\subset\partial K^h}
|F|
\left[
f_w(S_{w}^{*,H})-f_w(S_{K^h}^{H,n})
\right]
e_{u_t}^{n+1}\cdot n_{K^{h,F}}.
\end{aligned}
\]
On an outflow edge, \(S_{w}^{*,H}=S_{K^h}^{H,n}\), so the corresponding term
vanishes. On an inflow edge \(F=K^h \cap K^{h,i}\), the upwind value is \(S_{K^{h,i}}\).
Since we are on the smooth stencil,
\[
|S_{K^{h,i}}^{H,n}-S_{K^h}^{H,n}|
\le \frac{\zeta_h}{\sqrt{\omega_{n,F}^H}}. 
\]
Thus
\[
|f_w(S_{K^{h,i}}^{H,n})-f_w(S_{K^h}^{H,n})|
\le
\frac{L_f\zeta_h}{\sqrt{\omega_{n,F}^H}}.
\]
Consequently,
\[
\left|
\sum_{F\subset\partial K^h}
|F|f_w(S_{w}^{*,H})e_{u_t}^{n+1}\cdot n
\right|
\le
C L_f\zeta_h
\sum_{F\subset\partial K^h}
\frac{|F|}{\sqrt{\kappa_n^H} }\,|e_{u_t}^{n+1}\cdot n|.
\]
Then, by Cauchy's inequality, we have
\[
\begin{aligned}
\left(
\sum_{F\subset\partial K^h}
|F|(\kappa_n^H)^{-\frac12}
\left|e_{u_t}^{n+1}\cdot n\right|
\right)^2 &\le
\left(\sum_{F\subset\partial K^h}|F|\right)
\left(
\sum_{F\subset\partial K^h}
|F|(\kappa_n^H)^{-1}
\left|e_{u_t}^{n+1}\cdot n\right|^2
\right)  \\
& \le
C h
\sum_{F\subset\partial K^h}
|F|(\kappa_n^H)^{-1}
\left|e_{u_t}^{n+1}\cdot n\right|^2 .
\end{aligned}
\]
Applying the inverse trace estimate for \(RT_0\) functions gives
\[
\begin{aligned}
h\sum_{F\subset\partial K^h}
|F|(\kappa_n^H)^{-1}
\left|e_{u_t}^{n+1}\cdot n\right|^2  \le
C\left\|e_{u_t}^{n+1}\right\|_{a_n^H(K^h)}^2 .
\end{aligned}
\]
Consequently,
\[
\frac{1}{|K^h|}
\left|
\sum_{F\subset\partial K^h}
|F| f_w(S_F^{*,H})
e_{u_t}^{n+1}\cdot n
\right|^2
\le
\frac {C L_f^2 \zeta_h^2}{h^2}
\left\|e_{u_t}^{n+1}\right\|_{a_n^H(K^h)}^2 .
\]
Using (4.20), summing over \(K^h\in\mathcal G_h^{n+1}\), and using the
finite-overlap property of neighboring cells, we get
\[
\begin{aligned}
&
\|\mathcal L_{u_t^{h,n+1}}S-
\mathcal L_{u_t^{H,n+1}}S\|_{2,\mathcal G_h^{n+1}}^2
\le
C\Delta t^2 h^{-2}\phi^{-2}L_f^2\zeta_h^2
\|e_{u_t}^{n+1}\|_{a_n^H(\Omega)}^2.
\end{aligned}
\tag{4.21}
\].

Step 4: Smooth-region recursion. The fine-scale and multiscale saturation updates are
\[
S_w^{h,n+1}
=
\mathcal L_{u_t^{h,n+1}}S_w^{h,n},
\qquad
S_w^{H,n+1}
=
\mathcal L_{u_t^{H,n+1}}S_w^{H,n}.
\]
Therefore,
\[
\begin{aligned}
e_S^{n+1}
&=
\mathcal L_{u_t^{h,n+1}}S_w^{h,n}
-
\mathcal L_{u_t^{H,n+1}}S_w^{H,n}
\\
&=
\left(
\mathcal L_{u_t^{h,n+1}}S_w^{h,n}
-
\mathcal L_{u_t^{h,n+1}}S_w^{H,n}
\right)+
\left(
\mathcal L_{u_t^{h,n+1}}S_w^{H,n}
-
\mathcal L_{u_t^{H,n+1}}S_w^{H,n}
\right).
\end{aligned}
\]
Taking the \(\|\cdot\|_{2,\mathcal G_h^{n+1}}\)-norm and applying
(4.14) and (4.21) yields
\[
E_G^{n+1}
\le
(1+C_{\rm st}\Delta t)E_G^n
+
C_{\rm v}\Delta t
\|e_{u_t}^{n+1}\|_{a_n^H(\Omega)},
\]
where $C_{\rm v}=C\phi^{-1}L_f\zeta_hh^{-1}$.
\end{proof}

\begin{lemma}
Recall that \(i\) denote the last time level at which the
multiscale space was updated, and \(n\) is the current time level. $\mathbf{u}_t^{h,i+1}$ and $\mathbf{u}_t^{h,n+1}$ are solutions to (\ref{4.6})-(\ref{4.7}) at corresponding time levels. The following estimation holds:
\begin{align*}
\|\mathbf{u}_t^{h,n+1} -  \mathbf{u}_t^{h,i+1} \|_{a_n} &\leq 2 \|{\mathbf{u}_t^h}\|_{L^\infty(0,T ; \Omega)}  \|\kappa_i^{-\frac{1}{2}}-\kappa_n^{-\frac{1}{2}}\|_{L^2(\Omega)}
\end{align*}
\label{lemma 4.4}
\end{lemma}
\begin{proof}

 First we have 
 \begin{align*}
\|\mathbf{u}_t^{h,n+1} -  \mathbf{u}_t^{h,i+1} \|_{a_n} \leq \|\kappa_n^{-\frac{1}{2}}\mathbf{u}_t^{h,n+1}- \kappa_i^{-\frac{1}{2}}\mathbf{u}_t^{h,i+1}\|_{L^2(\Omega)}+\|\kappa_i^{-\frac{1}{2}}\mathbf{u}_t^{h,i+1}- \kappa_n^{-\frac{1}{2}}\mathbf{u}_t^{h,i+1}\|_{L^2(\Omega)}
\end{align*} 

For the second term on the right-hand side, we directly obtain 
\begin{align*}
    \|\kappa_i^{-\frac{1}{2}}\mathbf{u}_t^{h,i+1}- \kappa_n^{-\frac{1}{2}}\mathbf{u}_t^{h,i+1}\|_{L^2(\Omega)}\leq \|{\mathbf{u}_t^h}\|_{L^\infty(0,T ; \Omega)}  \|\kappa_n^{-\frac{1}{2}}-\kappa_i^{-\frac{1}{2}}\|_{L^2(\Omega)}
    \tag{4.22}
\end{align*}

For the first term, we recall the mixed variational forms at $t_{n}$ and $t_i$. By taking $\mathbf{v} = \mathbf{u}_t^{h,n+1}$ in (\ref{4.6}) and $q = p_h^{h,n+1}$ in (\ref{4.7}), we have:
\begin{equation*}
\int_{\Omega} \kappa_n^{-1} |\mathbf{u}_t^{h,n+1}|^2 = \int_{\Omega} F_t p_w^{h,n+1}
\end{equation*}
By choosing $\mathbf{v} = \mathbf{u}_t^{h,i+1}$ and $q = p_w^{h,n+1}$ as test functions in the corresponding systems at different time levels, we obtain relations:
\begin{equation}
\int_{\Omega} \kappa_n^{-1} \mathbf{u}_t^{h,n+1} \cdot \mathbf{u}_t^{h,i+1} = \int_{\Omega} \kappa_n^{-1} |\mathbf{u}_t^{h,n+1}|^2, 
\tag{4.23}
\end{equation}
and similarly,
\begin{equation}
\int_{\Omega} \kappa_i^{-1} \mathbf{u}_t^{h,i+1} \cdot \mathbf{u}_t^{h,n+1} = \int_{\Omega} \kappa_i^{-1} |\mathbf{u}_t^{h,i+1}|^2. 
\tag{4.24}
\end{equation}
Using the relations (4.23) and (4.24), the difference in weighted $L^2$ norms can be expanded and simplified as:
\begin{align}
\int_{\Omega} |\kappa_n^{-\frac{1}{2}} \mathbf{u}_t^{h,n+1} - \kappa_i^{-\frac{1}{2}} \mathbf{u}_t^{h,i+1}|^2 &= \int_{\Omega} |\kappa_n^{-\frac{1}{2}}-\kappa_i^{-\frac{1}{2}}|^2 |\mathbf{u}_t^{h,n+1} \cdot \mathbf{u}_t^{h,i+1}| \nonumber \\
&\le \|\mathbf{u}_t^h\|_{L^\infty(0,T ; \Omega)}^2 \|\kappa_n^{-\frac{1}{2}}-\kappa_i^{-\frac{1}{2}}\|_{L^2(\Omega)}^2 
\tag{4.25}
\end{align}

Combined with (4.22) and (4.25), we get the conclusion.
\end{proof}

\begin{lemma}
Assume at time level $t_i$, the multiscale space is newly updated, let $\mathbf u_t^{h,i+1}$ denote the fine-scale solution of
\eqref{4.6}--\eqref{4.7}  and $\mathbf u_t^{H,i+1}$ be the postprocessed multiscale velocity obtained from \eqref{4.8}--\eqref{4.9}.  Then
\begin{equation*}
\|\mathbf u_t^{h,i+1}-\mathbf u_t^{H,i+1}\|_{a_i^H}
\le
CH
\left(
\frac{1}{\sqrt{\Lambda}}
+\|F_t\|_{L^2(\Omega)}
\right)
+
C_\lambda
\|{\mathbf{u}_t^h}\|_{L^\infty(0,T ; \Omega)} 
\|e_S^i\|_{a_i^H} .
\end{equation*}
where $C$ is independent of the contrast and the mesh size.
\label{lemma 4.5}
\end{lemma}

\begin{proof}
Let
$\mathbf u_{\mathrm{ms}}^{H,i+1}$ denote the mixed CEM velocity before
the selective postprocessing step. We introduce the auxiliary mixed
pair
\(
(\widetilde{\mathbf u}_t^{h,i+1},
 \widetilde p_w^{h,i+1})
\in V_h\times Q_h/\mathbb R
\)
satisfying
\begin{align*}
a_i^H
\left(
\widetilde{\mathbf u}_t^{h,i+1},\mathbf v
\right)
-
b
\left(
\mathbf v,\widetilde p_w^{h,i+1}
\right)
&=0,
&&
\forall\mathbf v\in V_h,
\nonumber \tag{4.26} \label{4.26}\\
b
\left(
\widetilde{\mathbf u}_t^{h,i+1},q
\right)
&=(F_t,q),
&&
\forall q\in Q_h.
\end{align*}
By Theorem~1 in \cite{ChungEfendievLeung2018MixedCEM}, we have
\begin{equation}
\left\|
\widetilde{\mathbf u}_t^{h,i+1}
-
\mathbf u_{\mathrm{ms}}^{H,i+1}
\right\|_{a_i^H}^2
\lesssim
H^2
\left(
\frac{1}{\Lambda}
+
\|F_t\|_{L^2(\Omega)}^2
\right).
\tag{4.27}
\label{4.27}
\end{equation}

We next estimate the perturbation induced by the selective
postprocessing. Since
\(
\mathcal A_{H,\mathrm{pt}}^{i,i+1}
=
\mathcal T_{H,\mathrm{pt}}^{i+1},
\)
we define
\[
\mathbf e_K
=
\left(
\mathbf u_t^{H,i+1}
-
\mathbf u_{\mathrm{ms}}^{H,i+1}
\right)|_K,
\qquad
K\in\mathcal A_{H,\mathrm{pt}}^{i,i+1},
\]
and set $\mathbf e_K=0$ for
$K\notin\mathcal A_{H,\mathrm{pt}}^{i,i+1}$.
Since the postprocessing preserves the normal flux on $\partial K$,
$\mathbf e_K\in V_{h,0}(K)$. Moreover, fine-grid conservation gives
\[
b_K(\mathbf e_K,q)
=
\left(
F_t-\nabla\cdot\mathbf u_{\mathrm{ms}}^{H,i+1},
q
\right)_K,
\qquad
\forall q\in Q_h(K).
\]
The local mixed stability of the
$\mathrm{RT}_0$--$P_0$ pair therefore yields
\(
\|\mathbf e_K\|_{a_i^H(K)}
\lesssim
H
\left\|
F_t-\nabla\cdot\mathbf u_{\mathrm{ms}}^{H,i+1}
\right\|_{L^2(K)}.
\) 
Consequently,
\begin{align}
\left\|
\mathbf u_t^{H,i+1}
-
\mathbf u_{\mathrm{ms}}^{H,i+1}
\right\|_{a_i^H}^2
=
\sum_{K\in\mathcal A_{H,\mathrm{pt}}^{i,i+1}}
\|\mathbf e_K\|_{a_i^H(K)}^2
\nonumber &\lesssim
H^2
\sum_{K\in\mathcal A_{H,\mathrm{pt}}^{i,i+1}}
\left\|
F_t-\nabla\cdot\mathbf u_{\mathrm{ms}}^{H,i+1}
\right\|_{L^2(K)}^2
\nonumber\\
&\le
H^2
\left\|
F_t-\nabla\cdot\mathbf u_{\mathrm{ms}}^{H,i+1}
\right\|_{L^2(\Omega)}^2
\lesssim
H^2\|F_t\|_{L^2(\Omega)}^2.
\tag{4.28}
\label{4.28}
\end{align}
Combining \eqref{4.27} and \eqref{4.28} gives
\[
\|\widetilde{\mathbf u}_t^{h,i+1}
-\mathbf u_t^{H,i+1}\|_{a_i^H}
\lesssim
H\left(\frac{1}{\sqrt{\Lambda}}
+\|F_t\|_{L^2(\Omega)}\right).
\tag{4.29}
\]

It remains to estimate
\(\|\mathbf u_t^{h,i+1}-\mathbf u_t^{H,i+1}\|_{a_i^H}\).
By the triangle inequality, we have
\[
\begin{aligned}
\|\mathbf u_t^{h,i+1}-\mathbf u_t^{H,i+1}\|_{a_i^H} \le
\|\mathbf u_t^{h,i+1}-\tilde{\mathbf u}_t^{h,i+1}\|_{a_i^H}
+
\|\tilde{\mathbf u}_t^{h,i+1}-\mathbf u_t^{H,i+1}\|_{a_i^H},
\end{aligned}
\tag{4.30}
\]
For the first term on the right-hand side, using \eqref{4.6} and \eqref{4.26}, we have
\[
a_i(\mathbf u_t^{h,i+1},v)
=
a_i^H(\tilde{\mathbf u}_t^{h,i+1},v),
\qquad \forall v\in V_h .
\]
Subtracting \(a_i^H(\mathbf u_t^{h,i+1},v)\) from both sides and taking
\(v=\tilde{\mathbf u}_t^{h,i+1}-\mathbf u_t^{h,i+1}\), we obtain
\[
\begin{aligned}
\|\tilde{\mathbf u}_t^{h,i+1}-\mathbf u_t^{h,i+1}\|_{a_i^H}^2
&=
\int_{\Omega}
\left(
\frac{1}{\kappa_i}
-
\frac{1}{\kappa_i^H}
\right)
\mathbf u_t^{h,i+1}
\cdot
\left(
\tilde{\mathbf u}_t^{h,i+1}-\mathbf u_t^{h,i+1}
\right)  \\
&=
\int_{\Omega}\frac{1}{\sqrt{\kappa_i^H}} \left( \frac{\kappa_i^H}{\kappa_i} -1\right) \mathbf{u}_t^{h,i+1}\cdot \frac{1}{\sqrt{\kappa_i^H}} \left(\tilde{\mathbf u}_t^{h,i+1}-\mathbf u_t^{h,i+1}\right)
\end{aligned}
\]
Hence, by the Cauchy--Schwarz inequality
\begin{align*}
\|\tilde{\mathbf u}_t^{h,i+1}-\mathbf u_t^{h,i+1}\|_{a_i^H}^2
&\le \|{\mathbf{u}_t^h}\|_{L^\infty(0,T ; \Omega)} 
\|\frac{\lambda_t(S_w^{H,i})-\lambda_t(S_w^{h,i})}{\lambda_t(S_w^{h,i})}\|_{a_i^H}
\|\tilde{\mathbf u}_t^{h,i+1}-\mathbf u_t^{h,i+1}\|_{a_i^H}
\end{align*}
Since the mobility function $\lambda_t(S_w)$ is bounded and has Lipschitz condition, we have
\[
\begin{aligned}
\|\tilde{\mathbf u}_t^{h,i+1}-\mathbf u_t^{h,i+1}\|_{a_i^H}^2
&\le
C_\lambda
\|{\mathbf{u}_t^h}\|_{L^\infty(0,T ; \Omega)} 
\|e_S^i\|_{a_i^H}
\|\tilde{\mathbf u}_t^{h,i+1}-\mathbf u_t^{h,i+1}\|_{a_i^H}.
\end{aligned}
\]
Therefore,
\[
\|\tilde{\mathbf u}_t^{h,i+1}-\mathbf u_t^{h,i+1}\|_{a_i^H}
\le
C_\lambda
\|{\mathbf{u}_t^h}\|_{L^\infty(0,T ; \Omega)} 
\|e_S^i\|_{a_i^H} .
\tag{4.31}
\]

Combining (4.29)-(4.31), we finally obtain
\[
\|\mathbf u_t^{h,i+1}-\mathbf u_t^{H,i+1}\|_{a_i^H}
\le
CH
\left(
\frac{1}{\sqrt{\Lambda}}
+\|F_t\|_{L^2(\Omega)}
\right)
+
C_\lambda
\|{\mathbf{u}_t^h}\|_{L^\infty(0,T ; \Omega)} 
\|e_S^i\|_{a_i^H} .
\]

The proof is complete.

\end{proof}

 Remark that to obtain the $O(H)$ convergence in $\|\mathbf{u}_t^{h,i+1} - \mathbf{u}_t^{H,i+1} \|_{a_i^H}$, we need to choose the size of the oversampling domain $l=O(log(\mathbf{K} /{H^2}))$. For a detailed derivation of these contents, refer to \cite{ChungEfendievLeung2018MixedCEM}.\\

\begin{theorem}

\textit{We recall that $\mathbf{u}_t^{h,n+1}$ and $\mathbf{u}_t^{H,n+1}$ are solutions to \eqref{4.6}-\eqref{4.7} and \eqref{4.8}-\eqref{4.9} at time $t_{n}$. For a given tolerance \(\varepsilon\), let \(i\) be the most recent time level at which the multiscale space is updated. Then we have
the following estimate:}
\begin{equation*}
\begin{split}
  \|e_{\mathbf{u}_t}^{n+1}\|_{a^H_n(\Omega)}
    &\le
    C_0 
    \|\mathbf{u}_t^h\|_{L^\infty(0,T;\Omega)}
    \varepsilon 
    + C_H 
    H\left(
\frac{1}{\sqrt{\Lambda}} 
+\|F_t\|_{L^2(\Omega)}
\right)
 \\  & + C_{\lambda}' \frac{1}{\sqrt{\kappa_{min}}}
    \|\mathbf{u}_t^h\|_{L^\infty(0,T;\Omega)}
    \left(\|e_S^n\|_{L^2(\Omega)}+ \|e_S^i\|_{L^2(\Omega)} \right)
\end{split}
\end{equation*}
Where the coefficients $C_0$, $C_H$ and $C_{\lambda}'$ are independent of the mesh size $h$, $H$ and contrast $\kappa$.
\end{theorem}

\begin{proof}
    We divide the proof into 3 steps:\\

\noindent\textbf{Step 1}: First, by the cumulative construction of the marked set,
\(
\widetilde V_{\mathrm{ms}}^{i,i+1}
\subseteq
\widetilde V_{\mathrm{ms}}^{i,n+1}
\subseteq V_h
\)
So we have
\(
\mathbf z
=
\mathbf u_t^{H,i+1}
-
\mathbf u_t^{H,n+1}
\in
\widetilde V_{\mathrm{ms}}^{i,n+1}.
\)

Subtracting the enriched multiscale equation \eqref{4.8} from the fine-grid
equation \eqref{4.6} at time $t_n$ and taking $\mathbf v=\mathbf z$, we obtain
\begin{align*}
a_n(\mathbf u_t^{h,n+1},\mathbf z)
-
a_n^H(\mathbf u_t^{H,n+1},\mathbf z)
&=
b(\mathbf z,p_w^{h,n+1}-p_w^{H,n+1})
=0.
\end{align*}\\
Adding $\int_{\Omega} \frac{1}{\kappa_n^H} \mathbf{u}_t^{h,n+1}\cdot \mathbf{z}$ to both sides, we can deduce:
\begin{align}
\int_{\Omega} \frac{1}{\kappa_n^H} (\mathbf{u}_t^{h,n+1} - \mathbf{u}_t^{H,n+1}) \cdot \mathbf{z} 
= \int_{\Omega} \left( \frac{1}{\kappa_n^H} - \frac{1}{\kappa_n} \right) \mathbf{u}_t^{h,n+1}\cdot \mathbf{z} \tag{4.32}
\end{align}

Next we will estimate the equality (4.32) above. For the right hand side we have the following estimate:
\begin{align*}
\int_{\Omega} \left( \frac{1}{\kappa_n^H} - \frac{1}{\kappa_n} \right) \mathbf{u}_t^{h,n+1}\cdot \mathbf{z}&=\int_{\Omega} \frac{1}{\sqrt{\kappa_n^H}} \left( 1 - \frac{\kappa_n^H}{\kappa_n} \right) \mathbf{u}_t^{h,n+1}\cdot \frac{1}{\sqrt{\kappa_n^H}}\mathbf{z} \\
&\le \| \mathbf{u}_t^h\|_{L^\infty(0,T ; \Omega)} \|\frac{\lambda_t(S_w^{h,n})-\lambda_t(S_w^{H,n})}{\lambda_t(S_w^{h,n})}\|_{a_n^H} \|\mathbf{z}\|_{a_n^H} \\
&\le C_{\lambda}\|\mathbf{u}_t^h\|_{L^\infty(0,T ; \Omega)}\|e_S^n\|_{a_n^H}\|\mathbf{z}\|_{a_n^H}
\end{align*}
For the left hand side, we split it into:
\begin{align*}
\int_{\Omega} (\kappa_n^H)^{-1} (\mathbf{u}_t^{h,n+1} - \mathbf{u}_t^{H,n+1}) \cdot \mathbf{z} 
= \int_{\Omega} (\kappa_n^H)^{-1} (\mathbf{u}_t^{h,n+1} - \mathbf{u}_t^{H,i+1}+\mathbf{z}) \cdot \mathbf{z} 
\end{align*}
Combined with (4.32) and the equation above, we can obtain:
\begin{align*}
\int_{\Omega} (\kappa_n^H)^{-1} \mathbf{z} \cdot \mathbf{z} 
= \int_{\Omega} (\kappa_n^H)^{-1} (  \mathbf{u}_t^{H,i+1}-\mathbf{u}_t^{h,n+1}) \cdot \mathbf{z}+ \int_{\Omega} \left( \frac{1}{\kappa_n^H} - \frac{1}{\kappa_n} \right) \mathbf{u}_t^{h,n+1}\cdot \mathbf{z}
\end{align*}
Apply the Cauchy-Schwarz inequality on the right hand side and divide $\|\mathbf{z}\|_{a_n^H}$ in both sides, we have:
\begin{align*}
\|\mathbf{z}\|_{a_n^H}\le\|\mathbf{u}_t^{h,n+1}-\mathbf{u}_t^{H,i+1}\|_{a_n^H}+ C_{\lambda}\|\mathbf{u}_t^h\|_{L^\infty(0,T ; \Omega)}\|e_S^n\|_{a_n^H}
\end{align*}
By the triangle inequality in $e_{\mathbf{u}_t^{n+1}}=\mathbf{u}_t^{h,n+1}-\mathbf{u}_t^{H,i+1}+\mathbf{z}$, we finally obtain the estimate:
\begin{align*}
    \|e_{\mathbf{u}_t^{n+1}}\|_{a_n^H} 
    &\le \|\mathbf{u}_t^{h,n+1}-\mathbf{u}_t^{H,i+1}\|_{a_n^H}+\|\mathbf{z}\|_{a_n^H} \\
    &\le 2\|\mathbf{u}_t^{h,n+1}-\mathbf{u}_t^{H,i+1}\|_{a_n^H}+ C_{\lambda}\|\mathbf{u}_t^h\|_{L^\infty(0,T ; \Omega)}\|e_S^n\|_{a_n^H} \tag{4.33}
\end{align*}
\textbf{Step 2}: Then we will estimate the norm $\|\mathbf{u}_t^{h,n+1}-\mathbf{u}_t^{H,i+1}\|_{a_n^H}$. First, apply the triangle inequality
\begin{align*}
    \|\mathbf{u}_t^{h,n+1}-\mathbf{u}_t^{H,i+1}\|_{a_n^H}\le
    \|\mathbf{u}_t^{h,n+1}-\mathbf{u}_t^{h,i+1}\|_{a_n^H}+\|\mathbf{u}_t^{h,i+1}-\mathbf{u}_t^{H,i+1}\|_{a_n^H}
\end{align*}
For the first term in the right-hand side, we have
\begin{align*}
    \int_{\Omega} \frac{1}{\kappa_n^H}|\mathbf{u}_t^{h,n+1}-\mathbf{u}_t^{h,i+1}|^2=\int_{\Omega} \frac{\kappa_n}{\kappa_n^H}\cdot\frac{1}{\kappa_n}|\mathbf{u}_t^{h,n+1}-\mathbf{u}_t^{h,i+1}|^2\ &=\int_{\Omega} \frac{\lambda_t(S_w^{h,n})}{\lambda_t(S_w^{H,n})}\cdot\frac{1}{\kappa_n}|\mathbf{u}_t^{h,n+1}-\mathbf{u}_t^{h,i+1}|^2\\
   &\le \|\frac{\lambda_t(S_w^{h,n})}{\lambda_t(S_w^{H,n})}\|_{\infty} \|\mathbf{u}_t^{h,n+1}-\mathbf{u}_t^{h,i+1}\|_{a_n}^2
\end{align*}
Similarly for the second term $\|\mathbf{u}_t^{h,i+1}-\mathbf{u}_t^{H,i+1}\|_{a_n^H} \le \|\frac{\lambda_t(S_w^{H,i})}{\lambda_t(S_w^{H,n})}\|_{\infty} \|\mathbf{u}_t^{h,i+1}-\mathbf{u}_t^{H,i+1}\|_{a_i^H} $. Due to our numerical scheme P-IMPES, the saturation is bounded, which gives the these two coefficients are bounded by $C_M=1+2C_{\lambda}$ .Then combining it with lemma \ref{lemma 4.4}, lemma \ref{lemma 4.5}, we have
\begin{align*}
  \|\mathbf{u}_t^{h,n+1}-\mathbf{u}_t^{H,i+1}\|_{a_n^H}
  &\le C_{M} \|\mathbf{u}_t^{h,n+1}-\mathbf{u}_t^{h,i+1}\|_{a_n}+C_{M}\|\mathbf{u}_t^{h,i+1}-\mathbf{u}_t^{H,i+1}\|_{a_i^H}         \\
  & \le 2 C_{M}\|{\mathbf{u}_t^h}\|_{L^\infty(0,T ; \Omega)}  \|\kappa_i^{-\frac{1}{2}}-\kappa_n^{-\frac{1}{2}}\|_{L^2(\Omega)} \\
  &+CC_{M}H\left(
\frac{1}{\sqrt{\Lambda}} 
+\|F_t\|_{L^2(\Omega)}
\right) 
+
C_{\lambda} C_{M}
\|{\mathbf{u}_t^h}\|_{L^\infty(0,T ; \Omega)} 
\|e_S^i\|_{a_i^H} . \tag{4.34}
\end{align*}
\textbf{Step 3}: Combining conclusions (4.33) and (4.34) above, we obtain the estimate:
\begin{align*}
\|e_{\mathbf{u}_t}^{n+1}\|_{a_n^H(\Omega)} \le 4C_M\|{\mathbf{u}_t^h}\|_{L^\infty(0,T ; \Omega)}  \|\kappa_i^{-\frac{1}{2}}-\kappa_n^{-\frac{1}{2}}\|_{L^2(\Omega)}
    &+2CC_MH\left(
\frac{1}{\sqrt{\Lambda}} 
+\|F_t\|_{L^2(\Omega)}
\right)\\+C_{\lambda} \frac{1}{\sqrt{\kappa_{min}}}\|\mathbf{u}_t^h\|_{L^\infty(0,T ; \Omega)}\|e_S^n\|_{L^2(\Omega)} 
&+2C_{\lambda}C_M\frac{1}{\sqrt{\kappa_{min}}}\|\mathbf{u}_t^h\|_{L^\infty(0,T ; \Omega)}\|e_S^i\|_{L^2(\Omega)} 
\end{align*}
 
By the triangle inequality and the Lipschitz continuity of
\(\lambda_t(x)^{-\frac{1}{2}}\), we obtain
\begin{align*}
    \|\kappa_i^{-\frac{1}{2}}-\kappa_n^{-\frac{1}{2}}\|_{L^2(\Omega)} &\le \|\kappa_i^{-\frac{1}{2}}-(\kappa_i^H)^{-\frac{1}{2}}\|_{L^2(\Omega)}+\|(\kappa_i^H)^{-\frac{1}{2}}-(\kappa_n^H)^{-\frac{1}{2}}\|_{L^2(\Omega)}+\|(\kappa_n^H)^{-\frac{1}{2}}-\kappa_n^{-\frac{1}{2}}\|_{L^2(\Omega)} \\
&\le  \|(\kappa_i^H)^{-\frac{1}{2}}-(\kappa_n^H)^{-\frac{1}{2}}\|_{L^2(\Omega)}+C_\lambda\kappa_{min}^{-\frac{1}{2}} \left(\|e_S^n\|_{L^2(\Omega)}+ \|e_S^i\|_{L^2(\Omega)} \right)
\end{align*}

Combining these two inequality, we obtain:
\begin{align*}
\|e_{\mathbf{u}_t}^{n+1}\|_{a_n^H(\Omega)} \le &C_0\|{\mathbf{u}_t^h}\|_{L^\infty(0,T ; \Omega)}\|(\kappa_i^H)^{-\frac{1}{2}}-(\kappa_n^H)^{-\frac{1}{2}}\|_{L^2(\Omega)}
    +C_HH\left(
\frac{1}{\sqrt{\Lambda}} 
+\|F_t\|_{L^2(\Omega)}
\right)\\ &+C_{\lambda}' \frac{1}{\sqrt{\kappa_{min}}}
    \|\mathbf{u}_t^h\|_{L^\infty(0,T;\Omega)}
    \left(\|e_S^n\|_{L^2(\Omega)}+ \|e_S^i\|_{L^2(\Omega)} \right) .
\end{align*}
where $C_0=4C_M$, $C_H=2CC_M$ and $C_{\lambda}'=6C_{\lambda}C_M$ are the coefficients that are independent of the contrast and mesh size.

By our adaptive update method, before solving the system at time level
\(t_n\), we have
\[
   \eta_n= \|(\kappa_i^H)^{-\frac{1}{2}}-(\kappa_n^H)^{-\frac{1}{2}}\|_{L^2(\Omega)}
    \le \varepsilon.
\]
Otherwise, the multiscale space is regenerated and the update index is reset
to \(i=n\). Hence,
\[
\begin{split}
    \|e_{\mathbf{u}_t}^{n+1}\|_{a_n^H(\Omega)}
    &\le
    C_0
    \|\mathbf{u}_t^h\|_{L^\infty(0,T;\Omega)}
    \varepsilon 
    + C_H 
    H\left(
\frac{1}{\sqrt{\Lambda}} 
+\|F_t\|_{L^2(\Omega)}
\right)
 \\  & + C_{\lambda}' \frac{1}{\sqrt{\kappa_{min}}}
    \|\mathbf{u}_t^h\|_{L^\infty(0,T;\Omega)}
    \left(\|e_S^n\|_{L^2(\Omega)}+ \|e_S^i\|_{L^2(\Omega)} \right) .
\end{split}
\]
\end{proof}

\begin{theorem}
Recall that \(S_w^{h,n+1}\) and \(S_w^{H,n+1}\) are the fine-scale and
multiscale saturation solutions to \eqref{2.5} and \eqref{3.4}.
Then
\begin{align*}
\|e_S^{n+1}\|_{L^2(\Omega)}
\le
\frac{e^{C_1't_{n+1}}-1}{C_1'}
\left[C_0'
U\varepsilon
+C_H'
H
\left(
\frac{1}{\sqrt{\Lambda}}
+
\|F_t\|_{L^2(\Omega)}
\right)
+
C\delta_h^{\frac12}
\right]
\end{align*}
where \(U=\|\mathbf u_t^h\|_{L^\infty(0,T;L^\infty(\Omega))}\), 
\(C_1'=C\phi^{-1}
\bigl(
L_f\|F_t\|_{L^\infty(\Omega)}
+\kappa_{\min}^{-\frac12}L_f\zeta_h h^{-1}U
\bigr)
\allowbreak
+2C\phi^{-1}L_f\zeta_h h^{-1}C_\lambda'
\kappa_{\min}^{-\frac12}U\). All these constant may depend on the porosity, the relative permeability functions, the shape regularity of the mesh, but are independent of \(h\), \(H\), \(\Delta t\) and the contrast $\kappa$. 

\end{theorem}

\begin{proof}
By Lemma \ref{lemma 4.3}, we have
\[
E_G^{n+1}
\le
(1+C_{\rm st}\Delta t)E_G^n
+
C_{\rm v}\Delta t
\|e_{u_t}^{n+1}\|_{a^H_n(\Omega)}.
\]
Recall that $E_G^n=\|e_S^n\|_{2,\mathcal G_h^n}$. Using the velocity estimate in Theorem 4.2, we obtain
\[
E_G^{n+1}
\le
(1+C_{\rm st}\Delta t)E_G^n
+
C_{\rm v}\Delta t A
+
C_{\rm v}C_\lambda' \kappa_{min}^{-\frac{1}{2}} U\Delta t
\left(\|e_S^n\|_{L^2(\Omega)}+ \|e_S^i\|_{L^2(\Omega)} \right).
\tag{4.35}
\]
where $A
=
C_0
U\varepsilon
+
C_HH
\left(
\frac{1}{\sqrt{\Lambda}} 
+\|F_t\|_{L^2(\Omega)}
\right)$.

Splitting the saturation error into the smooth region and the front layer gives
\[
\|e_S^m\|_{L^2(\Omega)}
\le
\|e_S^m\|_{L^2(\Omega_{\mathcal G}^{m})}
+
\|e_S^m\|_{L^2(\Omega_{\mathcal B}^{m})}.
\]
By Lemma \ref{lemma 4.3},
\[
\|e_S^m\|_{L^2(\Omega)}
\le
E_G^m+\delta_h^{\frac{1}{2}},
\qquad m=i,n.
\tag{4.36}
\]
Substituting (4.36) into (4.35), we get
\[
E_G^{n+1}
\le
(1+C_1\Delta t)E_G^n
+
C_2\Delta t E_G^i
+
C_{\rm v}\Delta t A
+
C\Delta t\,\delta_h^{\frac{1}{2}},
\tag{4.37}
\]
where
\(
C_1=C_{\rm st}+C_{\rm v}C_\lambda' \kappa_{min}^{-\frac{1}{2}} U\), \(
C_2=C_{\rm v}C_\lambda' \kappa_{min}^{-\frac{1}{2}} U.
\) Then define
$M_G^n=\max\limits_{0\le m\le n}E_G^m$, which gives
\[
E_G^{n+1}
\le
\left(1+(C_1+C_2)\Delta t\right)M_G^n
+
\Delta t
\left(
C_{\rm v}A+C\delta_h^{\frac12}
\right).
\]
Moreover, the right-hand side is greater than \(M_G^n\).
Therefore,
\[
\begin{aligned}
M_G^{n+1}
=\max\left\{M_G^n,E_G^{n+1}\right\}\le
\left(1+(C_1+C_2)\Delta t\right)M_G^n
+
\Delta t
\left(
C_{\rm v}A+C\delta_h^{\frac12}
\right).
\end{aligned}
\]
Denoting \(
C_1'=C_1+C_2
=C_{\rm st}+2C_{\rm v}C_\lambda'
\kappa_{\min}^{-\frac12}U
\). Iterating the above inequality and using \(M_G^0=E_G^0=0\), we obtain
\[
\begin{aligned}
E_G^{n+1}
\le M_G^{n+1}
&\le
\frac{
\left(1+C_1'\Delta t\right)^{n+1}-1
}{
C_1'
}
\left(
C_{\rm v}A+C\delta_h^{\frac12}
\right)
\\
&\le
\frac{
e^{C_1't_{n+1}}-1
}{
C_1'
}
\left(
C_{\rm v}A+C\delta_h^{\frac12}
\right).
\end{aligned}
\tag{4.38}
\]
Applying the saturation-error decomposition (4.36) at time level
\(t_{n+1}\) and combining it with (4.38), we obtain
\[
\|e_S^{n+1}\|_{L^2(\Omega)}
\le
\frac{
e^{C_1't_{n+1}}-1
}{
C_1'
}
\left[
C_0'U\varepsilon
+
C_H'H
\left(
\frac{1}{\sqrt{\Lambda}}
+\|F_t\|_{L^2(\Omega)}
\right)
+
C\delta_h^{\frac12}
\right],
\]
where \(C_0'=C_{\rm v}C_0, C_H'=C_{\rm v}C_H\). These constants are independent of the mesh sizes \(h\) and \(H\), the time step \(\Delta t\), and the contrast \(\kappa\).

We remark that, under the smooth-region assumption \eqref{4.11}, the smoothness
modulus \(\zeta_h\) appearing in \(C_1'\) satisfies \(\zeta_h=O(h)\). Hence
\(\zeta_h /h=O(1)\), so the apparent factor \(h^{-1}\) in \(C_1'\) is uniformly bounded as \(h\to0\).

\end{proof}

\section{Numerical experiments}
\label{sec 5}

In this section, we present several numerical experiments to evaluate the performance of the proposed multiscale method. In particular, we focus on its mass conservation property, saturation approximation, and error convergence behavior in heterogeneous and high-contrast porous media. In the following experiments, the computational domain is $\Omega = [0, 1]^2$, the capillary pressure is modeled by the Hoteit--Firoozabadi formula $p_c(S_w)=-\frac{B_c}{\sqrt{\mathbf{K}}}\log \bar S_w$, where $B_c \geq 0$ is a capillary parameter, and $\bar S_w$ is the effective saturation defined as $\bar S_w=\frac{S_w-S_{rw}}{1-S_{rw}-S_{rn}}$, here we choose $S_{rw} = S_{rn} = 10^{-6}$. The porosity is fixed as $\phi=0.2$ throughout all numerical tests. Moreover, the relative permeabilities of the wetting and non-wetting phases are given by $k_{rw}=\bar S_w^2,\quad k_{rn}=(1-\bar S_w)^2$, the viscosity for two phases is set as $\mu_w=1$ and $\mu_n=5$. We use the following relative errors to evaluate the performance of the proposed method:
\[
    e_{u}
    =
    \frac{\|\mathbf u_h-\mathbf u_{\mathrm{ms}}\|_{\kappa^{-1}}}
         {\|\mathbf u_h\|_{\kappa^{-1}}},
    \quad
    e_{S}
    =
    \frac{\|S_h-S_{ms}\|_{L^2(\Omega)}}
         {\|S_h\|_{L^2(\Omega)}},
\]
where  $\|u\|_{\kappa^{-1}}^2 = \int_{\Omega} \mathbf{K}^{-1}|u|^2 $, $(\mathbf u_h,S_h)$ is the fine-scale solution as reference solution, and $(\mathbf u_{ms},S_{ms})$ is the multiscale solution from our scheme. \\

\textbf{Example 5.1. (Two phase conservation)}
In this example we will check the two phase conservation and unbiased property using the multiscale basis obtained in a local region of the high-contrast medium $\kappa_1$. As mentioned in the subsection \ref{4.1}, there are two ways
\eqref{4.1} and \eqref{4.2}-\eqref{4.3} to obtain the non-wetting phase saturation, denoted by $S_{n,1}$ and $S_{n,2}$ respectively. If $S_{n,1}=S_{n,2} $, the local mass conservation is attained for both two phases. We use a high contrast permeability field $\kappa_1$ to check the property. See Figure 2 for the illustration of the source term $F_t$ and the permeability field $\kappa_1$. We choose only $1$ basis in every coarse element, it is sufficient to conclude that multiscale solutions can well satisfy this conservation property since solutions computed with more multiscale bases will
obtain higher accuracy. We set the fine mesh size to $100 \times 100$, the tolerance $\varepsilon=0.1$ in adaptive algorithm and the capillary parameter $B_c=0$. The source/sink term function $F_t$ is set to be 
\[
F_t(\boldsymbol{x})=
\begin{cases}
  1,  & \boldsymbol{x}\in [0,0.01]^2,\\
 -1,  & \boldsymbol{x}\in [0.99,1]^2,\\
  0,  & \text{otherwise},
\end{cases}
\]

\begin{figure}[H]
    \centering

    \begin{subfigure}{0.47\textwidth}
        \centering
        \includegraphics[width=\textwidth]{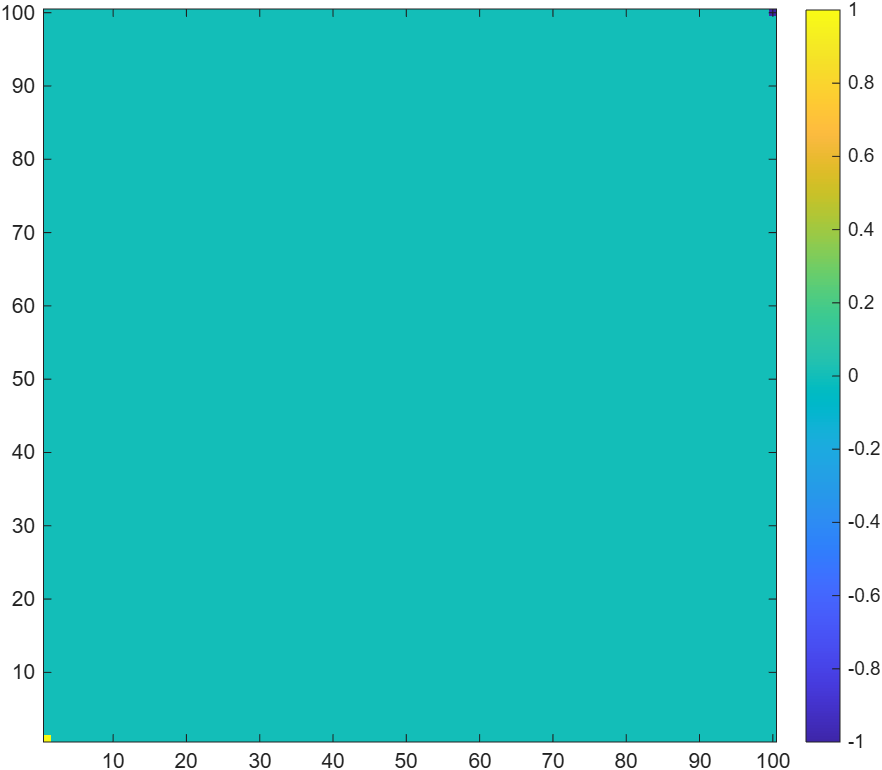}
        \caption{Source term function $F_t$.}
        \label{fig:F_t}
    \end{subfigure}
    \hfill
    \begin{subfigure}{0.47\textwidth}
        \centering
        \includegraphics[width=\textwidth]{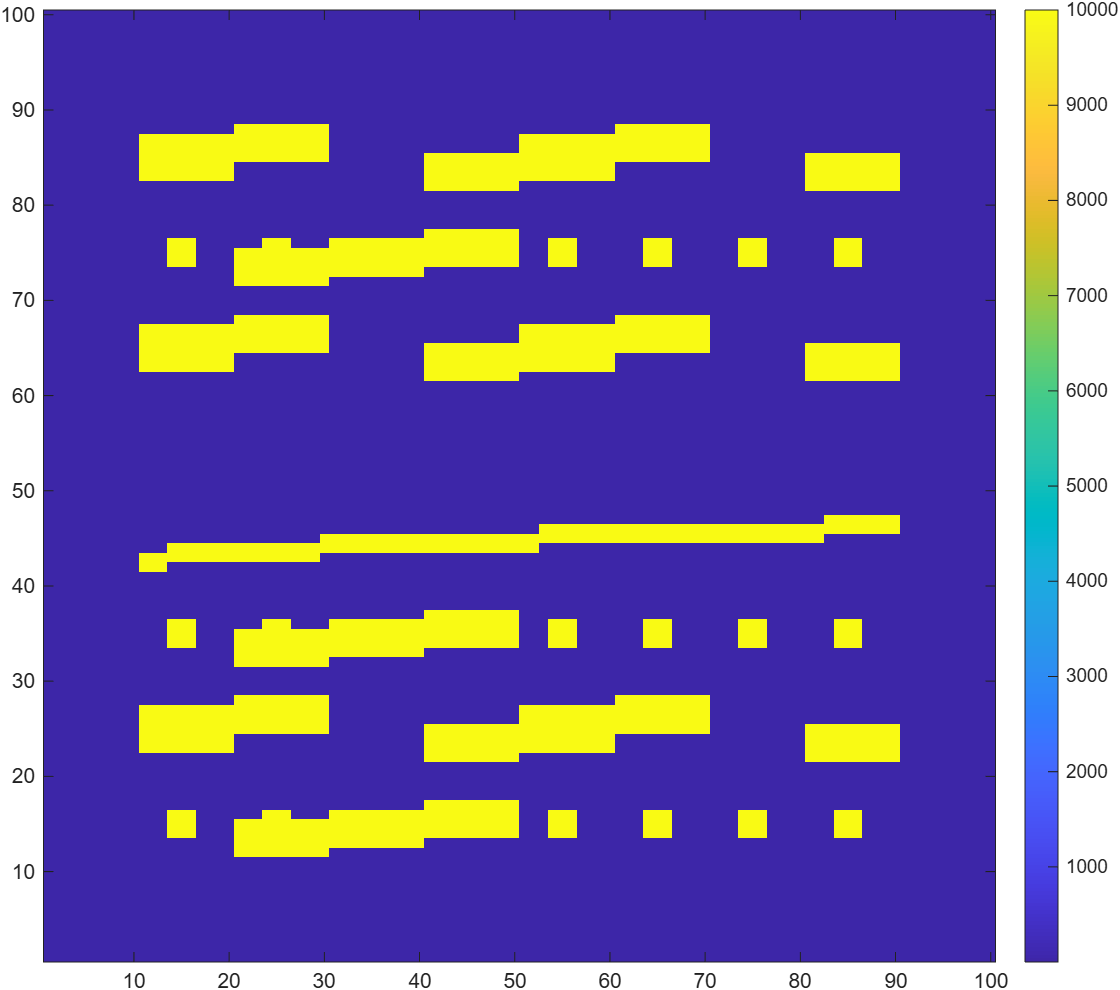}
        \caption{Permeability field $\kappa_1$.}
        \label{fig:kappa_1}
    \end{subfigure}

    \caption{Left: source term. Right: permeability field.}
    \label{fig:initial_kappa}
\end{figure}

\vspace{-5mm}

\setcounter{figthree}{0}

\refstepcounter{figthree}
\label{fig:Sn_T1}
\begin{figure}[H]
    \centering
    \setcounter{subfigure}{0}

    \begin{subfigure}{0.32\textwidth}
        \centering
        \includegraphics[width=\linewidth]{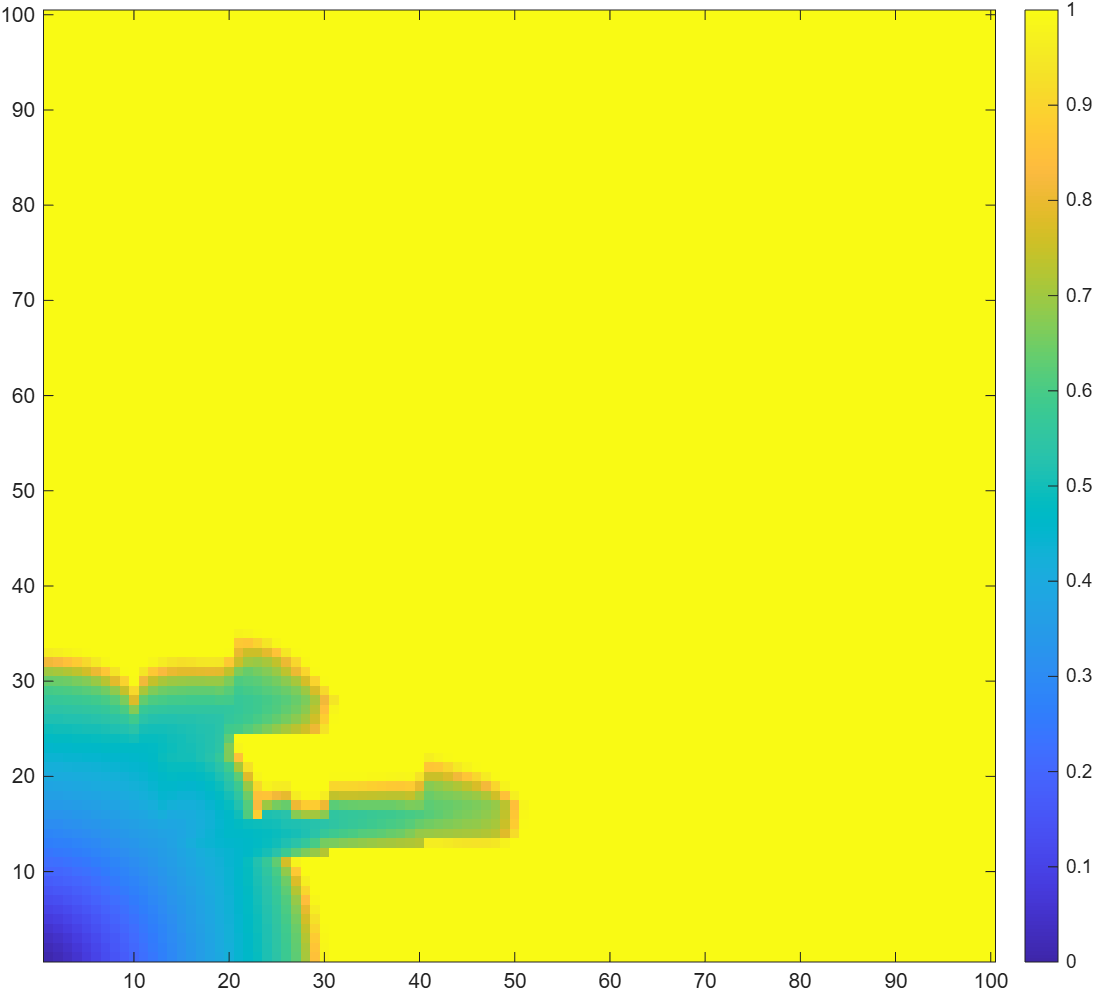}
        \caption{$S_{n,1}$ at $T=1$.}
        \label{fig:Sn1_T1}
    \end{subfigure}
    \hfill
    \begin{subfigure}{0.32\textwidth}
        \centering
        \includegraphics[width=\linewidth]{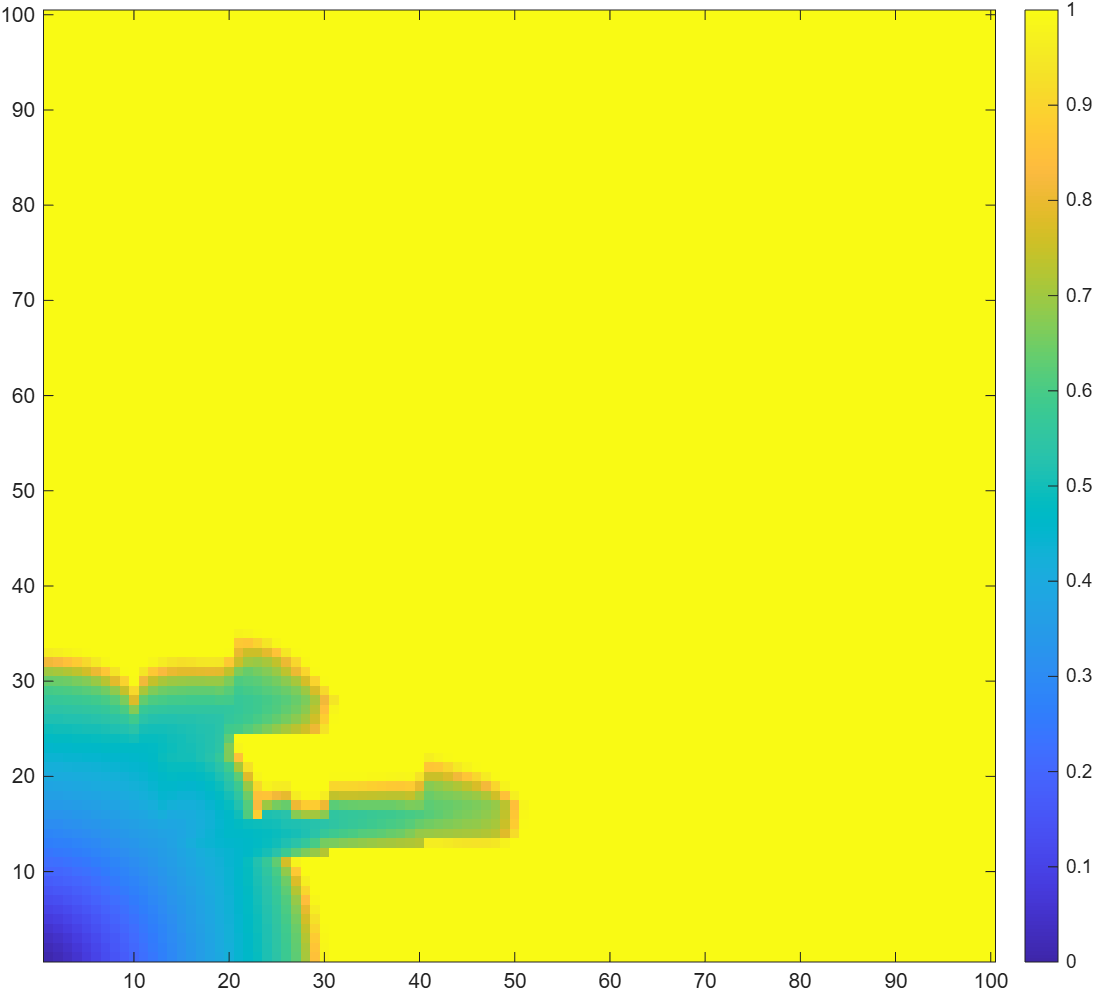}
        \caption{$S_{n,2}$ at $T=1$.}
        \label{fig:Sn2_T1}
    \end{subfigure}
    \hfill
    \begin{subfigure}{0.32\textwidth}
        \centering
        \includegraphics[width=\linewidth]{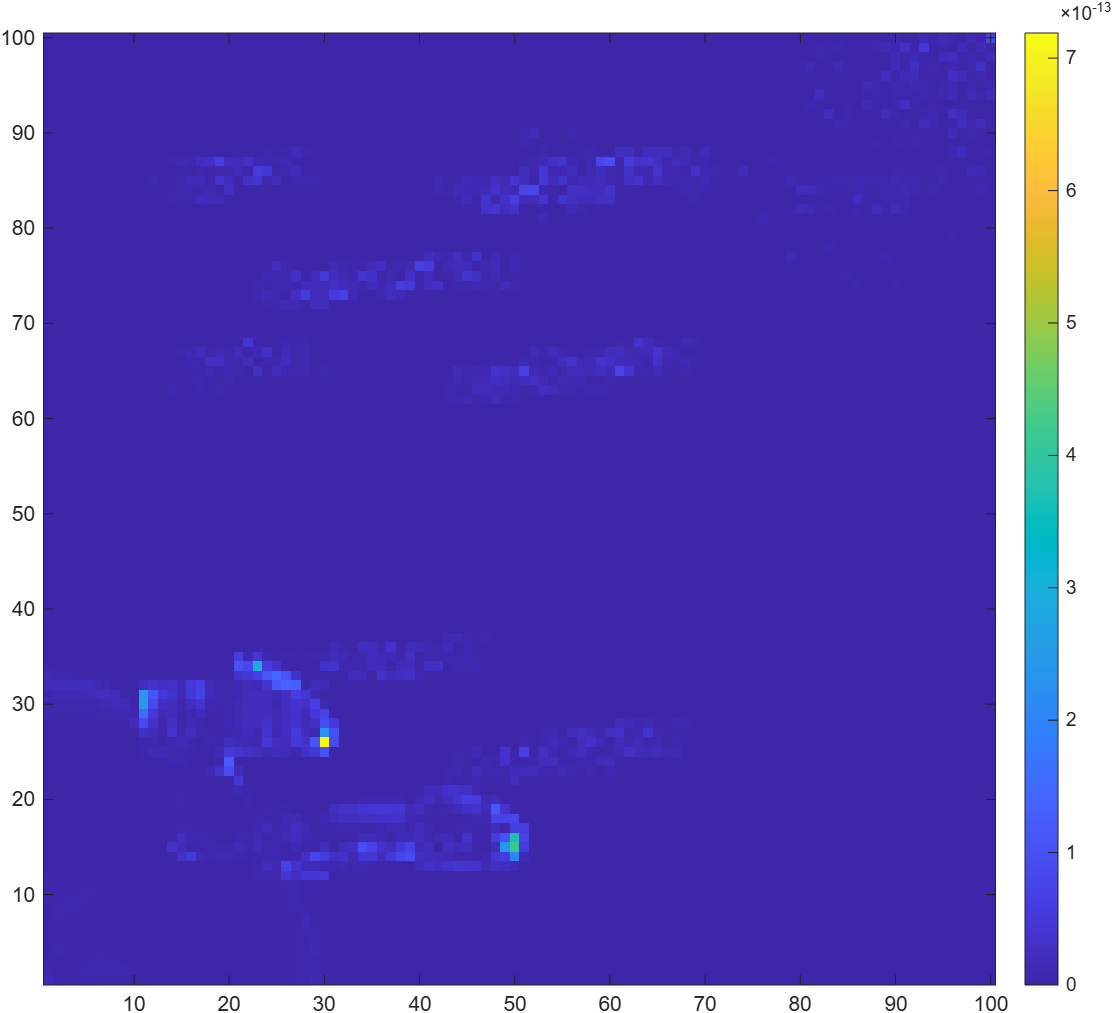}
        \caption{$|S_{n,1}-S_{n,2}|$ at $T=1$.}
        \label{fig:Sn_diff_T1}
    \end{subfigure}

    \caption*{Figure~\thefigthree. Two ways for non-wetting saturation $S_n$ at time $T=1$.}
\end{figure}

\refstepcounter{figthree}
\label{fig:Sn_T2}
\begin{figure}[H]
    \centering
    \setcounter{subfigure}{0}

    \begin{subfigure}{0.32\textwidth}
        \centering
        \includegraphics[width=\linewidth]{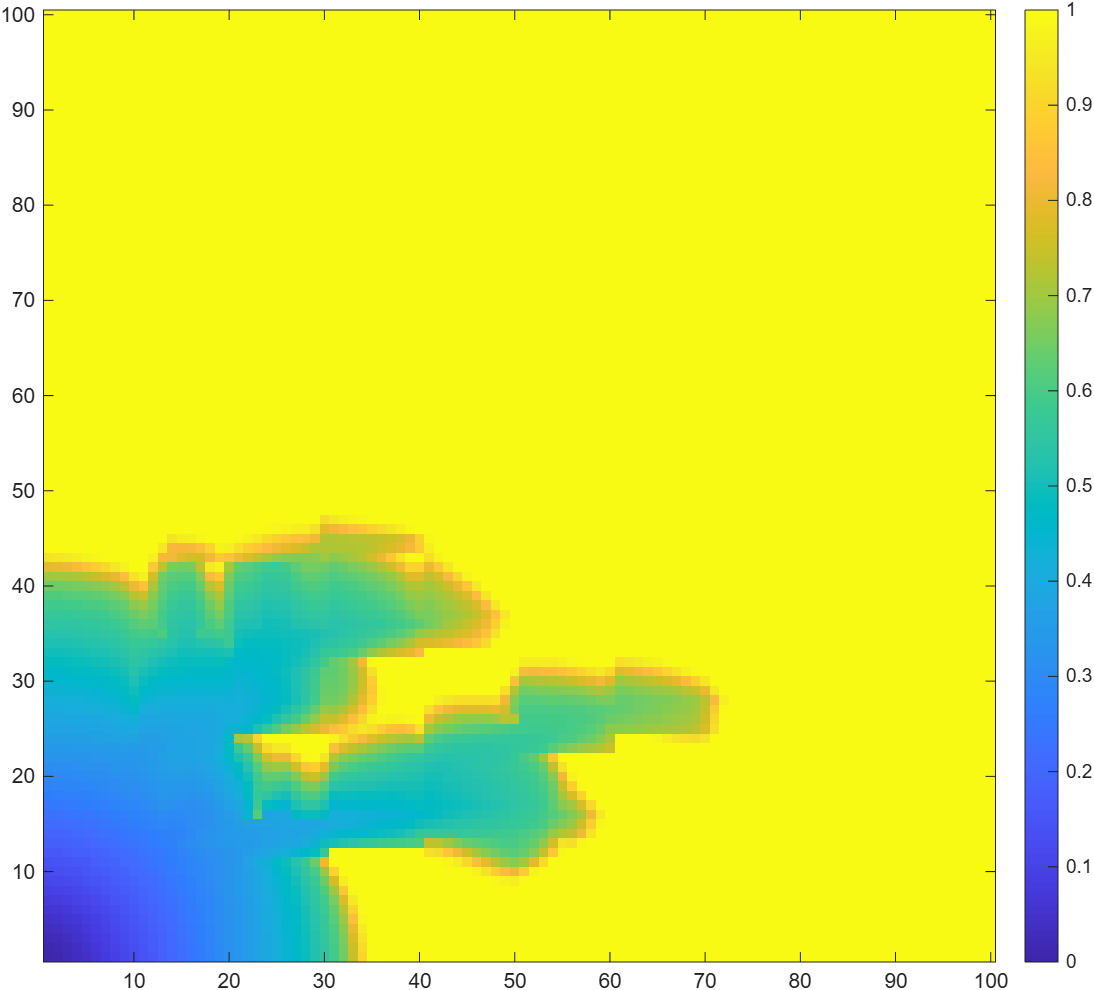}
        \caption{$S_{n,1}$ at $T=2$.}
        \label{fig:Sn1_T2}
    \end{subfigure}
    \hfill
    \begin{subfigure}{0.32\textwidth}
        \centering
        \includegraphics[width=\linewidth]{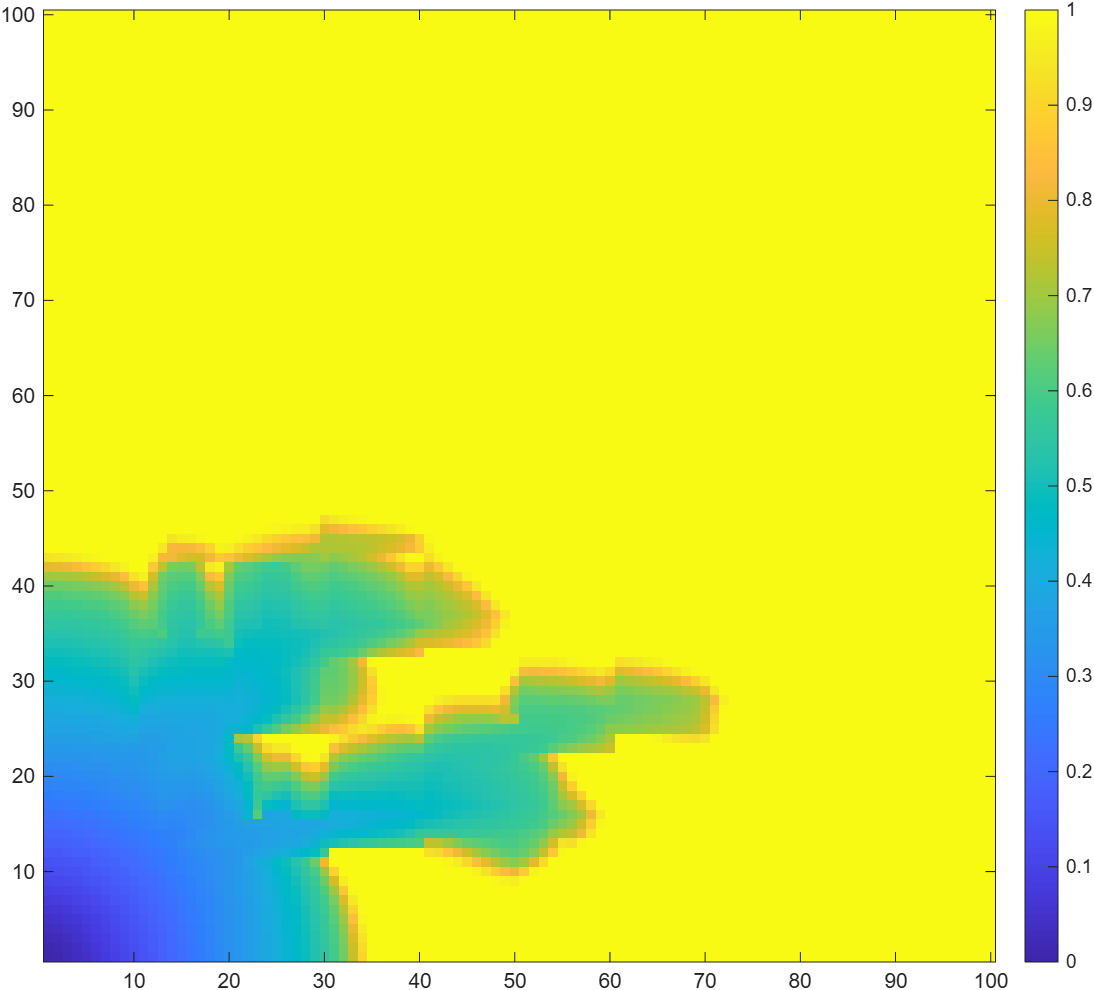}
        \caption{$S_{n,2}$ at $T=2$.}
        \label{fig:Sn2_T2}
    \end{subfigure}
    \hfill
    \begin{subfigure}{0.32\textwidth}
        \centering
        \includegraphics[width=\linewidth]{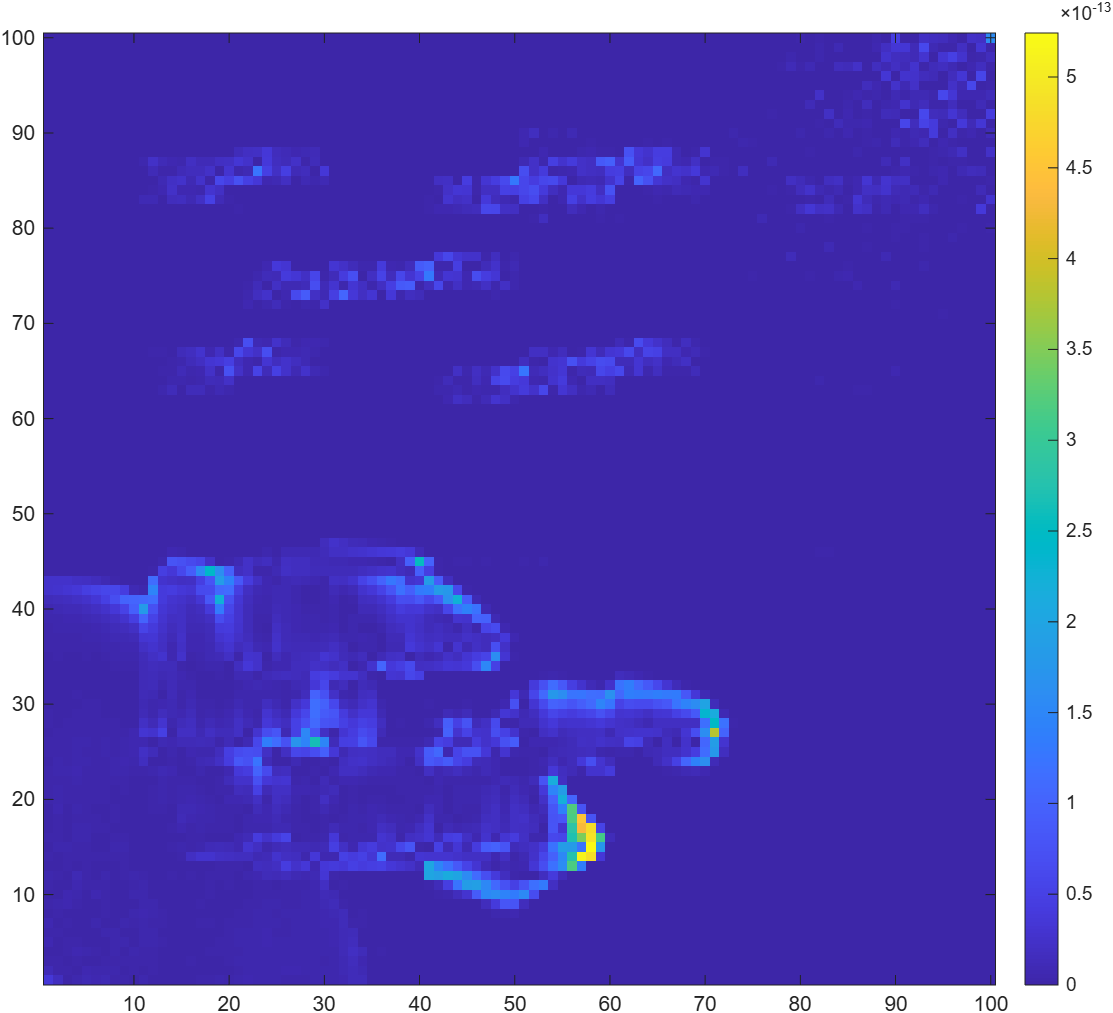}
        \caption{$|S_{n,1}-S_{n,2}|$ at $T=2$.}
        \label{fig:Sn_diff_T2}
    \end{subfigure}

    \caption*{Figure~\thefigthree. Two ways for non-wetting saturation $S_n$ at time $T=2$.}
\end{figure}

\refstepcounter{figthree}
\label{fig:Sn_T3}
\begin{figure}[H]
    \centering
    \setcounter{subfigure}{0}

    \begin{subfigure}{0.32\textwidth}
        \centering
        \includegraphics[width=\linewidth]{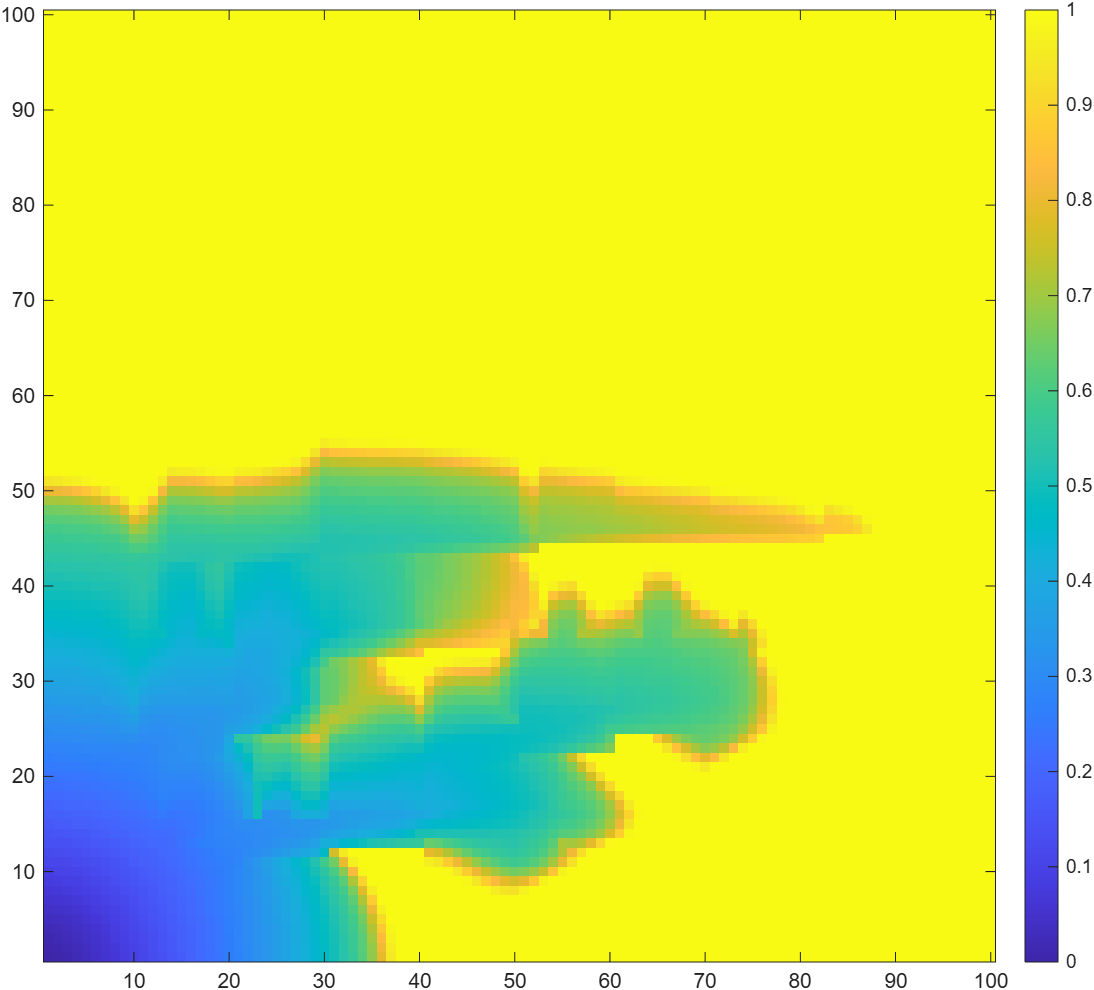}
        \caption{$S_{n,1}$ at $T=3$.}
        \label{fig:Sn1_T3}
    \end{subfigure}
    \hfill
    \begin{subfigure}{0.32\textwidth}
        \centering
        \includegraphics[width=\linewidth]{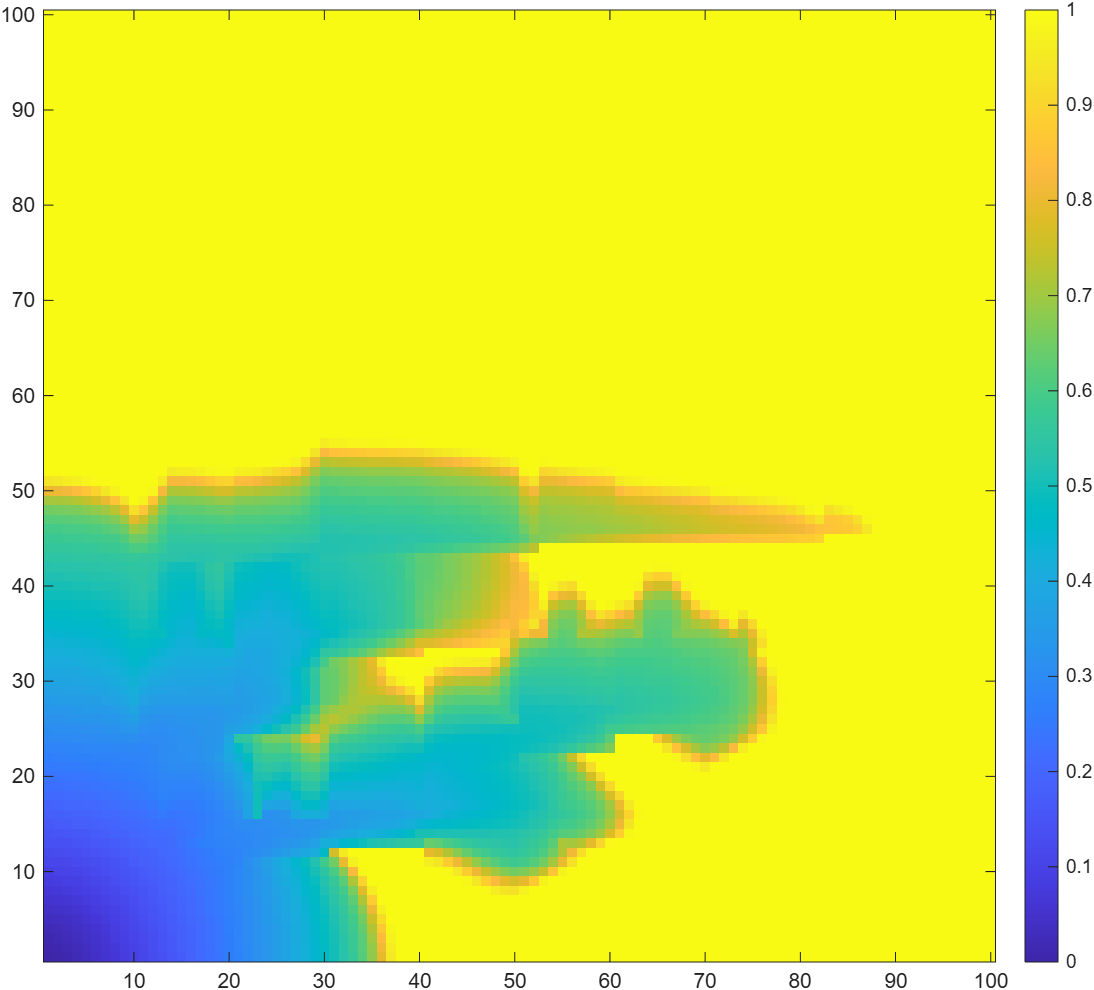}
        \caption{$S_{n,2}$ at $T=3$.}
        \label{fig:Sn2_T3}
    \end{subfigure}
    \hfill
    \begin{subfigure}{0.32\textwidth}
        \centering
        \includegraphics[width=\linewidth]{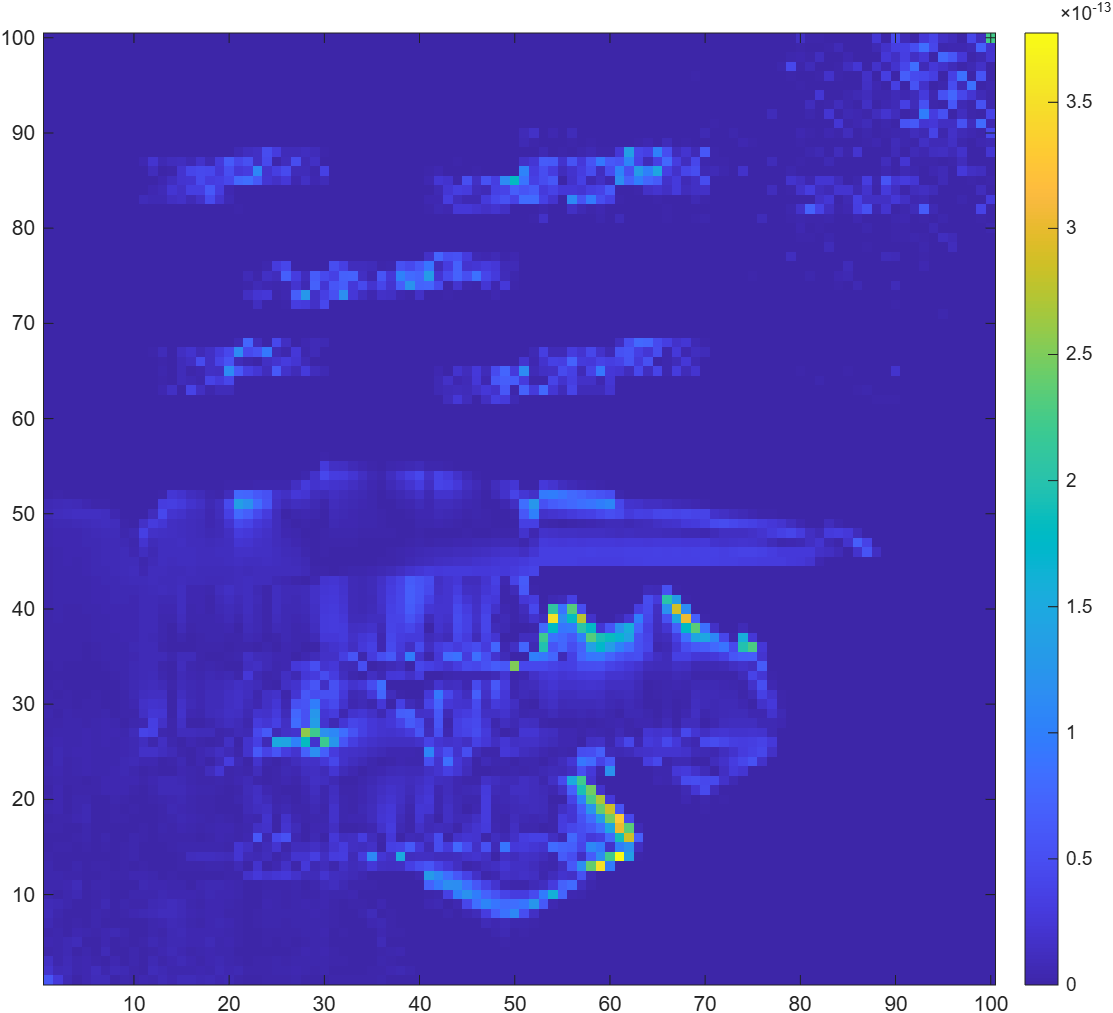}
        \caption{$|S_{n,1}-S_{n,2}|$ at $T=3$.}
        \label{fig:Sn_diff_T3}
    \end{subfigure}

    \caption*{Figure~\thefigthree. Two ways for non-wetting saturation $S_n$ at time $T=3$.}
\end{figure}

\setcounter{figure}{3}
Figure 3.1-3.3 shows the non-wetting phase saturation computed by the two different formulations at $T=1,2,3$, with their absolute differences. It can be seen that $S_{n,1}$ and $S_{n,2}$ produce nearly the same saturation distributions during the whole simulation. The difference $|S_{n,1}-S_{n,2}|$ remains at the order of about $10^{-13}$ throughout the simulation, indicating that the two formulations produce almost the same outcomes.\\

\textbf{Example 5.2. (Capillary pressure effect)} 
In this example, we investigate the influence of the capillary pressure on the two-phase flow simulation. Recall that the capillary pressure is modeled by the Hoteit-Firoozabadi formula $ p_c(S_w) = -\frac{B_c}{\sqrt{\mathbf{K}}}\log \bar{S}_w$, the parameter $B_c \ge 0$ controls the strength of the capillary effect. We test three representative cases: $B_c = 0$, $B_c = 10^{-3}$ and $B_c = 10^{-2}$ at different time $T=1,2,3$. For other parameters, we keep the same as in Example 5.1.
\setcounter{figfour}{0}

\refstepcounter{figfour}
\label{fig:capillary_Tgroup1}
\begin{figure}[H]
    \centering
    \setcounter{subfigure}{0}

    \begin{subfigure}{0.32\textwidth}
        \centering
        \includegraphics[width=\linewidth]{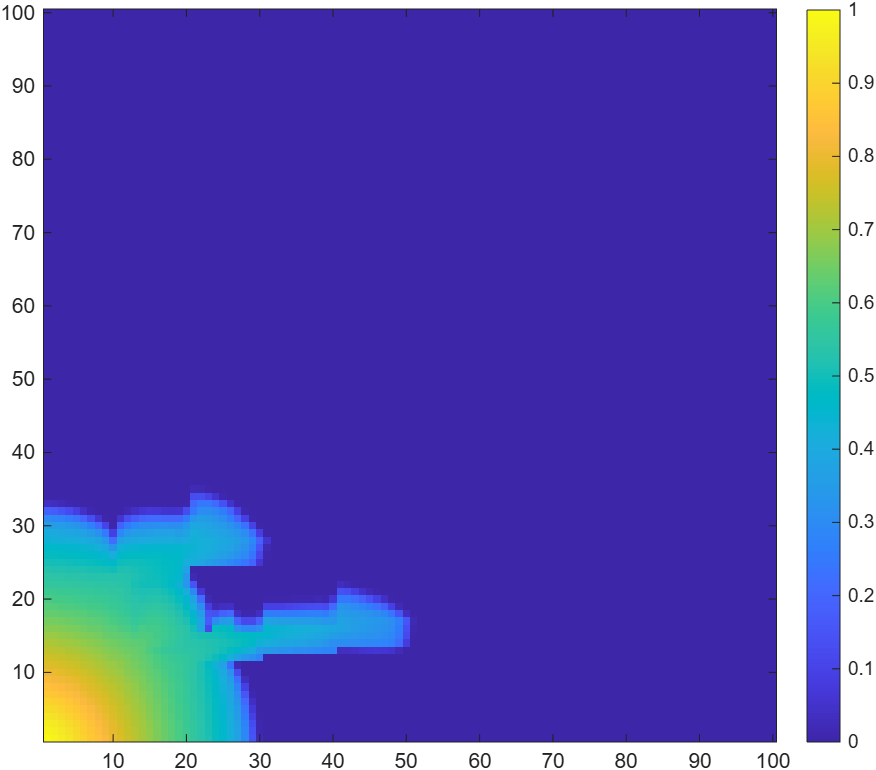}
        \caption{$B_c=0$ at $T=1$.}
        \label{fig:bc0_T1}
    \end{subfigure}
    \hfill
    \begin{subfigure}{0.32\textwidth}
        \centering
        \includegraphics[width=\linewidth]{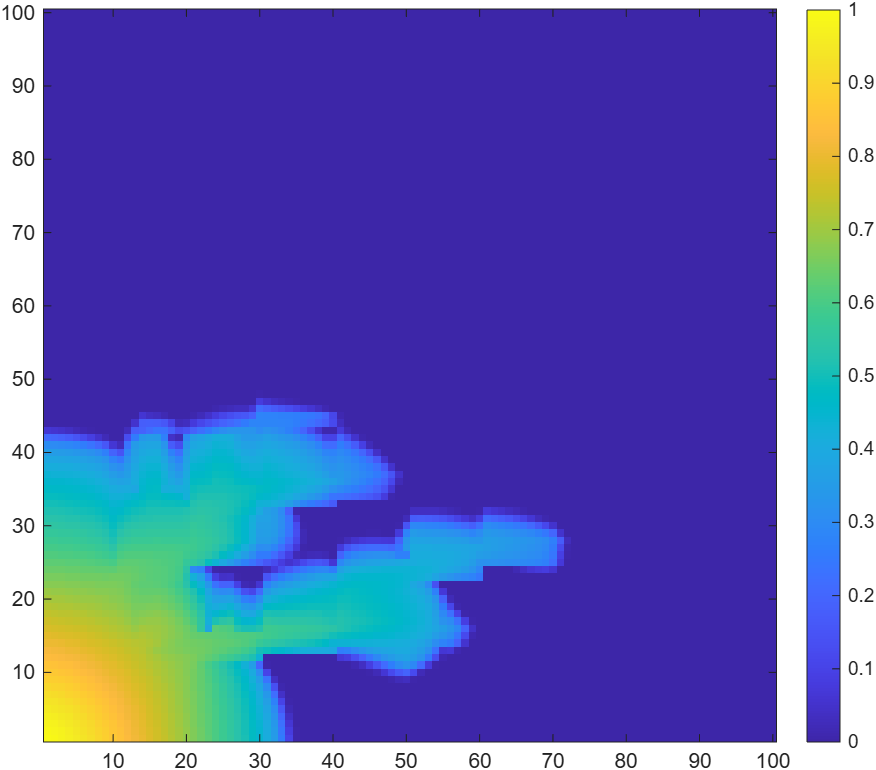}
        \caption{$B_c=0$ at $T=2$.}
        \label{fig:bc0_T2}
    \end{subfigure}
    \hfill
    \begin{subfigure}{0.32\textwidth}
        \centering
        \includegraphics[width=\linewidth]{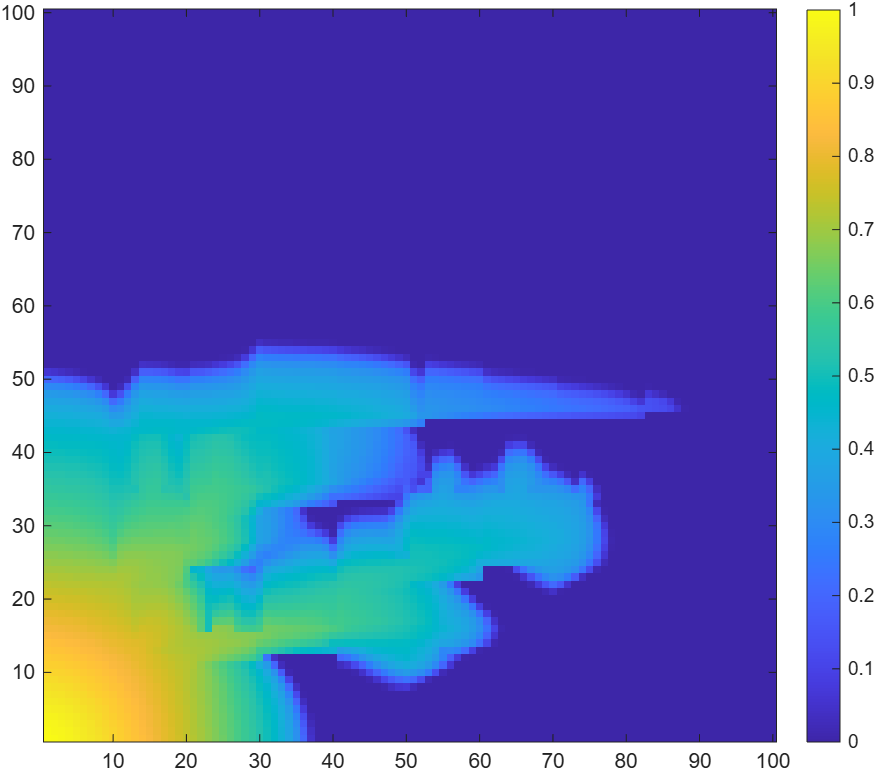}
        \caption{$B_c=0$ at $T=3$.}
        \label{fig:bc0_T3}
    \end{subfigure}

    \caption*{Figure~\thefigfour. Wetting-phase saturation $S_w$ for $B_c=0$ at $T=1,2,3$.}
\end{figure}

\refstepcounter{figfour}
\label{fig:capillary_Tgroup2}
\begin{figure}[H]
    \centering
    \setcounter{subfigure}{0}

    \begin{subfigure}{0.32\textwidth}
        \centering
        \includegraphics[width=\linewidth]{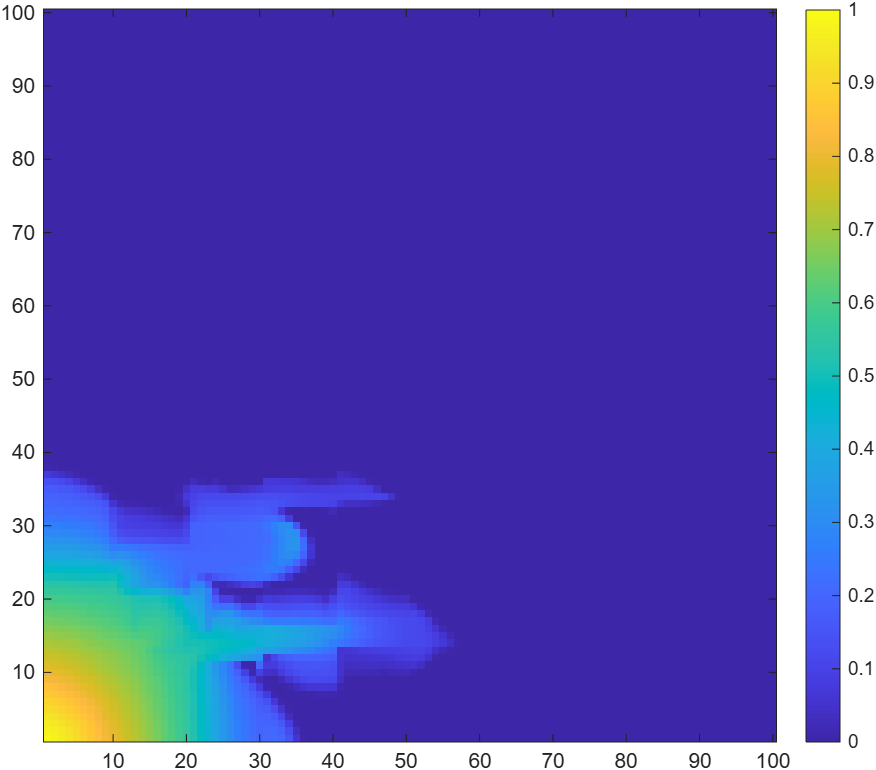}
        \caption{$B_c=10^{-3}$ at $T=1$.}
        \label{fig:bc1e-3_T1}
    \end{subfigure}
    \hfill
    \begin{subfigure}{0.32\textwidth}
        \centering
        \includegraphics[width=\linewidth]{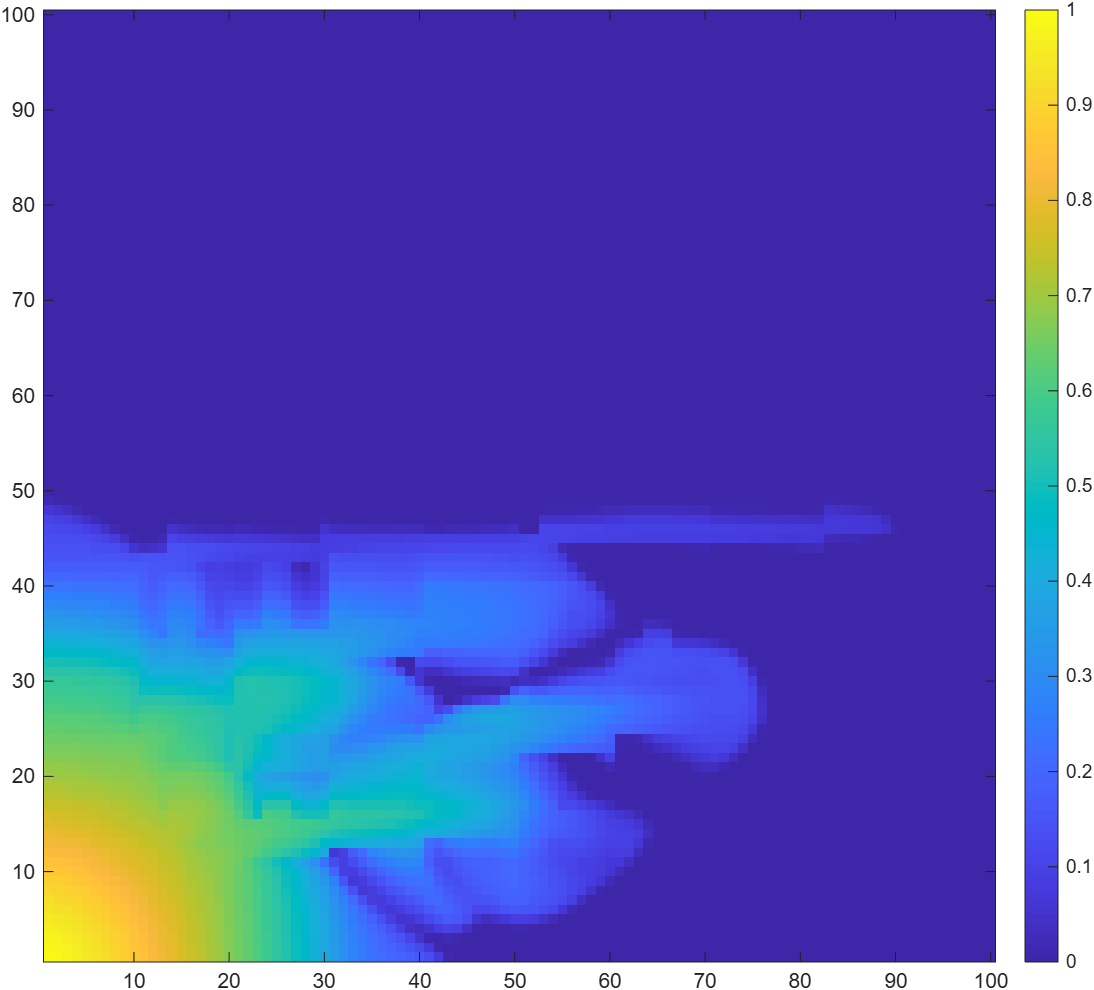}
        \caption{$B_c=10^{-3}$ at $T=2$.}
        \label{fig:bc1e-3_T2}
    \end{subfigure}
    \hfill
    \begin{subfigure}{0.32\textwidth}
        \centering
        \includegraphics[width=\linewidth]{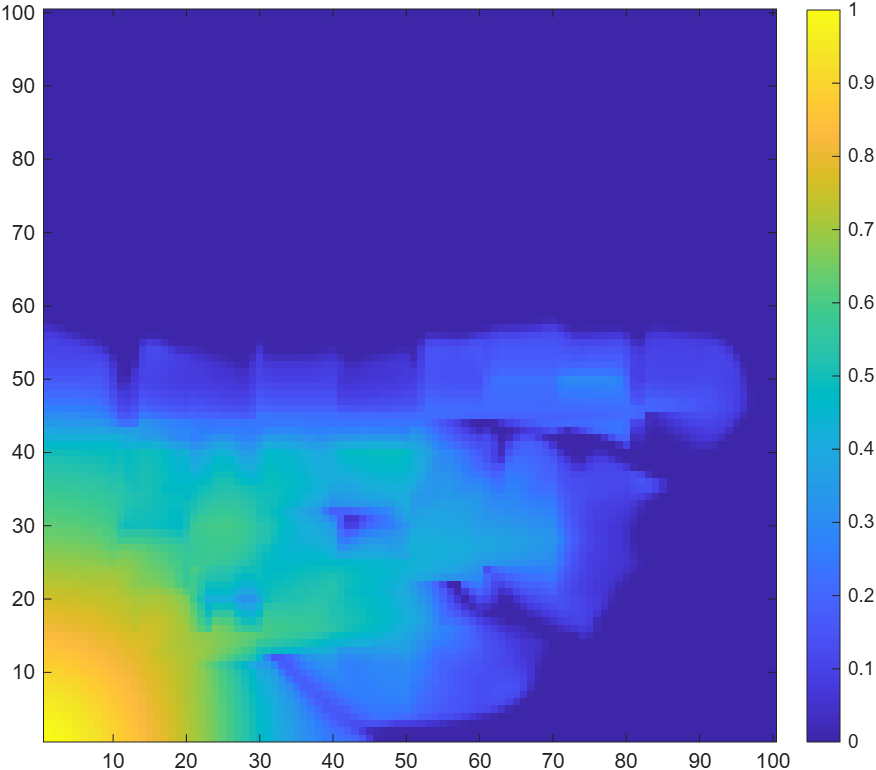}
        \caption{$B_c=10^{-3}$ at $T=3$.}
        \label{fig:bc1e-3_T3}
    \end{subfigure}

    \caption*{Figure~\thefigfour. Wetting-phase saturation $S_w$ for $B_c=10^{-3}$ at $T=1,2,3$.}
\end{figure}

\refstepcounter{figfour}
\label{fig:capillary_Tgroup3}
\begin{figure}[H]
    \centering
    \setcounter{subfigure}{0}

    \begin{subfigure}{0.32\textwidth}
        \centering
        \includegraphics[width=\linewidth]{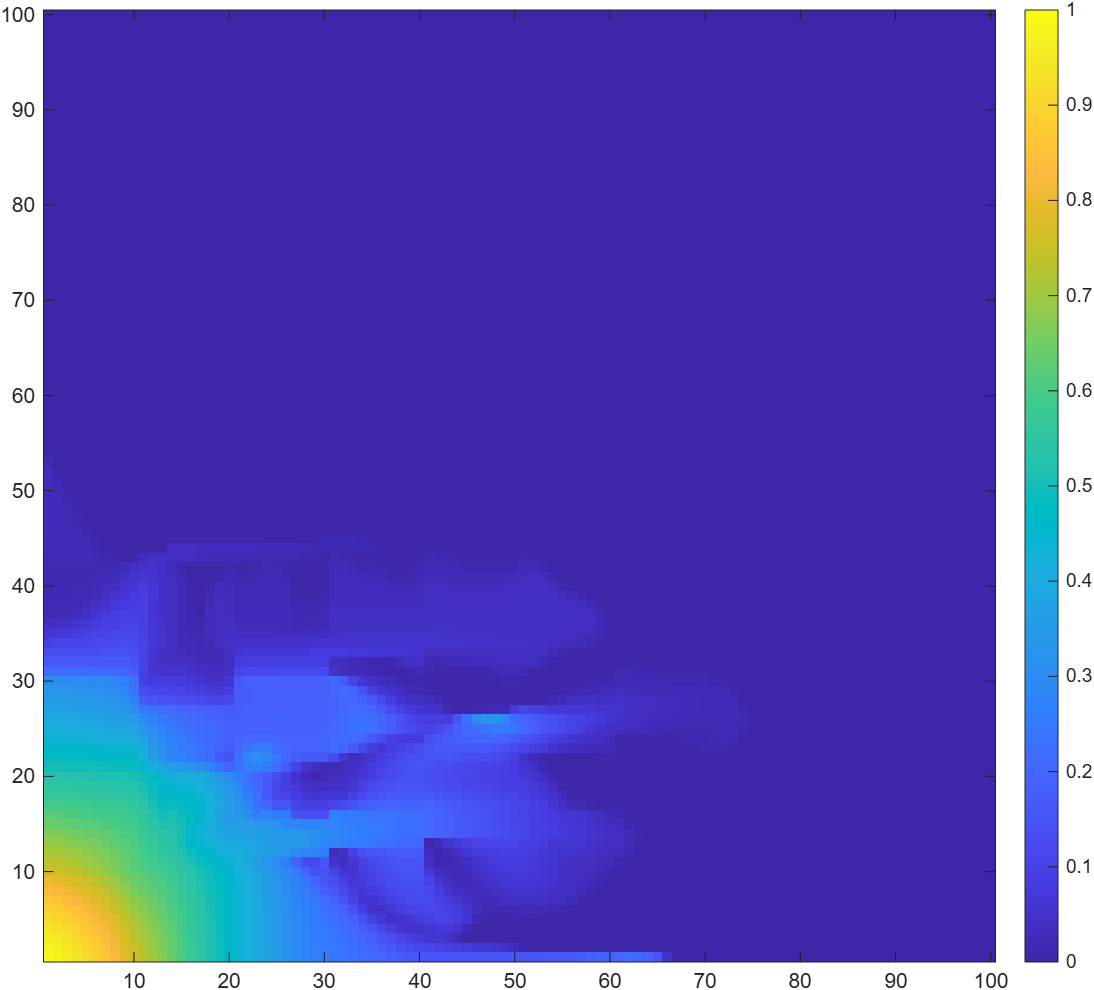}
        \caption{$B_c=10^{-2}$ at $T=1$.}
        \label{fig:bc1e-2_T1}
    \end{subfigure}
    \hfill
    \begin{subfigure}{0.32\textwidth}
        \centering
        \includegraphics[width=\linewidth]{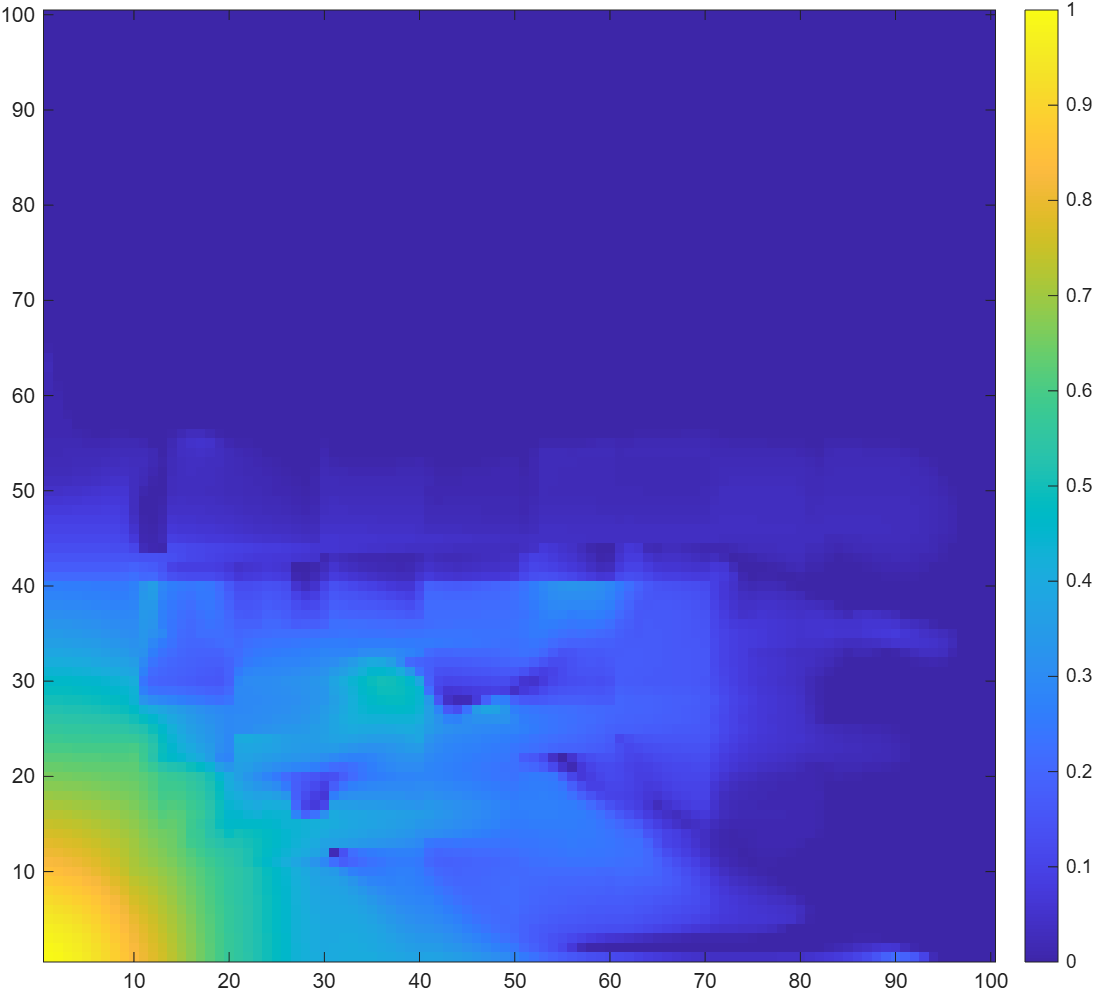}
        \caption{$B_c=10^{-2}$ at $T=2$.}
        \label{fig:bc1e-2_T2}
    \end{subfigure}
    \hfill
    \begin{subfigure}{0.32\textwidth}
        \centering
        \includegraphics[width=\linewidth]{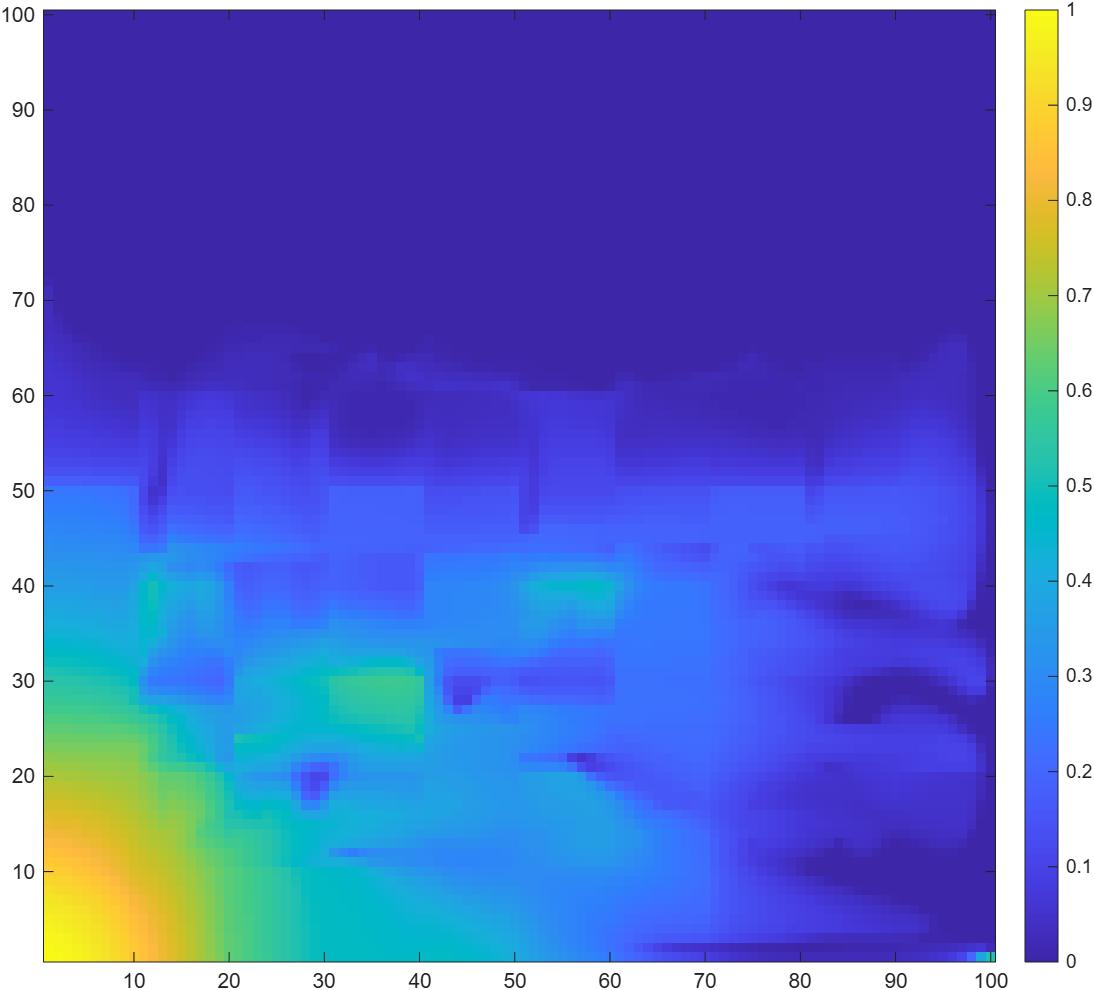}
        \caption{$B_c=10^{-2}$ at $T=3$.}
        \label{fig:bc1e-2_T3}
    \end{subfigure}

    \caption*{Figure~\thefigfour. Wetting-phase saturation $S_w$ for $B_c=10^{-2}$ at $T=1,2,3$.}
\end{figure}

\setcounter{figure}{4}

Figure 4.1-4.3 illustrate the influence of capillary pressure on the saturation evolution. When \(B_c=0\), the displacement is mainly dominated by the advection and the heterogeneous permeability field, so the wetting phase mainly moves through the high-permeability channels and the fronts remain relatively sharp. As \(B_c\) increases to \(10^{-3}\) and \(10^{-2}\), the capillary pressure introduces an additional diffusion-like effect, which smooths the saturation profile and broadens the transition region near the front. In particular, for \(B_c=10^{-2}\), the channelized displacement becomes less pronounced and the wetting phase spreads more uniformly into the surrounding region. These results show that stronger capillary pressure regularizes the saturation distribution and weakens sharp saturation discontinuities.

\textbf{Example 5.3.}
In this example, we study the influence of the adaptive update tolerance
$\varepsilon$, coarse element size $H$ and the number of multiscale basis functions on the velocity and saturation errors. We consider two representative high-contrast permeability fields. The first one is a high-permeability channel field, denoted by $\kappa_2$, the second one is taken from the top layer of the SPE10 permeability field and is denoted by $\kappa_3$. These two permeability fields are shown in the figure below. The capillary pressure parameter is set to be $B_c=0$ and fine-mesh size is still $h=0.01$. For other parameters, we keep the same as before.
\begin{figure}[H]
    \centering

    \begin{subfigure}{0.48\textwidth}
        \centering
        \includegraphics[width=\textwidth]{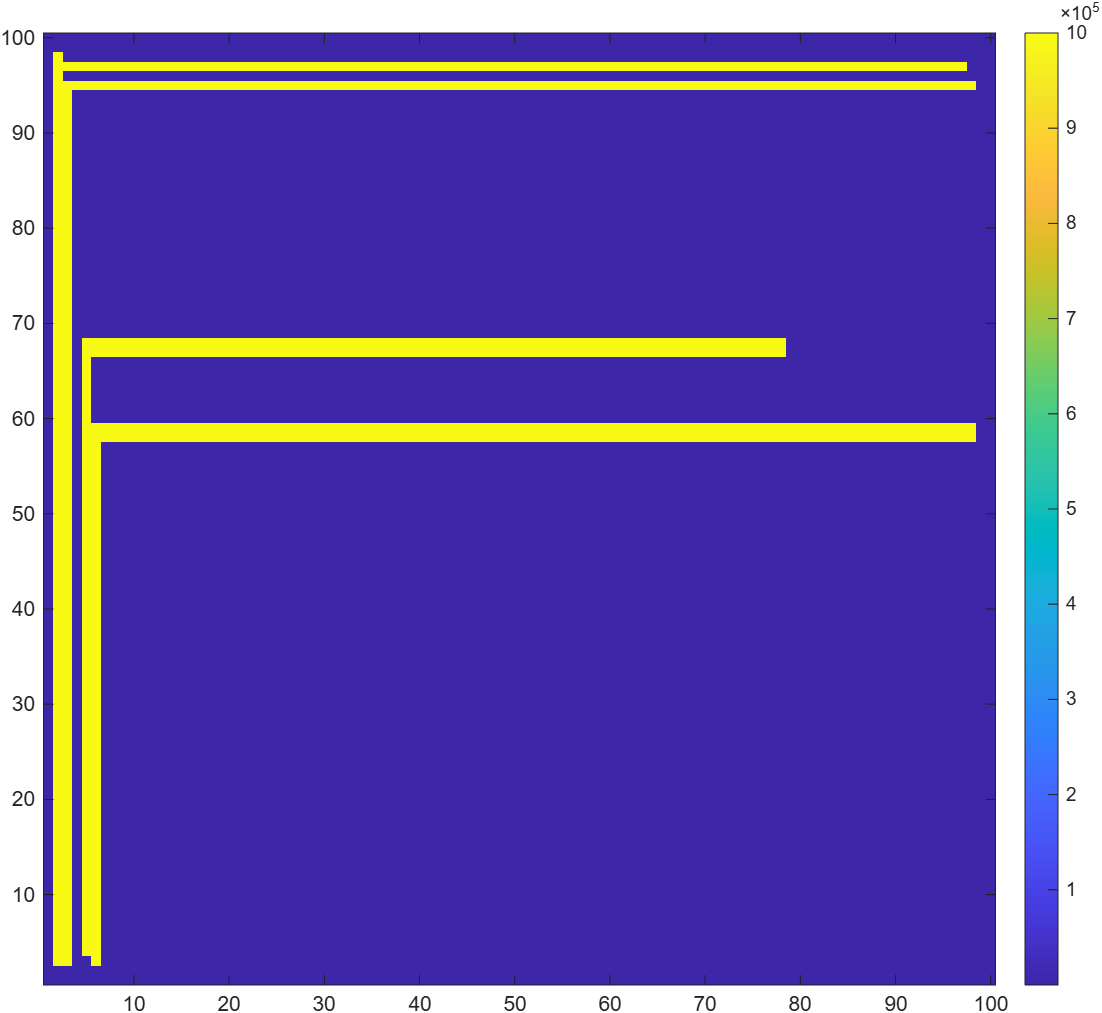}
        \caption{High-permeability channel: $\kappa_2$}
        \label{fig:initial_condition}
    \end{subfigure}
    \hfill
    \begin{subfigure}{0.48\textwidth}
        \centering
        \includegraphics[width=\textwidth]{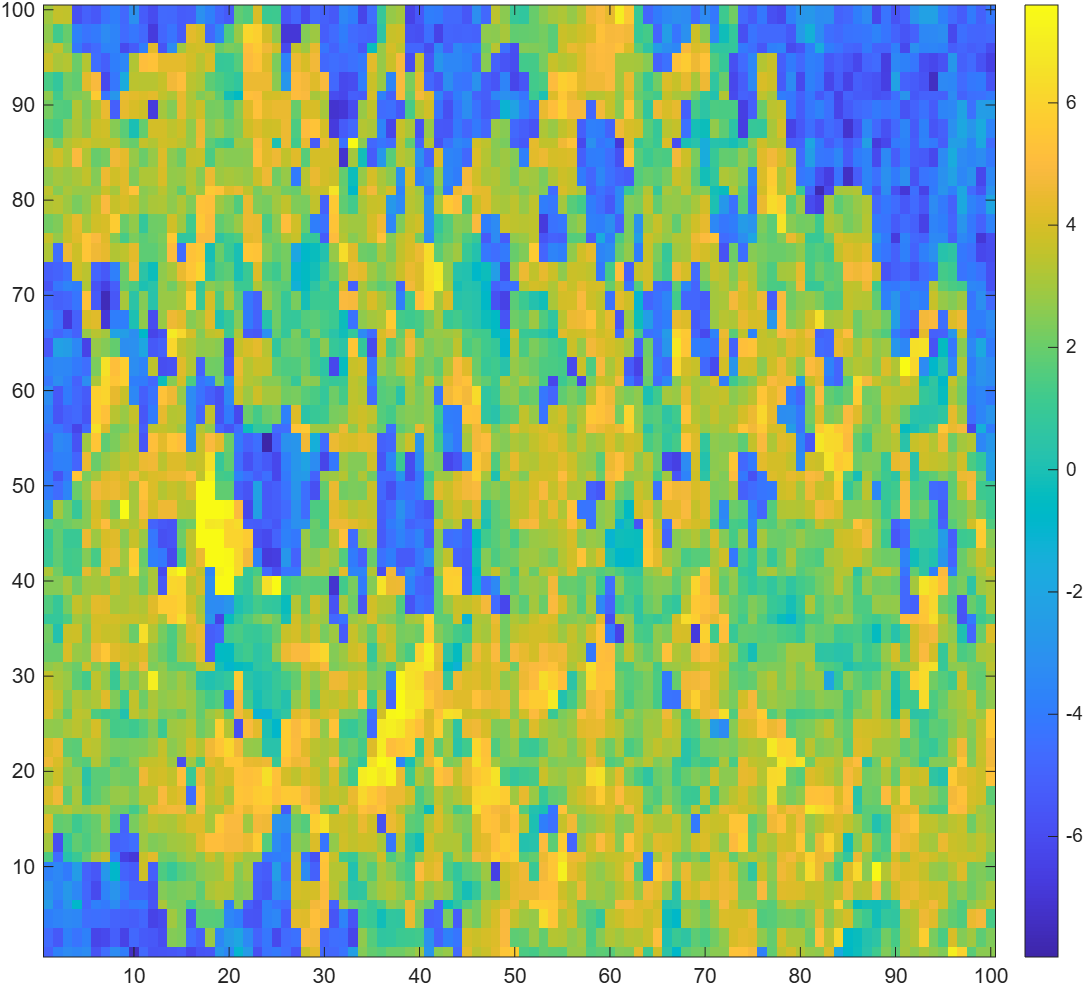}
        \caption{SPE10 model: $\log(\kappa_3)$}
        \label{fig:kappa_1}
    \end{subfigure}
\end{figure}

First, we investigate the influence of the coarse mesh size $H$ and the number of multiscale basis functions on the velocity and pressure errors at the initial time. The results are shown in Table 1 and Table 2, respectively.

\begin{table}[htbp]
\centering

\label{tab:channel_kappa2_convergence}
\small
\renewcommand{\arraystretch}{1.10}
\setlength{\tabcolsep}{4.5pt}
\begin{tabular}{c|cc|cc|cc|cc}
\hline
\multirow{3}{*}{$H$}
& \multicolumn{8}{c}{Number of basis functions per coarse element} \\
\cline{2-9}
& \multicolumn{2}{c|}{$n=1$}
& \multicolumn{2}{c|}{$n=2$}
& \multicolumn{2}{c|}{$n=3$}
& \multicolumn{2}{c}{$n=4$} \\
\cline{2-9}
& $e_u$ & $e_p$
& $e_u$ & $e_p$
& $e_u$ & $e_p$
& $e_u$ & $e_p$ \\
\hline
$0.10$ & $21.6223\%$ & $88.7208\%$ & $17.0759\%$ & $41.5622\%$ & $8.7309\%$ & $9.9324\%$ & $8.2117\%$ & $9.6430\%$ \\
$0.05$ & $29.1088\%$ & $38.6599\%$ & $13.1360\%$ & $16.3047\%$ & $2.6884\%$ & $12.5234\%$ & $1.9773\%$ & $2.8974\%$ \\
$0.04$ & $38.4517\%$ & $51.2609\%$ & $2.0160\%$ & $13.9070\%$ & $1.8241\%$ & $10.1284\%$ & $1.6149\%$ & $3.5814\%$ \\
$0.02$ & $0.0838\%$ & $21.1459\%$ & $0.2611\%$ & $14.0518\%$ & $0.0838\%$ & $4.4090\%$ & $0.0000\%$ & $0.0000\%$ \\
\hline
\end{tabular}
\caption{Convergence of the multiscale solution for $\kappa_2$.}
\end{table}

\begin{table}[H]
\centering

\label{tab:channel_kappa2_convergence}
\small
\renewcommand{\arraystretch}{1.10}
\setlength{\tabcolsep}{4.5pt}
\begin{tabular}{c|cc|cc|cc|cc}
\hline
\multirow{3}{*}{$H$}
& \multicolumn{8}{c}{Number of basis functions per coarse element} \\
\cline{2-9}
&\multicolumn{2}{c|}{$n=1$}
& \multicolumn{2}{c|}{$n=2$}
& \multicolumn{2}{c|}{$n=3$}
& \multicolumn{2}{c}{$n=4$} \\
\cline{2-9}
& $e_u$ & $e_p$
& $e_u$ & $e_p$
& $e_u$ & $e_p$
& $e_u$ & $e_p$ \\
\hline
$0.10$ & $0.3653\%$ & $21.0482\%$ & $0.2657\%$ & $17.2137\%$ & $0.2300\%$ & $11.4032\%$ & $0.2208\%$ & $7.0975\%$ \\
$0.05$  & $0.1002\%$ & $16.1050\%$ & $0.2788\%$ & $12.4962\%$ & $0.0769\%$ & $7.0368\%$  & $0.0846\%$ & $4.9290\%$ \\
$0.04$  & $0.0886\%$ & $15.0004\%$ & $0.2052\%$ & $9.9129\%$  & $0.0761\%$ & $6.0975\%$  & $0.0604\%$ & $4.6402\%$ \\
$0.02$  & $0.0324\%$ & $8.5261\%$  & $0.0515\%$ & $6.8928\%$  & $0.0231\%$ & $2.2568\%$  & $0.0153\%$ & $0.0000\%$ \\
\hline
\end{tabular}
\caption{Convergence of the multiscale solution for $\kappa_3$.}
\end{table}

Next, we investigate the evolution of the saturation error for the two
permeability fields $\kappa_2$ and $\kappa_3$. In particular, we study how
the saturation error $e_S$ depends on the adaptive update tolerance
$\varepsilon$. To balance computational time and accuracy, we select the parameters $H$ and $n$ according to the Table 1 and Table 2, where the corresponding velocity error $e_u$ is sufficiently small. These results are performed at different final times and for different permeability fields.

\begin{figure}[htbp]
    \centering

    \begin{subfigure}{0.32\textwidth}
        \centering
        \includegraphics[width=\linewidth]{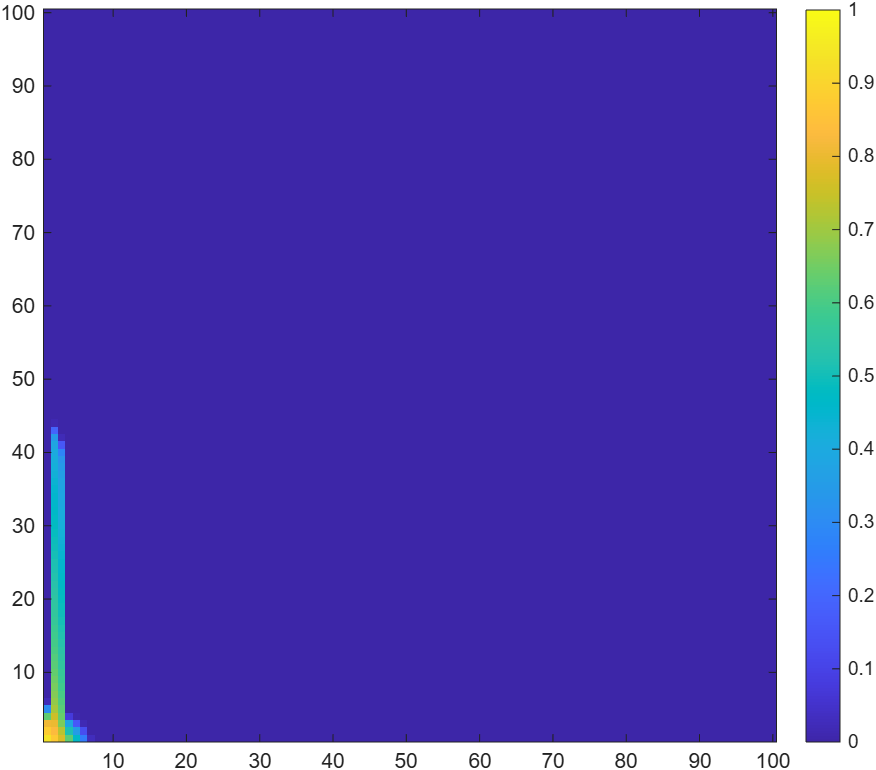}
        \caption{ $T=0.1$}
        \label{fig:ref_kappa2_T01}
    \end{subfigure}
    \hfill
    \begin{subfigure}{0.32\textwidth}
        \centering
        \includegraphics[width=\linewidth]{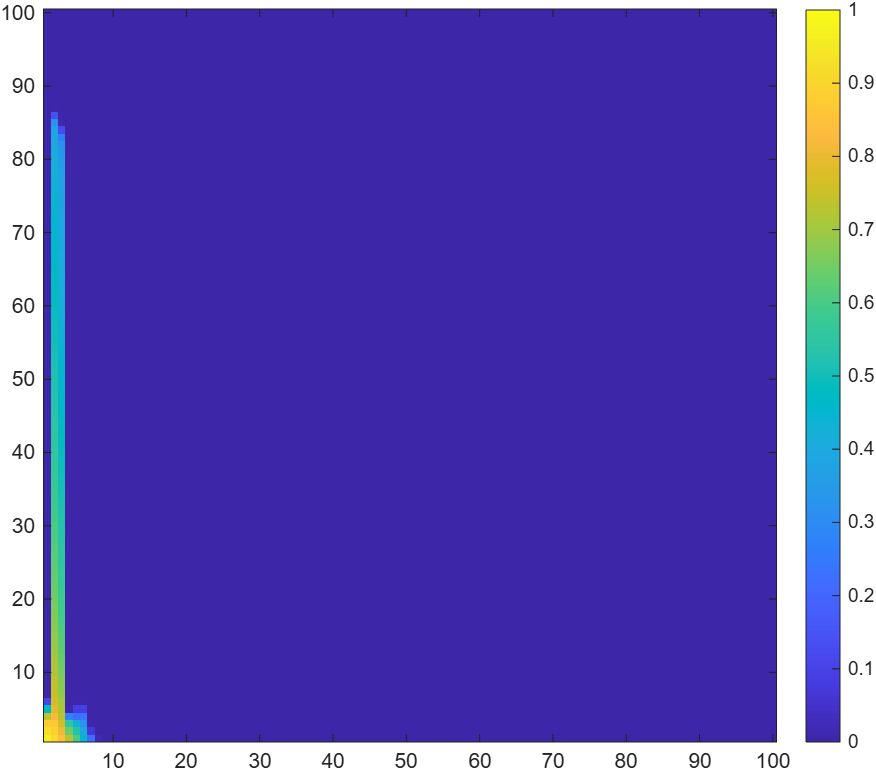}
        \caption{ $T=0.2$}
        \label{fig:ref_kappa2_T02}
    \end{subfigure}
    \hfill
    \begin{subfigure}{0.32\textwidth}
        \centering
        \includegraphics[width=\linewidth]{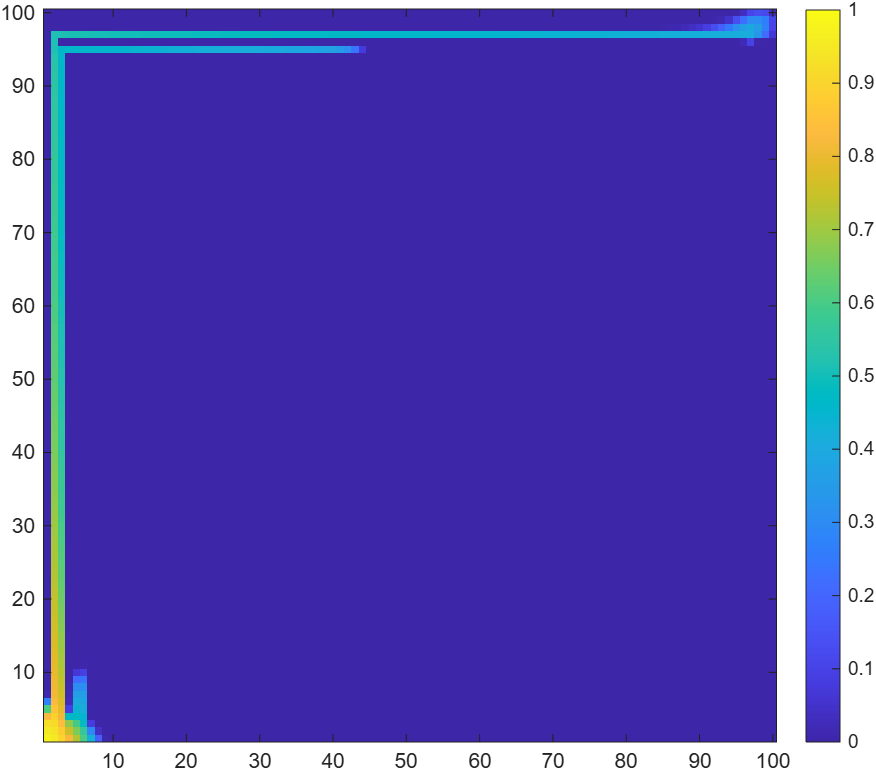}
        \caption{$T=0.4$}
        \label{fig:ref_kappa2_T03}
    \end{subfigure}

    \caption{Reference solutions of the wetting-phase saturation
    $S_w^h$ in the channel field $\kappa_2$ at different
    final times.}
    \label{fig:reference_solution_kappa2}
\end{figure}

\setcounter{figseven}{0}

\refstepcounter{figseven}
\begin{figure}[H]
    \centering

    \begin{subfigure}{0.32\textwidth}
        \centering
        \includegraphics[width=\linewidth]{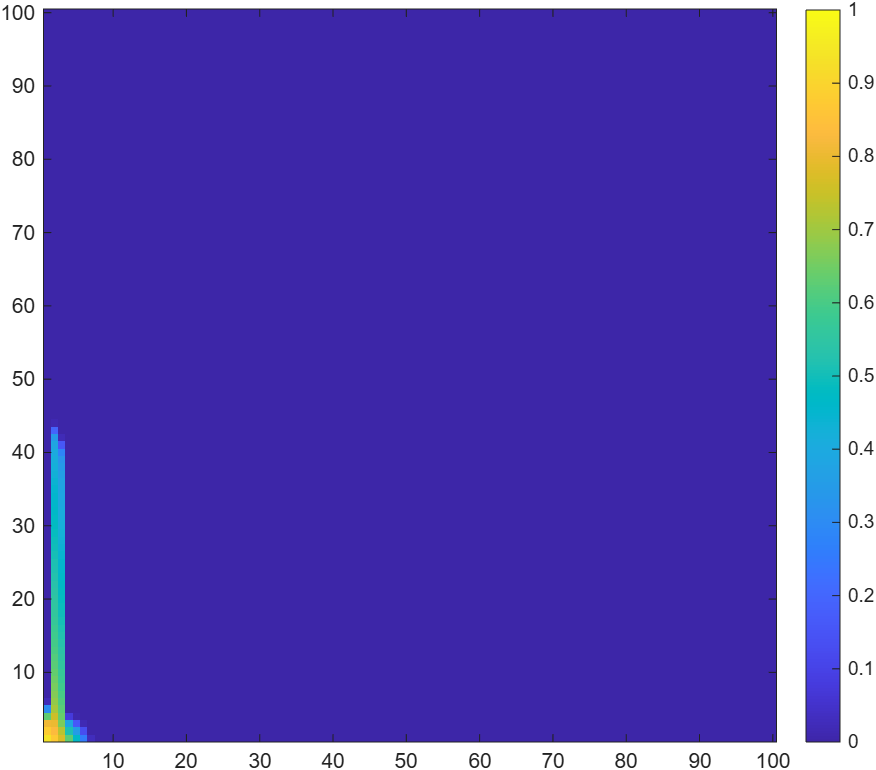}
        \caption{$\varepsilon=0.05$, $e_S=13.09\%$}
    \end{subfigure}
    \hfill
    \begin{subfigure}{0.32\textwidth}
        \centering
        \includegraphics[width=\linewidth]{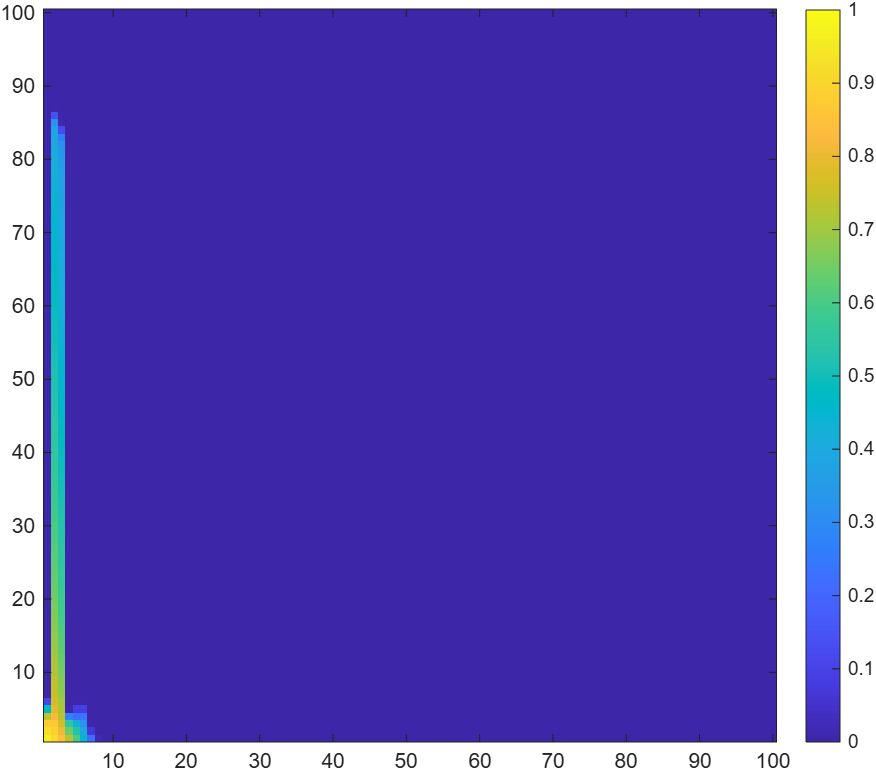}
        \caption{$\varepsilon=0.05$, $e_S=13.67\%$}
    \end{subfigure}
    \hfill
    \begin{subfigure}{0.32\textwidth}
        \centering
        \includegraphics[width=\linewidth]{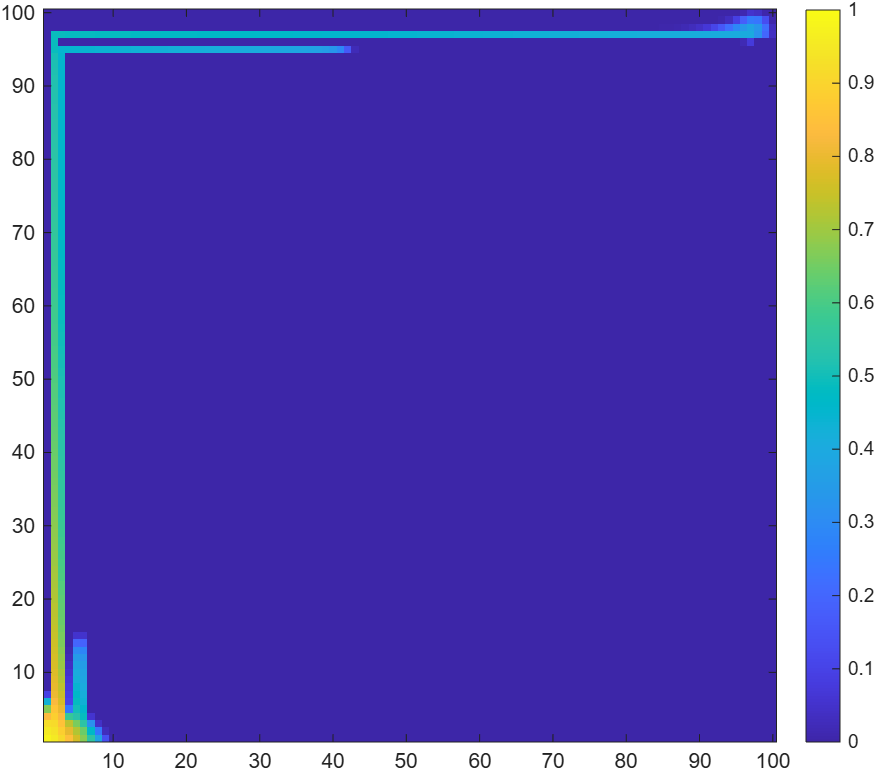}
        \caption{$\varepsilon=0.05$, $e_S=12.04\%$}
    \end{subfigure}

    \caption*{Figure~\thefigseven. Multiscale solutions with $H=0.02$ and one basis function per coarse element for $\varepsilon=0.05$ at the corresponding $T=0.1,0.2,0.4$.}
    \label{fig:7_1}
\end{figure}

\vspace{-3mm}

\refstepcounter{figseven}
\begin{figure}[H]
    \centering

    \begin{subfigure}{0.32\textwidth}
        \centering
        \includegraphics[width=\linewidth]{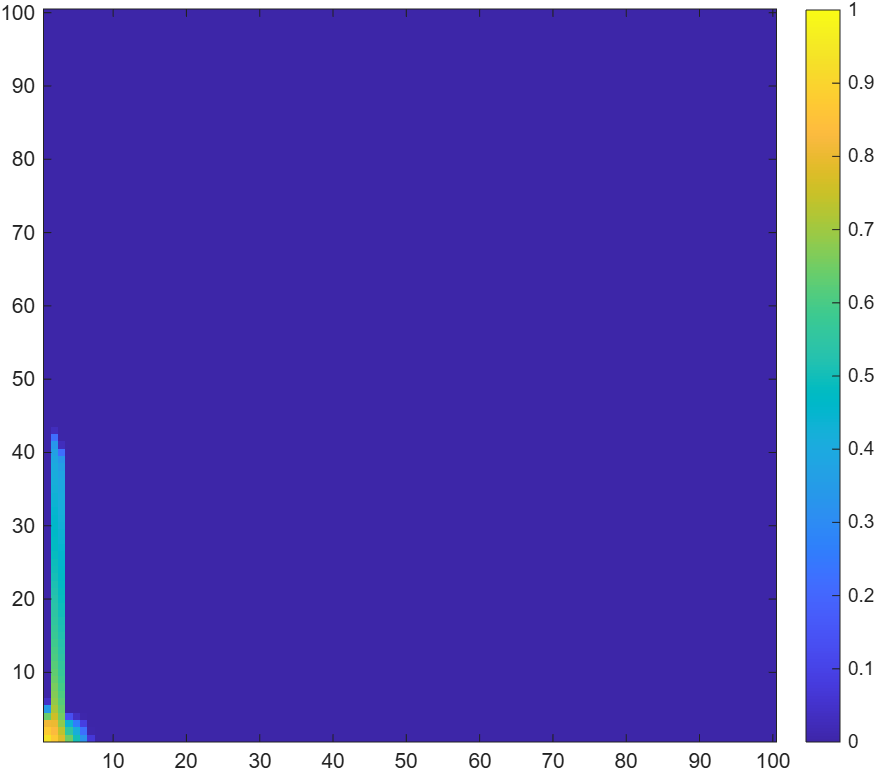}
        \caption{$\varepsilon=0.025$, $e_S=6.64\%$}
    \end{subfigure}
    \hfill
    \begin{subfigure}{0.32\textwidth}
        \centering
        \includegraphics[width=\linewidth]{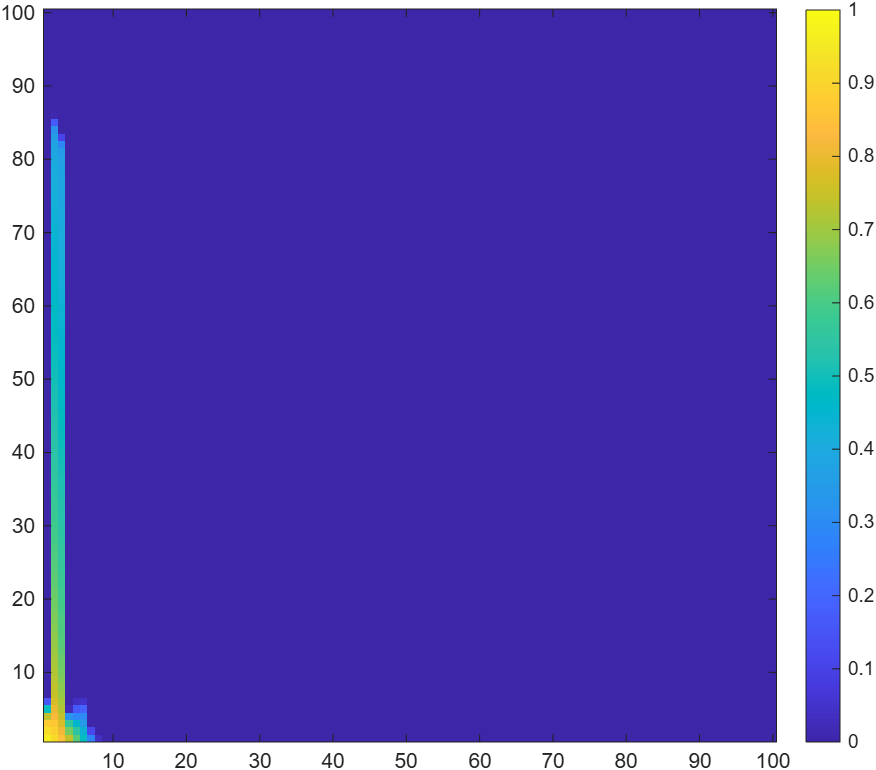}
        \caption{$\varepsilon=0.025$, $e_S=5.12\%$}
    \end{subfigure}
    \hfill
    \begin{subfigure}{0.32\textwidth}
        \centering
        \includegraphics[width=\linewidth]{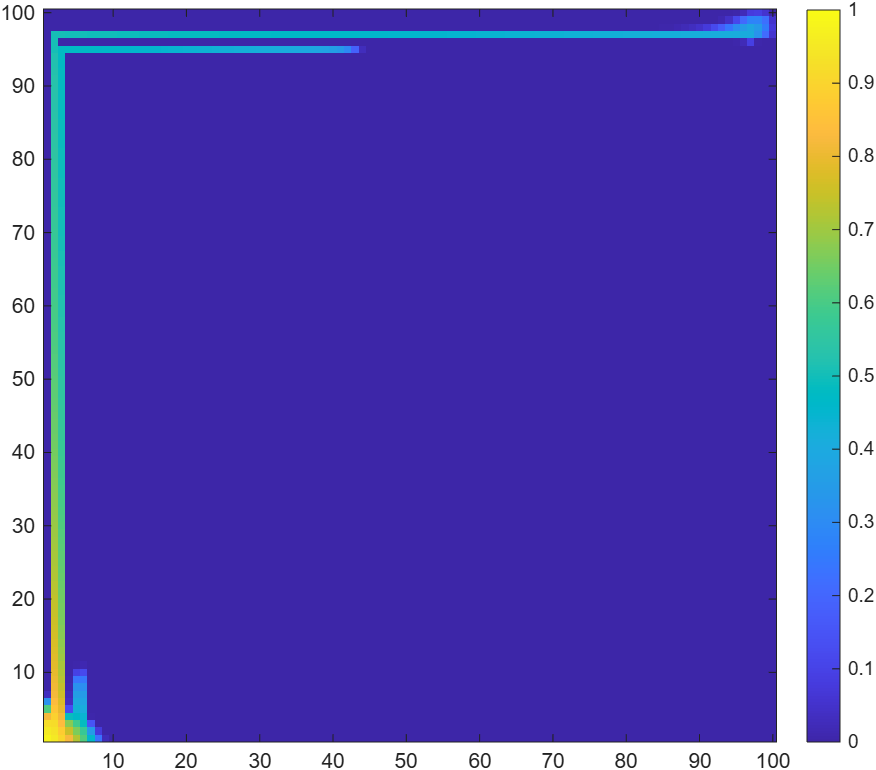}
        \caption{$\varepsilon=0.025$, $e_S=2.67\%$}
    \end{subfigure}

    \caption*{Figure~\thefigseven. Multiscale solutions with $H=0.02$ and one basis function per coarse element for $\varepsilon=0.025$ at the corresponding $T=0.1,0.2,0.4$.}
    \label{fig:7_2}
\end{figure}

\vspace{-3mm}

\refstepcounter{figseven}
\begin{figure}[H]
    \centering

    \begin{subfigure}{0.32\textwidth}
        \centering
        \includegraphics[width=\linewidth]{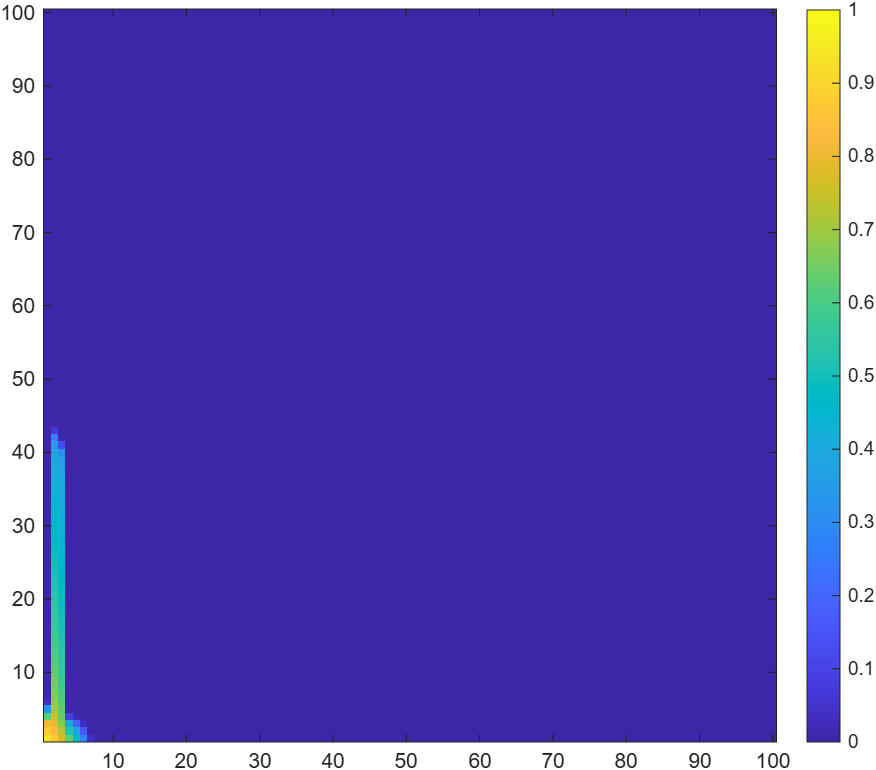}
        \caption{$\varepsilon=0.01$, $e_S=3.69\%$}
    \end{subfigure}
    \hfill
    \begin{subfigure}{0.32\textwidth}
        \centering
        \includegraphics[width=\linewidth]{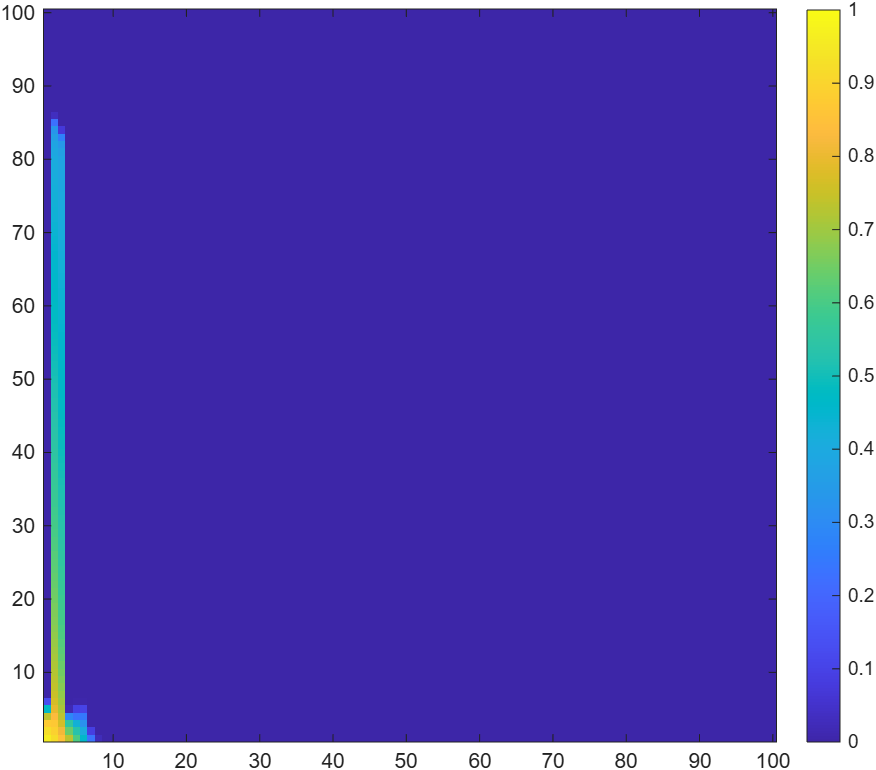}
        \caption{$\varepsilon=0.01$, $e_S=2.88\%$}
    \end{subfigure}
    \hfill
    \begin{subfigure}{0.32\textwidth}
        \centering
        \includegraphics[width=\linewidth]{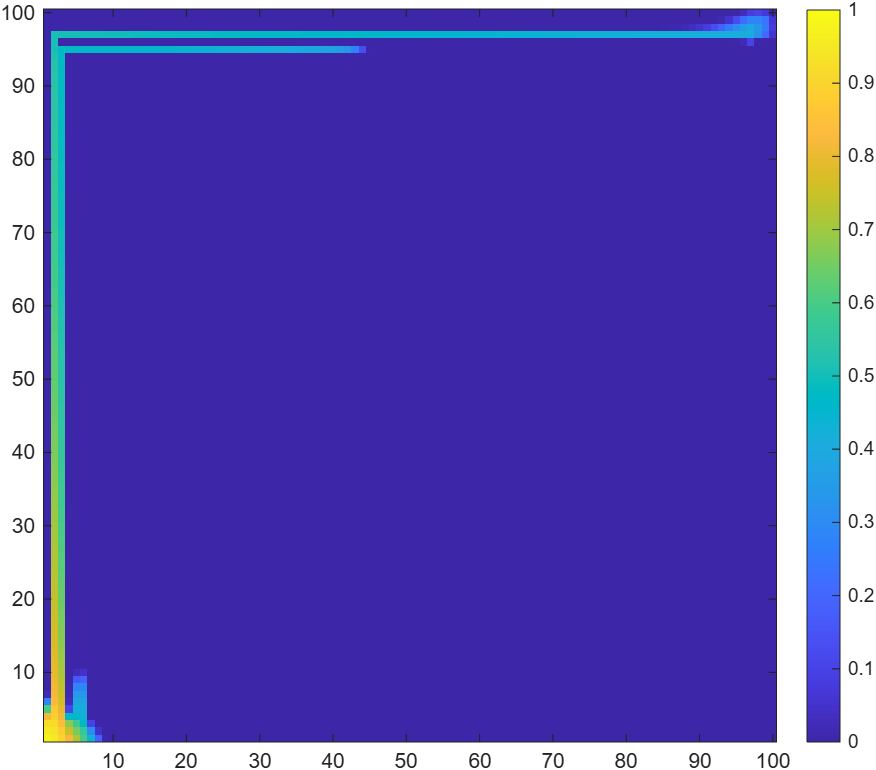}
        \caption{$\varepsilon=0.01$, $e_S=1.95\%$}
    \end{subfigure}

    \caption*{Figure~\thefigseven. Multiscale solutions with $H=0.02$ and one basis function per coarse element for $\varepsilon=0.01$ at the corresponding $T=0.1,0.2,0.4$.}
    \label{fig:7_3}
\end{figure}

Next, we use the SPE10 permeability field $\kappa_3$ to test the influence of different adaptive tolerances $\varepsilon$ at different final times.
\setcounter{figure}{7}
\begin{figure}[htbp]
    \centering

    \begin{subfigure}{0.32\textwidth}
        \centering
        \includegraphics[width=\linewidth]{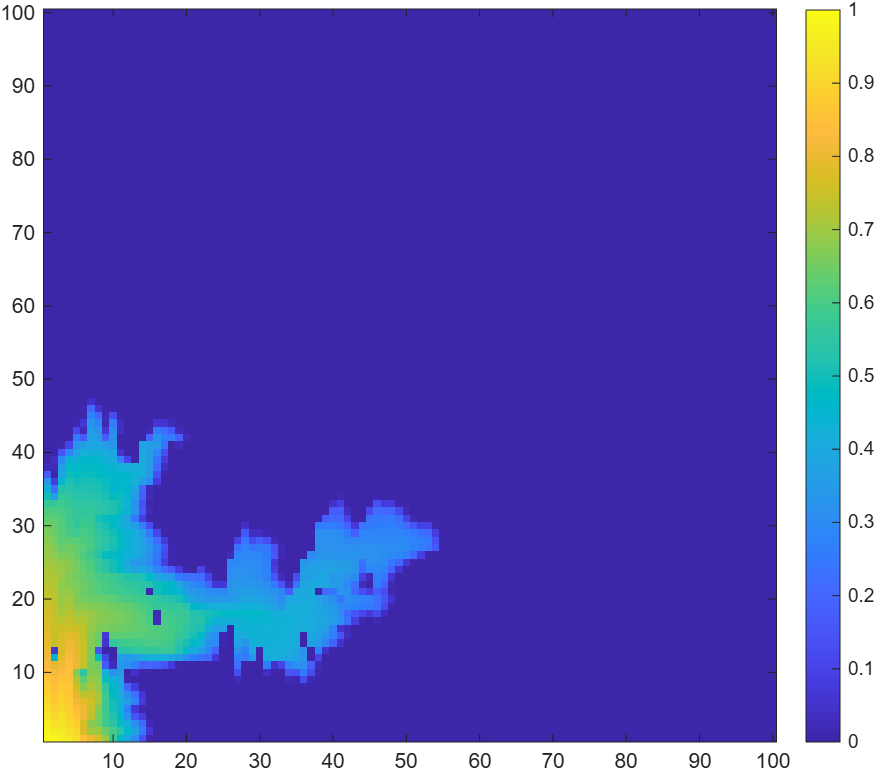}
        \caption{ $T=1$}
    
    \end{subfigure}
    \hfill
    \begin{subfigure}{0.32\textwidth}
        \centering
        \includegraphics[width=\linewidth]{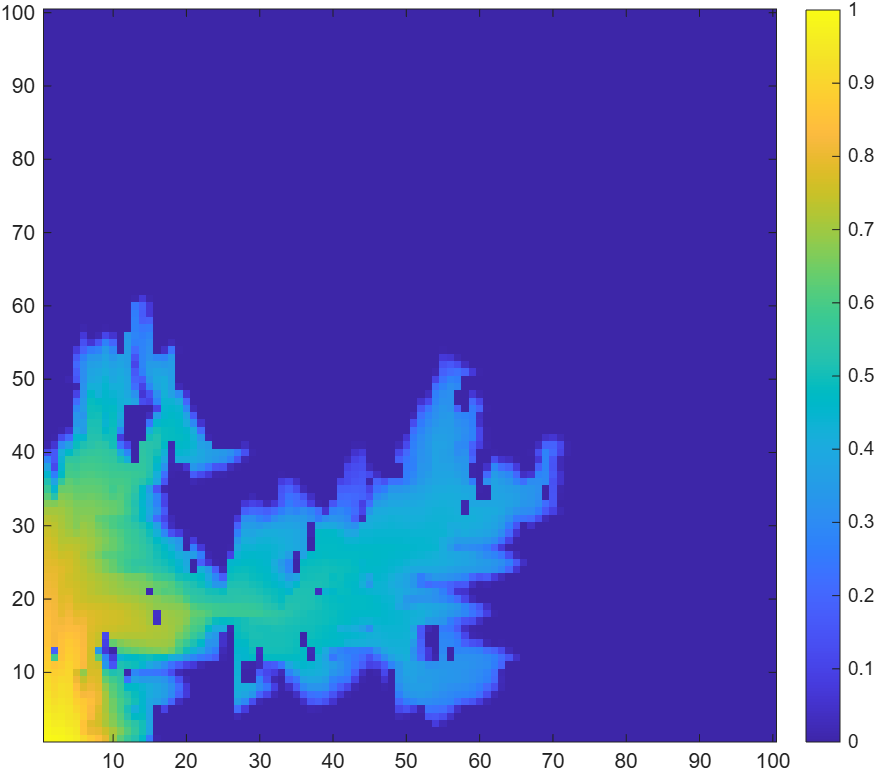}
        \caption{ $T=2$}
    
    \end{subfigure}
    \hfill
    \begin{subfigure}{0.32\textwidth}
        \centering
        \includegraphics[width=\linewidth]{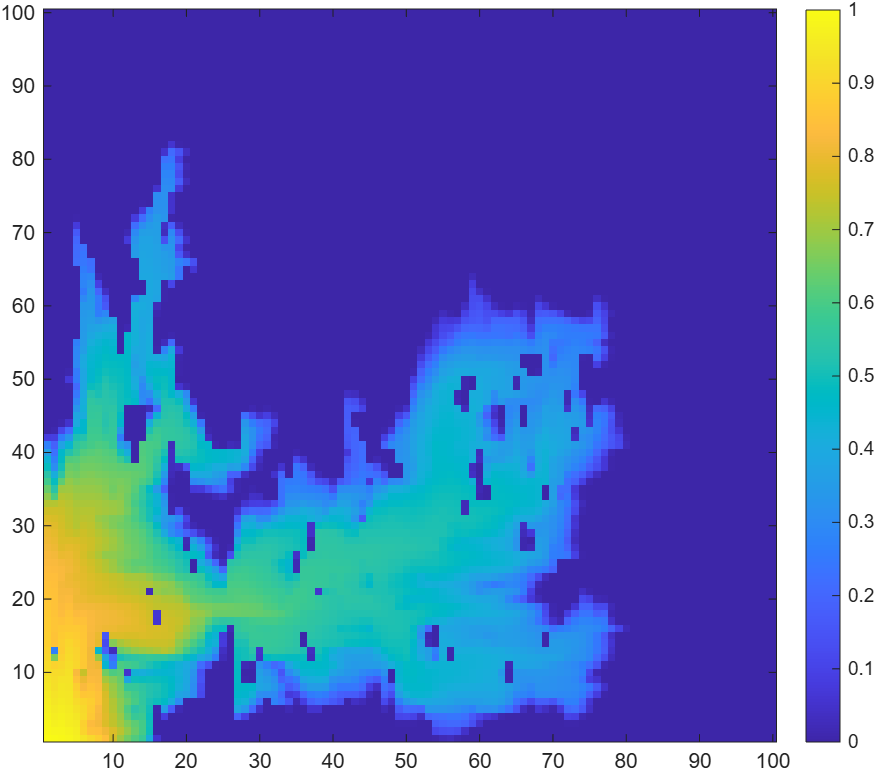}
        \caption{$T=3$}

    \end{subfigure}

    \caption{Reference solutions of the wetting-phase saturation
    $S_w^h$ in the SPE10 field $\kappa_3$ at different
    final times.}
\end{figure}

\refstepcounter{fignine}
\begin{figure}[H]
    \centering

    \begin{subfigure}{0.32\textwidth}
        \centering
        \includegraphics[width=\linewidth]{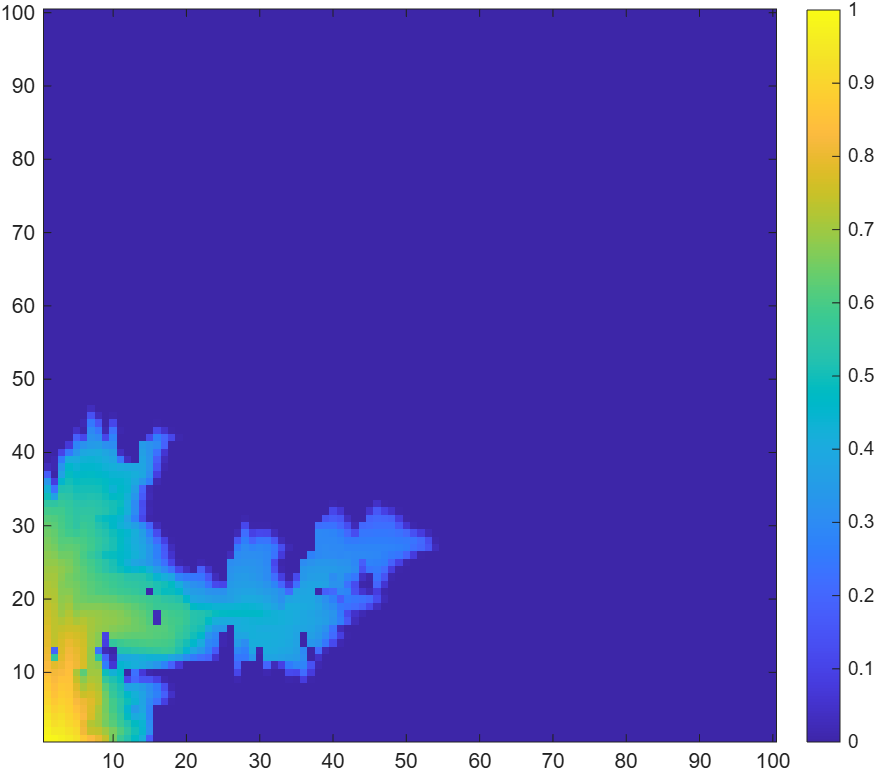}
        \caption{$\varepsilon=1$, $e_S=12.89\%$}
    \end{subfigure}
    \hfill
    \begin{subfigure}{0.32\textwidth}
        \centering
        \includegraphics[width=\linewidth]{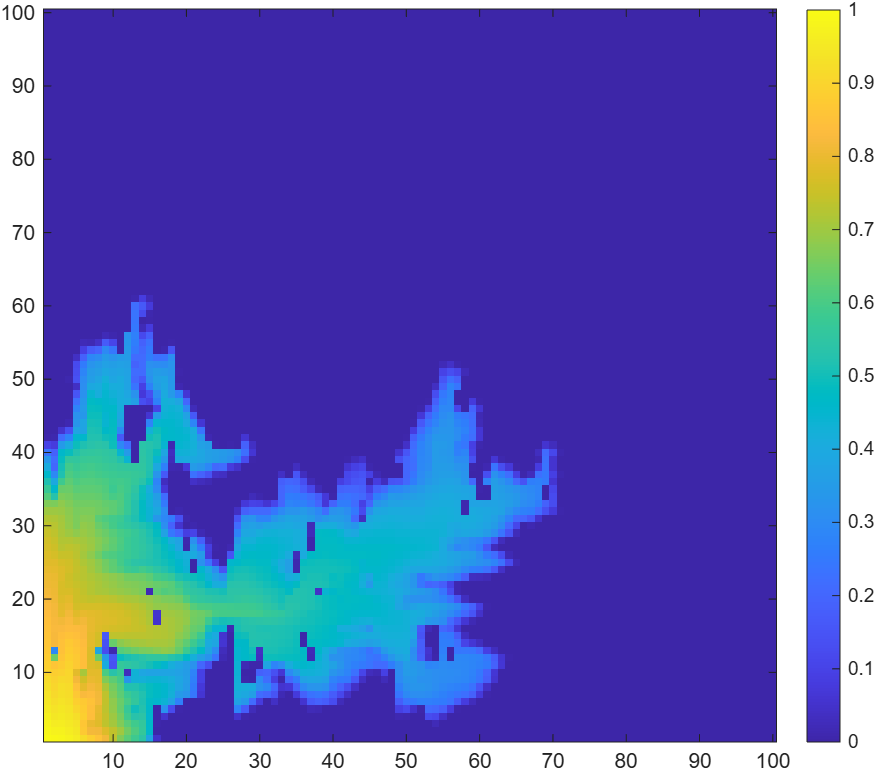}
        \caption{$\varepsilon=1$, $e_S=11.00\%$}
    \end{subfigure}
    \hfill
    \begin{subfigure}{0.32\textwidth}
        \centering
        \includegraphics[width=\linewidth]{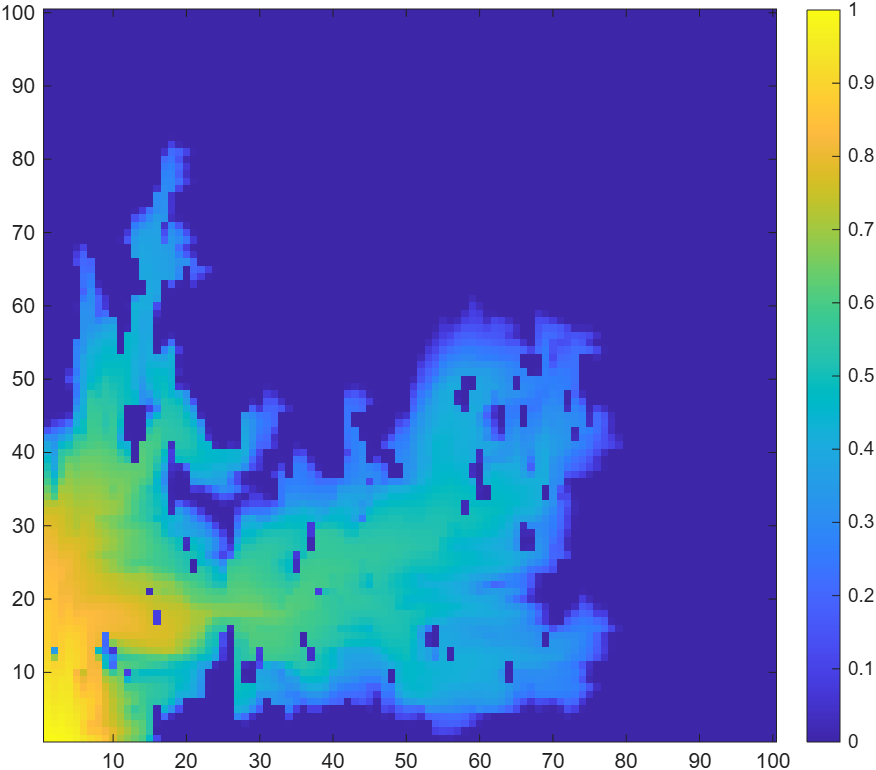}
        \caption{$\varepsilon=1$, $e_S=11.41\%$}
    \end{subfigure}

    \caption*{Figure~\thefignine. Multiscale solutions with $H=0.1$ and $3$ basis functions per coarse element for $\varepsilon=1$ at the corresponding $T=1,2,3$.}
    \label{fig:9_1}
\end{figure}

\refstepcounter{fignine}
\begin{figure}[H]
    \centering

    \begin{subfigure}{0.32\textwidth}
        \centering
        \includegraphics[width=\linewidth]{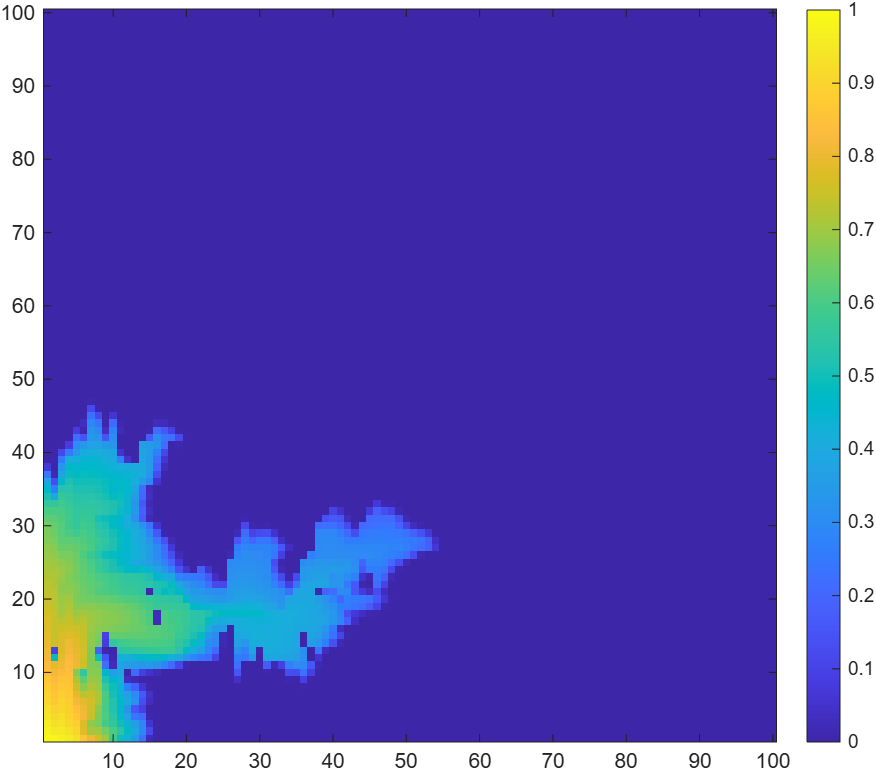}
        \caption{$\varepsilon=0.5$, $e_S=7.28\%$}
    \end{subfigure}
    \hfill
    \begin{subfigure}{0.32\textwidth}
        \centering
        \includegraphics[width=\linewidth]{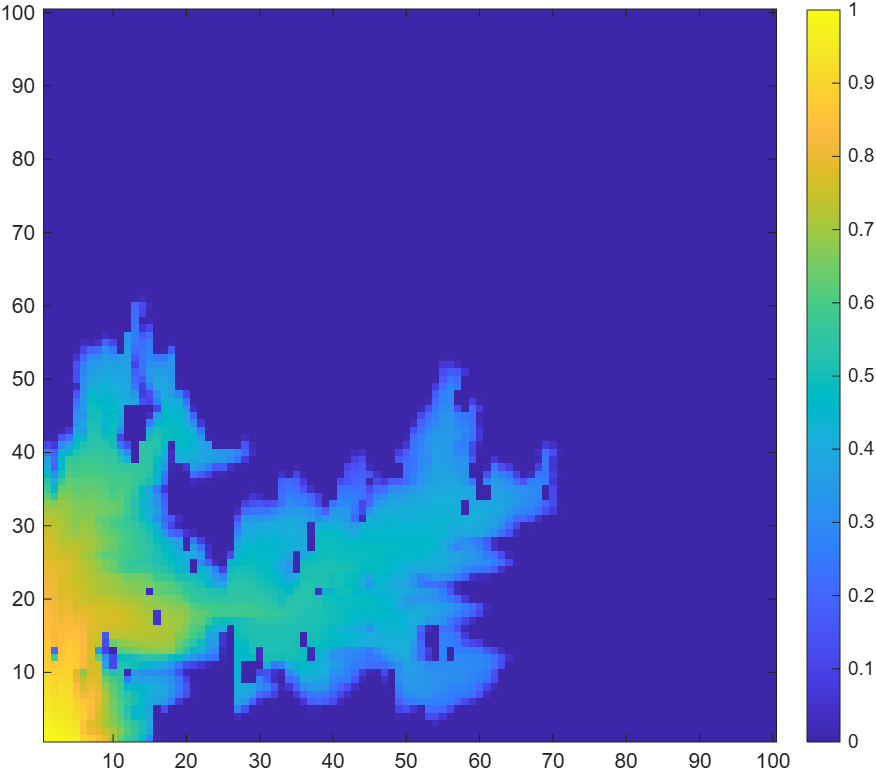}
        \caption{$\varepsilon=0.5$, $e_S=7.47\%$}
    \end{subfigure}
    \hfill
    \begin{subfigure}{0.32\textwidth}
        \centering
        \includegraphics[width=\linewidth]{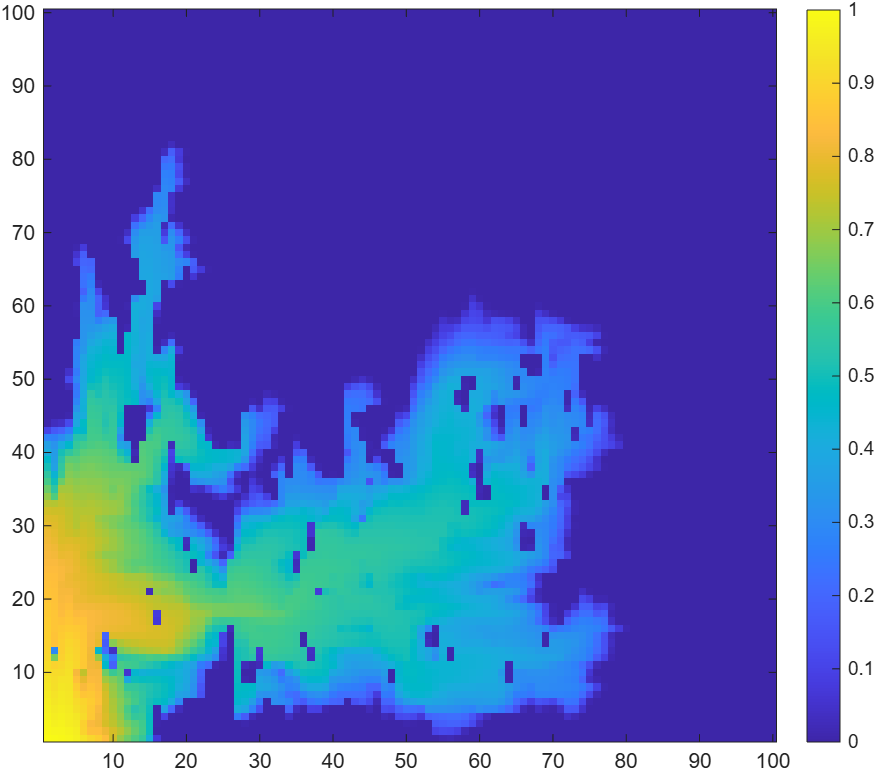}
        \caption{$\varepsilon=0.5$, $e_S=9.37\%$}
    \end{subfigure}

    \caption*{Figure~\thefignine. Multiscale solutions with $H=0.1$ and $3$ basis functions per coarse element for $\varepsilon=0.5$ at the corresponding $T=1,2,3$.}
    \label{fig:9_2}
\end{figure}

\refstepcounter{fignine}
\begin{figure}[H]
    \centering

    \begin{subfigure}{0.32\textwidth}
        \centering
        \includegraphics[width=\linewidth]{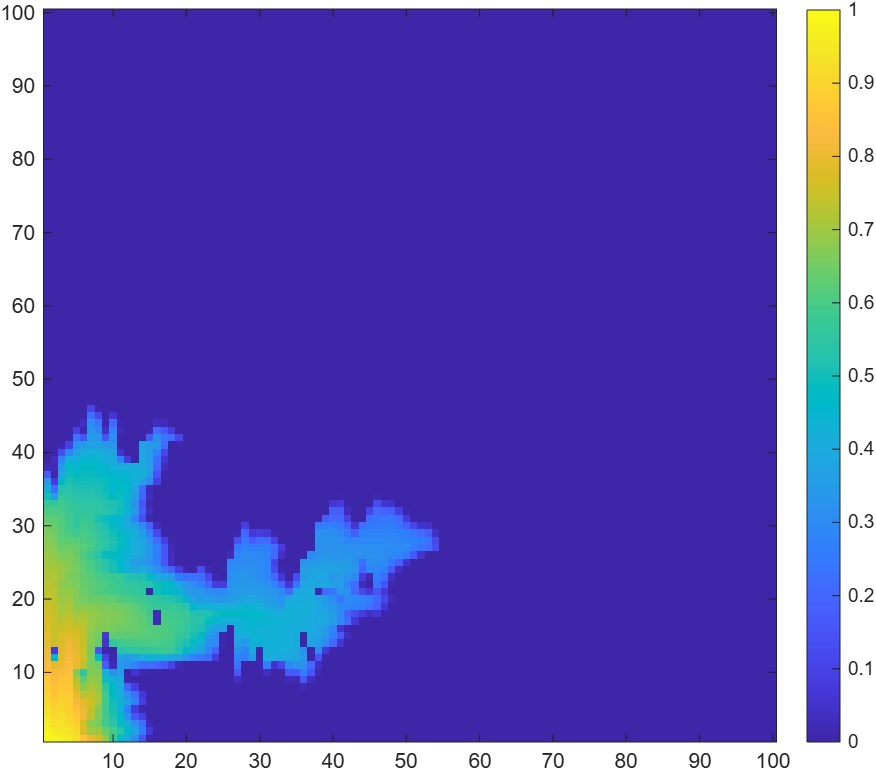}
        \caption{$\varepsilon=0.1$, $e_S=2.78\%$}
    \end{subfigure}
    \hfill
    \begin{subfigure}{0.32\textwidth}
        \centering
        \includegraphics[width=\linewidth]{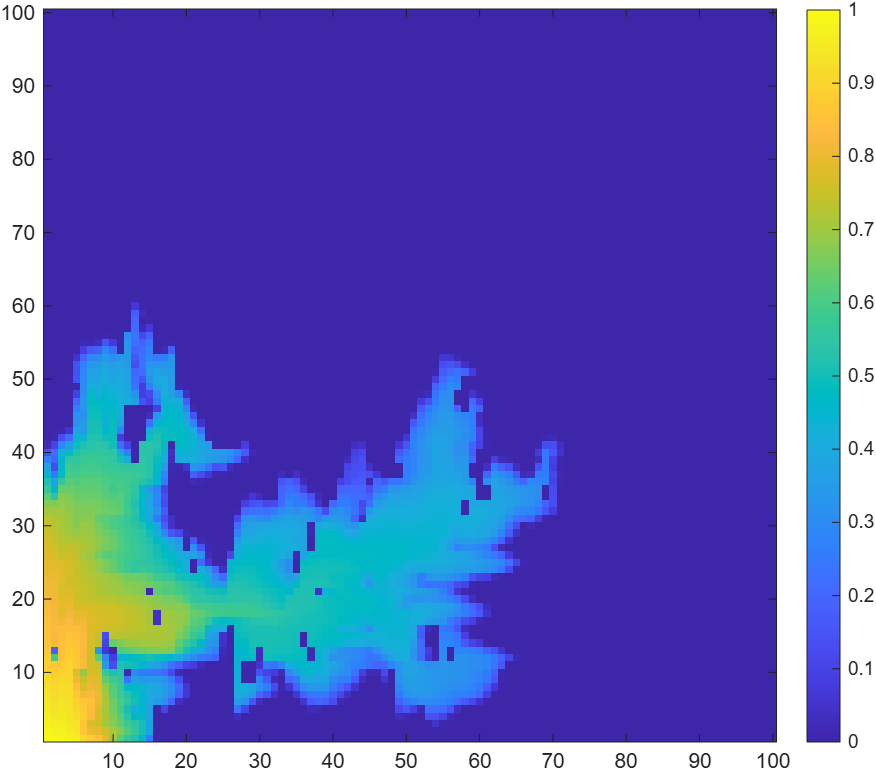}
        \caption{$\varepsilon=0.1$, $e_S=4.51\%$}
    \end{subfigure}
    \hfill
    \begin{subfigure}{0.32\textwidth}
        \centering
        \includegraphics[width=\linewidth]{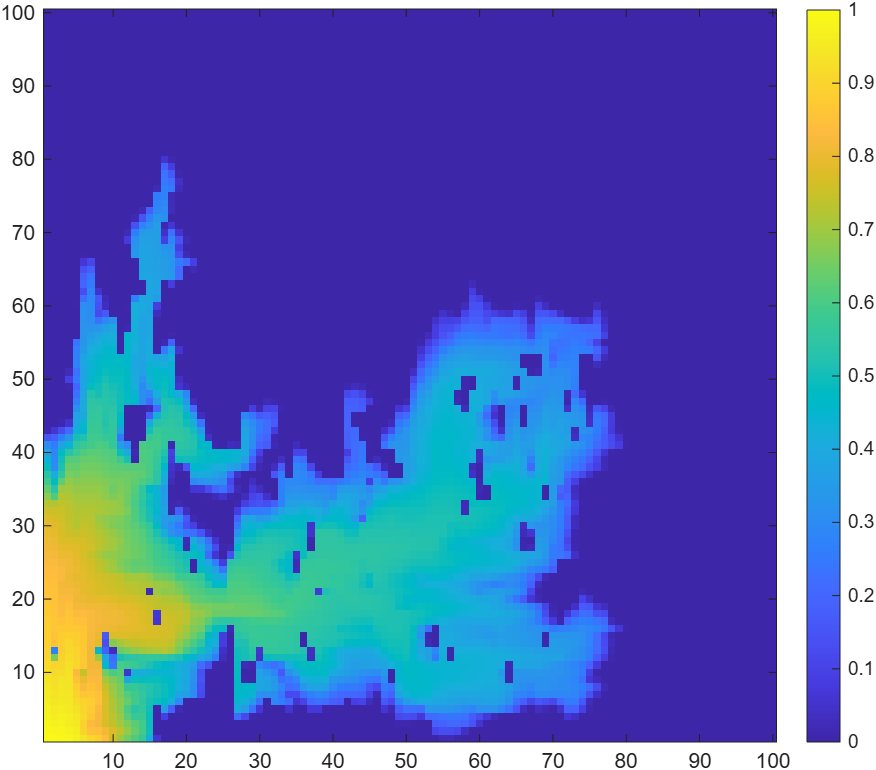}
        \caption{$\varepsilon=0.1$, $e_S=5.39\%$}
    \end{subfigure}

    \caption*{Figure~\thefignine. Multiscale solutions with $H=0.1$ and $3$ basis functions per coarse element for $\varepsilon=0.1$ at the corresponding $T=1,2,3$.}
    \label{fig:9_3}
\end{figure}

The numerical results show that the proposed adaptive multiscale method can provide accurate approximations for both permeability fields $\kappa_2$ and $\kappa_3$. In general, smaller coarse mesh size $H$ and more multiscale basis functions lead to smaller velocity and pressure errors. For the saturation equation, the numerical results show that the saturation error $e_S$ decreases as the adaptive tolerance $\varepsilon$ becomes smaller. This indicates that  updating of the multiscale space appropriately can improve the accuracy of the saturation approximation.

\section{Conclusion}
\label{sec 6}

In this paper, we developed an adaptive mixed CEM-GMsFEM for incompressible
and immiscible two-phase flow in heterogeneous porous media. The method combines with a physics-preserving IMPES scheme, and the multiscale spaces are
adaptively regenerated according to the variation of the effective permeability
induced by the saturation evolution. We proved the local mass conservation for
both phases, the unbiased property of the formulation, and the bounds-preserving
property under a suitable CFL condition. For the zero-capillary case, velocity
and saturation error estimates were derived, showing the influence of the
adaptive tolerance, coarse mesh size, spectral approximation error, and front
layer. Numerical results confirm the theoretical properties and show that the
capillary pressure enhances the diffusive nature of the advection-dominated
motion. Moreover, smaller adaptive tolerance generally leads to smaller
saturation errors, while suitable coarse mesh size and number of multiscale basis functions
provide accurate velocity and pressure approximations with reduced
computational cost.

\section*{Acknowledgement}

The research of Eric Chung is
partially supported by the Hong Kong
RGC General Research Fund
(Projects: 14305624 and 14304525).

\end{document}